\documentclass[a4paper,10pt,english,french]{smfart}
\usepackage[frenchb,english]{babel}
\usepackage[T1]{fontenc}
\usepackage{lmodern}
\usepackage{graphicx}
\usepackage{amsmath,amsthm,amssymb,amsfonts}
\usepackage[a4paper,left=4cm,right=4cm,top=4cm,bottom=4cm]{geometry}
\numberwithin{equation}{section}
\title[Manin's conjecture for two singular quartic del Pezzo surfaces]{Manin's conjecture for two quartic del Pezzo surfaces with
$3 \mathbf{A}_1$ and $\mathbf{A}_1 + \mathbf{A}_2$ singularity types}
\author{Pierre Le Boudec}
\subjclass{$11$D$45$, $14$G$05$}
\address{Université Denis Diderot (Paris VII) \\ Institut de Mathématiques de Jussieu \\ UMR 7586 \\ Case $7012$ - Bâtiment Chevaleret \\ Bureau $7$C$14$ \\ $75205$ Paris Cedex 13, France}
\email{pleboude@math.jussieu.fr}

\begin{document}

\makeatletter
\def\imod#1{\allowbreak\mkern10mu({\operator@font mod}\,\,#1)}
\makeatother

\newtheorem{lemma}{Lemma}
\newtheorem{theorem}{Theorem}

\newcommand{\vol}{\operatorname{vol}}
\newcommand{\D}{\mathrm{d}}
\newcommand{\rank}{\operatorname{rank}}
\newcommand{\Pic}{\operatorname{Pic}}
\newcommand{\Gal}{\operatorname{Gal}}
\newcommand{\meas}{\operatorname{meas}}
\newcommand{\Spec}{\operatorname{Spec}}

\begin{abstract}
We prove Manin's conjecture for two del Pezzo surfaces of degree four which are split over $\mathbb{Q}$ and whose singularity types are respectively $3 \mathbf{A}_1$ and $\mathbf{A}_1 + \mathbf{A}_2$. For this, we study a certain restricted divisor function and use a result about the equidistribution of its values in arithmetic progressions. In this task, Weil's bound for Kloosterman sums plays a key role.
\end{abstract}

\maketitle

\tableofcontents

\section{Introduction}

Let $V \subset \mathbb{P}^n$ be a singular del Pezzo surface defined over $\mathbb{Q}$ and anticanonically embedded and let
$U \subset V$ be the open subset formed by deleting the lines from $V$. Manin's conjecture \cite{MR974910} predicts the asymptotic behaviour of the number of rational points of bounded height on $U$, namely of the quantity
\begin{eqnarray}
\label{N_U}
N_{U,H}(B) & = & \# \{x \in U(\mathbb{Q}), H(x) \leq B \} \textrm{,}
\end{eqnarray}
where $H : \mathbb{P}^n(\mathbb{Q}) \to \mathbb{R}_{> 0}$ is the exponential height defined for
$(x_0, \dots, x_n) \in \mathbb{Z}^{n+1}$ such that
$\gcd(x_0, \dots, x_n) = 1$ by
\begin{eqnarray*}
H(x_0: \dots :x_n) & = & \max \{ |x_i|, 0 \leq i \leq n \}\textrm{.}
\end{eqnarray*}
More precisely, if $\widetilde{V}$ denotes the minimal desingularization of $V$, it is expected that
\begin{eqnarray}
\label{Manin's conjecture}
N_{U,H}(B) & = & c_{V,H} B \log(B)^{\rho -1} (1+o(1)) \textrm{,}
\end{eqnarray}
where $c_{V,H}$ is a constant which has been given a conjectural interpretation from Peyre \cite{MR1340296} and where
$\rho = \rho_{\widetilde{V}}$ is the rank of the Picard group of $\widetilde{V}$.

In this paper, we are interested in singular del Pezzo surfaces of degree four. These surfaces can be defined as the intersection of two quadrics in $\mathbb{P}^4$. Their classification is well-known and can be found in the work of Coray and Tsfasman \cite{MR940430}. Up to isomorphism over $\overline{\mathbb{Q}}$, there are fifteen types of such surfaces and they are categorized by their extended Dynkin diagrams, which describe the intersection behaviour of the negative curves on the minimal desingularizations (see for instance \cite[Table $4$]{D-hyper}).

From now on, we restrict our attention to surfaces which are split over $\mathbb{Q}$. Manin's conjecture is already known to hold for seven surfaces of different types. Batyrev and Tschinkel have proved it for toric varieties \cite{MR1620682} (which covers the three types $4 \mathbf{A}_1$, $2 \mathbf{A}_1 + \mathbf{A}_2$ and $2 \mathbf{A}_1 + \mathbf{A}_3$) and Chambert-Loir and Tschinkel have proved it for equivariant compactifications of vector groups \cite{MR1906155} (which covers the type $\mathbf{D}_5$). In these two proofs, the conjecture follows from the study of the height Zeta function
\begin{eqnarray*}
Z_{U,H}(s) & = & \sum_{x \in U(\mathbb{Q})} H(x)^{-s} \textrm{,}
\end{eqnarray*}
which is well-defined for $\Re(s) \gg 1$, using techniques coming from harmonic analysis. Let us note that for a certain surface of type $\mathbf{D}_5$, la Bretèche and Browning have given an independent proof \cite{MR2320172}. Furthermore, they have proved the following result, which is much stronger than \eqref{Manin's conjecture}. There exists a monic polynomial of degree $5 = \rho - 1$ such that for any fixed $\varepsilon >0$,
\begin{eqnarray}
\label{ImprovedManin}
N_{U,H}(B) & = & c_{V,H} B P(\log(B)) + O \left( B^{11/12+ \varepsilon} \right) \textrm{.}
\end{eqnarray}
Manin's conjecture has also been proved for three other surfaces, a surface of type $\mathbf{D}_4$ by Derenthal and Tschinkel \cite{MR2290499}, a surface of type $\mathbf{A}_1 + \mathbf{A}_3$ by Derenthal \cite{MR2520770} and a surface of type $\mathbf{A}_4$ by Browning and Derenthal \cite{MR2543667}. These proofs are intrinsically very different from those using harmonic analysis. They use a passage to universal torsors, which consists in defining a bijection between the set of points to be counted on $U$ and a certain set of integral points on an affine variety of higher dimension. This can be done using only elementary techniques (see section \ref{torsor section} for an example).

The aim of this paper is to give a proof of Manin's conjecture for two other surfaces which are split over $\mathbb{Q}$. The first, $V_1 \subset \mathbb{P}^4$, has singularity type $3 \mathbf{A}_1$ and is defined as the intersection of the two quadrics
\begin{eqnarray*}
x_0 x_1 - x_2^2 & = & 0 \textrm{,} \\
x_2^2 + x_1 x_2 + x_3 x_4 & = & 0 \textrm{.}
\end{eqnarray*}
We denote by $U_1$ the complement of the lines in $V_1$ and $N_{U_1,H}(B)$ is defined as in \eqref{N_U}. There are six lines on $V_1$ and they are given by $x_i = x_2 = x_j = 0$ and $x_0 + x_2 = x_1 + x_2 = x_j = 0$ for $i \in \{0,1\}$ and $j \in \{3,4\}$. The three singularities of $V_1$ are $(1:0:0:0:0)$, $(0:0:0:1:0)$ and $(0:0:0:0:1)$. We see that $V_1$ is actually split over $\mathbb{Q}$ and thus, if $\widetilde{V_1}$ denotes the minimal desingularization of $V_1$, the Picard group of $\widetilde{V_1}$ has rank
$\rho_1 = 6$. The universal torsor we use is an open subset of the hypersurface embedded in
$\mathbb{A}^9 \simeq \Spec \left( \mathbb{Q}[\eta_1, \dots, \eta_9] \right)$ and defined by
\begin{eqnarray}
\label{torsor 1}
\eta_4\eta_5 + \eta_1\eta_6\eta_7 + \eta_8\eta_9 & = & 0 \textrm{.}
\end{eqnarray}
Note that the two variables $\eta_2$ and $\eta_3$ do not appear in the equation.

The second surface $V_2 \subset \mathbb{P}^4$ has singularity type $\mathbf{A}_1 + \mathbf{A}_2$ and is defined as the intersection of the two quadrics
\begin{eqnarray*}
x_0 x_1 - x_2 x_3 & = & 0 \textrm{,} \\
x_1 x_2 + x_2 x_4 + x_3 x_4 & = & 0 \textrm{.}
\end{eqnarray*}
The open subset $U_2$, $N_{U_2,H}(B)$ and $\widetilde{V_2}$ are defined in a similar way. There are also six lines on $V_2$ and they are given by $x_i = x_2 = x_j = 0$ for $i \in \{0,1\}$ and $j \in \{3,4\}$, $x_1 = x_3 = x_4 = 0$ and
$x_0 = x_3 = x_1 + x_4 = 0$. The two singularities of $V_2$ are $(1:0:0:0:0)$ and $(0:0:0:0:1)$, of type $\mathbf{A}_2$ and $\mathbf{A}_1$ respectively. Just as before we have $\rho_2 = 6$. In this case, the universal torsor we use is an open subset of the hypersurface embedded in $\mathbb{A}^9 \simeq \Spec \left( \mathbb{Q}[\xi_1, \dots, \xi_9] \right)$ and defined by
\begin{eqnarray}
\label{torsor 2}
\xi_4\xi_5 + \xi_1^2\xi_6\xi_7 + \xi_8\xi_9 & = & 0 \textrm{.}
\end{eqnarray}
We immediately see that the equations \eqref{torsor 1} and \eqref{torsor 2} are very much alike and it is not hard to imagine that the proofs have strong similarities, that is why we have decided to couple them in this paper.

This work has been motivated by a result of Browning \cite[Theorem 3]{MR2362193}. Using the equation \eqref{torsor 1} of the universal torsor described above, he has proved the upper bound of the expected order of magnitude for $N_{U_1,H}(B)$, namely
\begin{eqnarray}
\label{bound N_{U_1,H}}
N_{U_1,H}(B) & \ll & B \log(B)^5 \textrm{.}
\end{eqnarray}

In most of the proofs of Manin's conjecture for del Pezzo surfaces using universal torsors, the first step consists in summing over two variables, viewing the torsor equation as a congruence and counting the number of integers lying in a prescribed region and satisfying this congruence. The novelty here is that we start by summing over three variables instead. In our two cases, this is linked to studying the distribution of the values of a certain restricted divisor function in arithmetic progressions. In this task, Weil's bound for Kloosterman sums plays a crucial role. Our result is the following.

\begin{theorem}
\label{Main Theorem}
For $i = 1, 2$, as $B$ tends to $+ \infty$, we have the estimate
\begin{eqnarray*}
N_{U_i,H}(B) & = & c_{V_i,H} B \log(B)^{5} \left( 1 + O \left( \frac{\log(\log(B))}{\log(B)} \right) \right) \textrm{,}
\end{eqnarray*}
where $c_{V_1,H}$ and $c_{V_2,H}$ agree with Peyre's prediction.
\end{theorem}

Since $\rho_1 = \rho_2 = 6$, these estimates prove that Manin's conjecture holds for $V_1$ and $V_2$. Let us note that Derenthal has shown that $V_1$ and $V_2$ are not toric \cite[Proposition $12$]{D-hyper} and Derenthal and Loughran have proved that they are not equivariant compactifications of $\mathbb{G}_a^2$ \cite{DL-equi}, so this work is not covered by the existing general results. In view of theorem \ref{Main Theorem}, it remains to deal with six types of split singular quartic del Pezzo surfaces among the list of fifteen.

In both cases, we have noted that the universal torsor is an open subset of a hypersurface embedded in $\mathbb{A}^9$. In
\cite{D-hyper}, Derenthal has determined the del Pezzo surfaces whose universal torsors are hypersurfaces and it turns out that in the case of split quartic surfaces, Manin's conjecture has only been proved for surfaces whose universal torsors are either open affine subsets (which is equivalent to being toric), or open subsets of hypersurfaces. It would be interesting to prove Manin's conjecture for a surface which is in neither of these two classes.

The author is extremely grateful to his supervisor Professor de la Bretèche for introducing him to the problem of counting rational points on algebraic varieties and also for his precious advice all along this work. It is also a pleasure to thank Professor Browning for his careful reading of the manuscript and for his useful comments.

This work has received the financial support of the ANR PEPR (\textit{Points Entiers Points Rationnels}).

\section{Preliminary results}

\subsection{Equidistribution of the values of a restricted divisor function in \text{arithmetic} progressions}

Let $\tau$ denote the divisor function. We start by recalling a classical fact about the sums of the values of $\tau$ in arithmetic progressions. For $a, q \in \mathbb{Z}_{\geq 1}$ two coprime integers and $X \geq 1$, define
\begin{eqnarray*}
D(X;q,a) & = & \sum_{\substack{n \leq X \\ n \equiv a \imod{q}}} \tau(n) \textrm{.}
\end{eqnarray*}
Then (see \cite[Corollary $1$]{MR532980} for instance) there exists an explicit quantity $D^{\ast}(X;q)$ independent of $a$ such that for $q \leq X^{2/3}$,
\begin{eqnarray*}
D(X;q,a) - D^{\ast}(X;q) & \ll & X^{1/3 + \varepsilon} \textrm{.}
\end{eqnarray*}
We need a more general result since we have to consider a sum similar to $D(X;q,a)$ but with $\tau$ replaced by a function which only counts certain divisors. However, we will not determine a specific value of our main term and we will content ourselves with the value provided by averaging the estimate over $a$ coprime to $q$.

The results stated in this section use several classical ideas which have for example been developed in Heath-Brown's investigation of the divisor function $\tau_3 := \tau \ast 1$ \cite{MR866901}. Let $\mathcal{I}$ and $\mathcal{J}$ be two ranges. We define the quantities
\begin{eqnarray*}
N(\mathcal{I}, \mathcal{J};q,a) & = & \# \left\{ (u,v) \in \mathcal{I} \times \mathcal{J} \cap \mathbb{Z}^2, uv \equiv a \imod{q} \right\} \textrm{,}
\end{eqnarray*}
and
\begin{eqnarray*}
N^{\ast}(\mathcal{I}, \mathcal{J}; q) & = &
\frac1{\varphi(q)} \# \left\{(u,v) \in \mathcal{I} \times \mathcal{J} \cap \mathbb{Z}^2, \gcd(uv,q) = 1 \right\} \textrm{.}
\end{eqnarray*}

\begin{lemma}
\label{Weil lemma}
Let $\varepsilon > 0$ be fixed. We have the estimate
\begin{eqnarray*}
N(\mathcal{I}, \mathcal{J};q,a) - N^{\ast}(\mathcal{I}, \mathcal{J};q) & \ll & q^{1/2 + \varepsilon} \textrm{.}
\end{eqnarray*}
\end{lemma}

\begin{proof}
Let $e_q$ be the function defined by $e_q(x) = e^{2 i \pi x /q}$. We detect the congruence using sums of exponentials, we get
\begin{eqnarray*}
N(\mathcal{I}, \mathcal{J};q,a) & = & \sum_{\substack{\alpha,\beta = 1 \\ \alpha \beta \equiv a \imod{q}}}^q
\# \{ (u,v) \in \mathcal{I} \times \mathcal{J} \cap \mathbb{Z}^2, q| \alpha - u, \beta - v \} \\
& = & \sum_{\substack{\alpha,\beta = 1 \\ \alpha \beta \equiv a \imod{q}}}^q \frac1{q^2}
\left( \sum_{u \in \mathcal{I}} \sum_{r=1}^q e_q(r \alpha - r u) \right)
\left( \sum_{v \in \mathcal{J}} \sum_{s=1}^q e_q(s \beta - s v) \right) \\
& = & \frac1{q^2} \sum_{r,s = 1}^q K(r,as,q) F_q(r,s) 
 \textrm{,}
\end{eqnarray*}
where $K(r,as,q)$ is the Kloosterman sum defined by
\begin{eqnarray*}
K(r,as,q) & = & \sum_{\substack{\alpha = 1 \\ \gcd(\alpha,q) = 1}}^q e_q \left( r \alpha + a s \alpha^{-1} \right) \textrm{,}
\end{eqnarray*}
where $\alpha^{-1}$ denotes the inverse of $\alpha$ modulo $q$ and where
\begin{eqnarray*}
F_q(r,s) & = & \left( \sum_{u \in \mathcal{I}} e_q(- r u) \right) \left( \sum_{v \in \mathcal{J}} e_q(- s v) \right)
\textrm{.}
\end{eqnarray*}
Let $||x||$ denote the distance from $x$ to the set of integers. If $r,s \neq q$, $F_q(r,s)$ is a product of two geometric sums and we therefore have
\begin{eqnarray*}
F_q(r,s) & \ll & \left| \left| \frac{r}{q} \right| \right| ^{-1} \left| \left| \frac{s}{q} \right| \right| ^{-1} \textrm{.}
\end{eqnarray*}
Let $N(\mathcal{I}, \mathcal{J};q)$ be the sum of the terms corresponding to $r=q$ or $s=q$. Since $\gcd(a,q) = 1$,  we see that $K(q,as,q)$ and $K(r,aq,q)$ are independent of $a$ and thus $N(\mathcal{I}, \mathcal{J};q)$ is also independent of $a$. We are therefore led to give a bound for $N(\mathcal{I}, \mathcal{J};q,a) - N(\mathcal{I}, \mathcal{J};q)$. Weil's bound for Kloosterman sums (see \cite{MR0126420}) yields
\begin{eqnarray*}
N(\mathcal{I}, \mathcal{J};q,a) - N(\mathcal{I}, \mathcal{J};q) & = & \frac1{q^2} \sum_{r,s = 1}^{q-1} K(r,as,q) F_q(r,s) \\
& \ll & \frac1{q^2} \tau(q) q^{1/2} \sum_{r,s = 1}^{q-1} \gcd(r,s,q)^{1/2} \left| \left| \frac{r}{q} \right| \right| ^{-1}
\left| \left| \frac{s}{q} \right| \right| ^{-1} \\
& \ll & \tau(q) q^{1/2} \sum_{0 < |r|,|s| \leq q/2} \gcd(r,s,q)^{1/2} |r|^{-1} |s|^{-1} \textrm{.}
\end{eqnarray*}
Let us bound the sum of the right-hand side, we get
\begin{eqnarray*}
\sum_{0 < |r|,|s| \leq q/2} \gcd(r,s,q)^{1/2} |r|^{-1} |s|^{-1} & \ll &
\sum_{d|q} d^{1/2} \sum_{\substack{r = 1 \\ d|r }}^q r^{-1} \sum_{\substack{s = 1 \\ d|r }}^q s^{-1} \\
& \ll & \log(q)^2  \textrm{.}
\end{eqnarray*}
Since $N(\mathcal{I}, \mathcal{J};q)$ does not depend on $a$, averaging over $a$ coprime to $q$ shows that we can replace $N(\mathcal{I}, \mathcal{J};q)$ by $N^{\ast}(\mathcal{I}, \mathcal{J};q)$, which completes the proof.
\end{proof}

An immediate consequence of this lemma is the bound
\begin{eqnarray}
\label{average bound}
N(\mathcal{I}, \mathcal{J};q,a) & \ll & \frac1{\varphi(q)} \# \left( \mathcal{I} \times \mathcal{J} \cap \mathbb{Z}^2 \right)
+ q^{1/2 + \varepsilon} \textrm{.}
\end{eqnarray}
Lemmas \ref{lemma tau} and \ref{lemma tau2} below are respectively devoted to the treatment of the varieties $V_1$ and $V_2$.

Let $X, X_1, X_2, T, Z, L_1, L_2 > 0$. Let also $\mathcal{S} = \mathcal{S}(X, X_1, X_2, T, Z, L_1, L_2)$ be the set of
$(x,y) \in \mathbb{R}^2$ such that
\begin{eqnarray}
\label{A}
|xy| & \leq & X \textrm{,} \\
\label{B}
|x| |xy + T| & \leq & X_1 \textrm{,} \\
\label{C}
|y| & \leq & X_2 \textrm{,} \\
\label{D}
Z & \leq & |xy + T|  \textrm{,} \\
\label{E}
L_1 & \leq & |x| \textrm{,} \\
\label{E'}
L_2 & \leq & |y| \textrm{.}
\end{eqnarray}
Finally, we introduce
\begin{eqnarray}
\label{definition D}
D(\mathcal{S};q,a) & = & \# \left\{ (u,v) \in \mathcal{S} \cap \mathbb{Z}^2, uv \equiv a \imod{q} \right\} \textrm{,}
\end{eqnarray}
and
\begin{eqnarray}
\label{definition E}
D^{\ast}(\mathcal{S};q) & = & \frac1{\varphi(q)} \# \left\{ (u,v) \in \mathcal{S} \cap \mathbb{Z}^2, \gcd(uv,q) = 1 \right\} \textrm{.}
\end{eqnarray}

\begin{lemma}
\label{lemma tau}
Let $\varepsilon > 0$ be fixed. If $T \leq X$ then for $q \leq X^{2/3}$, we have the estimate
\begin{eqnarray*}
D(\mathcal{S};q,a) - D^{\ast}(\mathcal{S};q) & \ll & \frac{X^{2/3 + \varepsilon}}{q^{1/2}} + \frac{X}{\varphi(q)} \left( \frac1{L_1} + \frac1{L_2} \right) \textrm{.}
\end{eqnarray*}
\end{lemma}

Note that the conditions $T \leq X$, $|xy| \leq X$ and $|xy + T| \geq Z$ imply $Z \leq 2 X$. 

\begin{proof}
The result is true if $\mathcal{S} \cap \mathbb{Z}_{\neq 0}^2 = \emptyset$ so we assume from now on that
$\mathcal{S} \cap \mathbb{Z}_{\neq 0}^2 \neq \emptyset$.  Let $0 < \delta \leq 1$ be a parameter to be selected in due course and $\zeta = 1 + \delta$. Let also $U$ and $V$ be variables running over the set $\{ \pm \zeta^n, n \in \mathbb{Z}_{\geq -1} \}$ and let $\mathcal{I} = ]U,\zeta U]$ if $U > 0$ ($\mathcal{I} = [\zeta U,U[$ if $U < 0$) and $\mathcal{J} = ]V,\zeta V]$ if $V > 0$ ($\mathcal{J} = [\zeta V, V[$ if $V < 0$). We have
\begin{eqnarray*}
D(\mathcal{S};q,a) -  \sum_{\mathcal{I} \times \mathcal{J} \cap \mathbb{Z}^2 \subset \mathcal{S}} N(\mathcal{I},\mathcal{J};q,a) & \ll & \sum_{\substack{\mathcal{I} \times \mathcal{J} \cap \mathbb{Z}^2 \nsubseteq \mathcal{S} \\
\mathcal{I} \times \mathcal{J} \cap \mathbb{Z}^2 \nsubseteq  \mathbb{R}^2 \setminus \mathcal{S}}} N(\mathcal{I},\mathcal{J};q,a) \textrm{.}
\end{eqnarray*}
We define the quantity
\begin{eqnarray*}
D(\mathcal{S};q) & = & \sum_{\mathcal{I} \times \mathcal{J} \cap \mathbb{Z}^2 \subset \mathcal{S}} N^{\ast}(\mathcal{I},\mathcal{J};q) \textrm{.}
\end{eqnarray*}
Note that since $N^{\ast}(\mathcal{I},\mathcal{J};q)$ is independent of $a$, so is $D(\mathcal{S};q)$. Furthermore, we have
\begin{eqnarray*}
\sum_{\mathcal{I} \times \mathcal{J} \cap \mathbb{Z}^2 \subset \mathcal{S}} N(\mathcal{I},\mathcal{J};q,a) - D(\mathcal{S};q) & \ll & \frac{X^{\varepsilon} q^{1/2 + \varepsilon}}{\delta^2} \textrm{,}
\end{eqnarray*}
using lemma \ref{Weil lemma} and noticing that the number of rectangles $\mathcal{I} \times \mathcal{J}$ subject to the condition
$\mathcal{I} \times \mathcal{J} \cap \mathbb{Z}^2 \subset \mathcal{S}$ is less than
$\left( 2 \log(X) / \log(\zeta) \right)^2 \ll X^{\varepsilon} \delta^{-2}$ since $\delta \leq 1$. Since
$q^{\varepsilon} \leq X^{\varepsilon}$, we have obtained
\begin{eqnarray*}
D(\mathcal{S};q,a) - D(\mathcal{S};q) & \ll &
\sum_{\substack{\mathcal{I} \times \mathcal{J} \cap \mathbb{Z}^2 \nsubseteq \mathcal{S} \\
\mathcal{I} \times \mathcal{J} \cap \mathbb{Z}^2 \nsubseteq \mathbb{R}^2 \setminus \mathcal{S}}} N(\mathcal{I},\mathcal{J};q,a) + \frac{X^{\varepsilon} q^{1/2}}{\delta^2} \textrm{.}
\end{eqnarray*}
Using the bound \eqref{average bound}, we finally deduce
\begin{eqnarray*}
D(\mathcal{S};q,a) - D(\mathcal{S};q) & \ll & \frac1{\varphi(q)}
\sum_{\substack{\mathcal{I} \times \mathcal{J} \cap \mathbb{Z}^2 \nsubseteq \mathcal{S} \\
\mathcal{I} \times \mathcal{J} \cap \mathbb{Z}^2 \nsubseteq \mathbb{R}^2 \setminus \mathcal{S}}}
\# \left( \mathcal{I} \times \mathcal{J}  \cap \mathbb{Z}^2 \right) + \frac{X^{\varepsilon} q^{1/2}}{\delta^2} \textrm{,}
\end{eqnarray*}
since the number of rectangles $\mathcal{I} \times \mathcal{J}$ such that
$\mathcal{I} \times \mathcal{J} \cap \mathbb{Z}^2 \nsubseteq \mathcal{S}$ and
$\mathcal{I} \times \mathcal{J} \cap \mathbb{Z}^2 \nsubseteq  \mathbb{R}^2 \setminus \mathcal{S}$ is also
$\ll X^{\varepsilon} \delta^{-2}$. The sum of the right-hand side is over all the rectangles $\mathcal{I} \times \mathcal{J}$ for which we have $(\zeta^{s_1} U, \zeta^{s_2} V) \in \mathbb{Z}^2 \cap \mathcal{S}$ and
$(\zeta^{t_1} U, \zeta^{t_2} V) \in \mathbb{Z}^2 \setminus \mathcal{S}$ for some $(s_1,s_2) \in ]0,1]^2$ and $(t_1,t_2) \in ]0,1]^2$. This implies that one of the inequalities defining $\mathcal{S}$ is not satisfied by $(\zeta^{t_1} U, \zeta^{t_2} V)$ and we need to estimate the contribution coming from each condition among \eqref{A}, \eqref{B}, \eqref{C}, \eqref{D}, \eqref{E} and \eqref{E'}. Note that the conditions \eqref{A}, \eqref{E} and \eqref{E'} together imply
\begin{eqnarray}
\label{condition A}
|U| & \ll & \frac{X}{L_2} \textrm{,} \\
\label{condition B}
|V| & \ll & \frac{X}{L_1} \textrm{.}
\end{eqnarray}
In the following, we could sometimes write strict inequalities instead of non-strict ones but this would not change anything in our reasoning. Let us start by treating the case of the condition \eqref{A}. For the rectangles $\mathcal{I} \times \mathcal{J}$ described above, we have $\zeta^{s_1 + s_2} |U V| \leq X$ and $\zeta^{t_1 + t_2} |U V| > X$ for some $(s_1,s_2) \in ]0,1]^2$ and $(t_1,t_2) \in ]0,1]^2$. These two inequalities imply
\begin{eqnarray}
\label{condition C}
& & \zeta^{-2} X < |U V| \leq X \textrm{.}
\end{eqnarray}
Going back to the variables $u$ and $v$, we get $\zeta^{-2} X < |u v| \leq \zeta^2 X$. Therefore, the error we aim to estimate is bounded by
\begin{eqnarray*}
\sum_{\eqref{condition B}, \eqref{condition C}} \# \left( \mathcal{I} \times \mathcal{J} \cap \mathbb{Z}^2 \right)
& \ll & \# \left\{ (u,v) \in \mathbb{Z}_{\neq 0}^2, 
\begin{array}{l}
\zeta^{-2} X < |u v| \leq \zeta^2 X \\
|v| \ll X / L_1 
\end{array}
\right\} \\
& \ll & \sum_{|v| \ll X / L_1} \left( \frac{\delta X}{|v|} + 1 \right) \\
& \ll & \delta X^{1 + \varepsilon} + \frac{X}{L_1} \textrm{.}
\end{eqnarray*}
We now deal with the other conditions in a similar fashion. Let us treat the case of the condition \eqref{B}. In this case, for some $(s_1,s_2) \in ]0,1]^2$ and $(t_1,t_2) \in ]0,1]^2$, we have
\begin{eqnarray}
\label{condition C'}
\zeta^{s_1} |U| \left| \zeta^{s_1 + s_2} UV + T \right| & \leq & X_1 \textrm{,} \\
\label{condition C''}
\zeta^{t_1} |U| \left| \zeta^{t_1 + t_2} UV + T \right| & > & X_1 \textrm{.}
\end{eqnarray}
Note that using $|UV| \leq X$ and $T \leq X$, the inequality \eqref{condition C''} gives
\begin{eqnarray}
\label{condition E}
|U| & \gg & \frac{X_1}{X} \textrm{.}
\end{eqnarray}
The inequalities \eqref{condition C'} and \eqref{condition C''} imply
\begin{eqnarray}
\label{condition D}
& & \zeta^{-3} \frac{X_1}{|U|} - \left( 1 - \zeta^{-2} \right) T < \left| UV + T \right| \leq
\frac{X_1}{|U|} + \left( 1 - \zeta^{-2} \right) T \textrm{.}
\end{eqnarray} 
Going back to the variables $u$ and $v$, we easily get
\begin{eqnarray*}
\big| \left| u v + T \right| - \left|U V + T \right| \big| & \leq &
\left| uv - UV \right| \\
& \leq & 3 \delta |UV| \\
& \leq & 3 \delta \left( \frac{X_1}{|U|} + T \right) \textrm{,}
\end{eqnarray*}
using the condition \eqref{condition C'}. Since $1 - \zeta^{-2} \leq 2 \delta$, the inequality \eqref{condition D} gives
\begin{eqnarray*}
& & \left( \zeta^{-3} - 3 \delta \right) \frac{X_1}{|U|} - 5 \delta T < \left| uv + T \right| \leq
(1 + 3 \delta) \frac{X_1}{|U|} + 5 \delta T \textrm{,}
\end{eqnarray*}
and therefore
\begin{eqnarray}
\label{condition F}
& & \left( \zeta^{-3} - 3 \delta \right) \frac{X_1}{|u|} - 5 \delta T < \left| uv + T \right| \leq
\zeta (1 + 3 \delta) \frac{X_1}{|u|} + 5 \delta T \textrm{.}
\end{eqnarray}
Note that we have not tried to sharpen this inequality because this is useless for our purpose. Thus in this case, the error is bounded by
\begin{eqnarray*}
\sum_{\substack{\eqref{condition A} \\ \eqref{condition E}, \eqref{condition D}}}
\# \left( \mathcal{I} \times \mathcal{J} \cap \mathbb{Z}^2 \right) & \ll &
\# \left\{ (u,v) \in \mathbb{Z}_{\neq 0}^2, 
\begin{array}{l}
\eqref{condition F} \\
|u| \gg X_1 / X \\
|u| \ll X / L_2 
\end{array}
\right\} \\
& \ll & \sum_{\substack{|u| \gg X_1 / X \\ |u| \ll X / L_2}}
\left( \frac{\delta X_1}{u^2} + \frac{\delta T}{|u|} + 1 \right) \\
& \ll & \delta X^{1 + \varepsilon} + \frac{X}{L_2} \textrm{,}
\end{eqnarray*}
since $T \leq X$. In the case of the condition \eqref{C}, the condition which plays the role of the conditions \eqref{condition C} and \eqref{condition D} in the previous two cases is
\begin{eqnarray}
\label{condition G}
& & \zeta^{-1} X_2 < |V| \leq X_2 \textrm{.}
\end{eqnarray}
Combined with $|UV| \leq X$, this condition gives
\begin{eqnarray}
\label{condition H}
|U| & \ll & \frac{X}{X_2} \textrm{.}
\end{eqnarray}
Moreover, in terms of the variable $v$, we have $\zeta^{-1} X_2 < |v| \leq \zeta X_2$. Therefore, this contribution is bounded by
\begin{eqnarray*}
\sum_{\substack{\eqref{condition A} \\ \eqref{condition G}, \eqref{condition H}}}
\# \left( \mathcal{I} \times \mathcal{J} \cap \mathbb{Z}^2 \right) & \ll &
\# \left\{ (u,v) \in \mathbb{Z}_{\neq 0}^2, 
\begin{array}{l}
\zeta^{-1} X_2 < |v| \leq \zeta X_2 \\
|u| \ll X / X_2 \\
|u| \ll X / L_2 
\end{array}
\right\} \\
& \ll & \sum_{\substack{|u| \ll X / X_2 \\ |u| \ll X / L_2}} \left( \delta X_2 + 1 \right) \\
& \ll & \delta X + \frac{X}{L_2} \textrm{.}
\end{eqnarray*}
Let us now deal with the condition \eqref{D}. Here, reasoning as we did to deduce the inequality \eqref{condition D} from the conditions \eqref{condition C'} and \eqref{condition C''}, we get the inequality
\begin{eqnarray}
\label{condition I}
& & \zeta^{-2} Z - \left( 1 - \zeta^{-2} \right) T \leq \left| UV + T \right| <
Z + \left( 1 - \zeta^{-2} \right) T \textrm{,}
\end{eqnarray}
and following the reasoning we made to derive the condition \eqref{condition F} from the condition \eqref{condition D}, we obtain
\begin{eqnarray}
\label{condition I'}
& & \left( \zeta^{-2} - 3 \delta \right) Z - 5 \delta T \leq \left| uv + T \right| <
(1 + 3 \delta) Z + 5 \delta T \textrm{.}
\end{eqnarray}
We therefore see that this contribution is bounded by
\begin{eqnarray*}
\sum_{\eqref{condition A}, \eqref{condition I}} \# \left( \mathcal{I} \times \mathcal{J} \cap \mathbb{Z}^2 \right) & \ll &
\# \left\{ (u,v) \in \mathbb{Z}_{\neq 0}^2, 
\begin{array}{l}
\eqref{condition I'} \\
|u| \ll X / L_2 
\end{array}
\right\} \\
& \ll & \sum_{|u| \ll X / L_2} \left( \frac{\delta Z}{|u|} + \frac{\delta T}{|u|} + 1 \right) \\
& \ll & \delta X^{1 + \varepsilon} + \frac{X}{L_2} \textrm{,}
\end{eqnarray*}
since $T \leq X$ and $Z \leq 2 X$. Mimicking what we have done for the condition \eqref{C}, we find that the contributions corresponding to the conditions \eqref{E} and \eqref{E'} are respectively $\ll \delta X + X / L_1$ and $\ll \delta X + X / L_2$. Writing $1 / \varphi(q) \ll X^{\varepsilon} /q$ and rescaling $\varepsilon$, we have finally proved that
\begin{eqnarray*}
D(\mathcal{S};q,a) - D(\mathcal{S};q) & \ll & X^{\varepsilon} \left( \frac{\delta X}{q} + \frac{q^{1/2}}{\delta^2} \right)
+ \frac{X}{\varphi(q)} \left( \frac1{L_1} + \frac1{L_2} \right) \textrm{.}
\end{eqnarray*}
Averaging this estimate over $a$ coprime to $q$ and using the fact that $D(\mathcal{S};q)$ does not depend on $a$, we see that we can replace $D(\mathcal{S};q)$ by $D^{\ast}(\mathcal{S};q)$. Furthermore, the choice $\delta = q^{1/2} X^{-1/3}$ is allowed provided that $q \leq X^{2/3}$, which completes the proof.
\end{proof}

We emphasize that the average effect which yields the term $1/\varphi(q)$ in $D^{\ast}(\mathcal{S};q)$ is really the key step of the proof. Note that the estimate of lemma \ref{lemma tau} is actually true for $q \leq X$ but the error term is no longer better than the trivial error term $X^{1 + \varepsilon} / q$ when $q \geq X^{2/3}$.

For given $X_3 > 0$, let $\mathcal{S}_1 = \mathcal{S}_1(X, X_1, X_2, X_3, T, Z, L_1, L_2)$ be the set of $(x,y) \in \mathbb{R}^2$ satisfying the conditions \eqref{A}, \eqref{B}, \eqref{C}, \eqref{D}, \eqref{E}, \eqref{E'} and
\begin{eqnarray}
\label{F}
|xy + T| & \leq & X \textrm{,} \\
\label{G}
|x| y^2 & \leq & X_3 \textrm{.}
\end{eqnarray}
Let also $\mathcal{S}_2 = \mathcal{S}_2(X, X_1, X_2, X_3, T, Z, L_1, L_2)$ be the set of $(x,y) \in \mathbb{R}^2$ satisfying the conditions \eqref{A}, \eqref{D}, \eqref{E}, \eqref{E'}, \eqref{F} and
\begin{eqnarray}
\label{H}
|x| & \leq & X_1 \textrm{,} \\
\label{I}
|y| |xy + T| & \leq & X_2 \textrm{,} \\
\label{J}
|x| y^2 |xy + T| & \leq & X_3 \textrm{.}
\end{eqnarray}
Finally, $D(\mathcal{S}_1;q,a)$ and $D(\mathcal{S}_2;q,a)$ are defined exactly as in \eqref{definition D} and $D^{\ast}(\mathcal{S}_1;q)$ and $D^{\ast}(\mathcal{S}_2;q)$ as in \eqref{definition E}.

\begin{lemma}
\label{lemma tau2}
Let $\varepsilon > 0$ be fixed. If $T \leq 2 X$ then for $q \leq X^{2/3}$, we have the estimates
\begin{eqnarray*}
D(\mathcal{S}_1;q,a) - D^{\ast}(\mathcal{S}_1;q) & \ll & \frac{X^{2/3 + \varepsilon}}{q^{1/2}}
+ \frac{X}{\varphi(q)} \left( \frac1{L_1} + \frac1{L_2} \right) \textrm{,}
\end{eqnarray*}
and
\begin{eqnarray*}
D(\mathcal{S}_2;q,a) - D^{\ast}(\mathcal{S}_2;q) & \ll & \frac{X^{4/5 + \varepsilon}}{q^{7/10}}
+ \frac{X}{\varphi(q)} \left( \frac1{L_1} + \frac1{L_2} \right) \textrm{.}
\end{eqnarray*}
\end{lemma}

To prove lemma \ref{lemma tau2}, we can proceed almost exactly as in the proof of lemma \ref{lemma tau} apart from the fact that the condition \eqref{J} is more complicated than the others. Indeed, it is the only condition where both $x$ and $y$ appear with powers greater or equal to $2$. To solve this problem, we need the following result.

\begin{lemma}
\label{square}
Let $0 < \delta \leq 1$, $Y  \in \mathbb{R}_{>0}$ and $A, Y' \in \mathbb{R}$ be such that $0 < Y - Y' \ll \delta M^2$ where
$M = \max \left( |A|, Y^{1/2} \right)$. Let $\mathcal{R} \subset \mathbb{R}$ be the set of real numbers $y$ subject to
\begin{eqnarray}
\label{double inequality}
& & Y' < \left| y^2 + 2 A y \right| \leq Y \textrm{.}
\end{eqnarray}
We have the bound
\begin{eqnarray*}
\# \left( \mathcal{R} \cap \mathbb{Z} \right) & \ll & \delta^{1/2} M + 1 \textrm{.}
\end{eqnarray*}
\end{lemma}

\begin{proof}
It clearly suffices to show that $\meas(\mathcal{R}) \ll \delta^{1/2} M$. Setting $z = y + A$, the condition \eqref{double inequality} can be rewritten as $Y' < \left| z^2 - A^2 \right| \leq Y$. Let us treat first the case where $z^2 - A^2 > 0$. If $Y' + A^2 > 0$ then
\begin{eqnarray*}
& & \left( Y' + A^2 \right)^{1/2} < |z| \leq \left(Y + A^2 \right)^{1/2} \textrm{.}
\end{eqnarray*}
Therefore,
\begin{eqnarray*}
\meas \left( \mathcal{R} \cap \left\{ y \in \mathbb{R}, (y+A)^2 > A^2 \right\} \right) & \ll & \left(Y + A^2 \right)^{1/2} - \left(Y' + A^2 \right)^{1/2} \\
& = & \frac{Y - Y'}{\left(Y + A^2 \right)^{1/2} + \left(Y' + A^2 \right)^{1/2}} \\
& \ll & \delta M \textrm{,}
\end{eqnarray*}
which is satisfactory. Now if $Y' + A^2 \leq 0$ then $Y + A^2 \leq Y - Y' \ll \delta M^2$ and thus
\begin{eqnarray*}
\meas \left( \mathcal{R} \cap \left\{ y \in \mathbb{R}, (y+A)^2 > A^2 \right\} \right) & \ll & \left(Y + A^2 \right)^{1/2} \\
& \ll & \delta^{1/2} M \textrm{.}
\end{eqnarray*}
Let us now deal with the case where $z^2 - A^2 \leq 0$. Under this assumption, we have $A^2 - Y \leq z^2 < A^2 - Y'$ so we can assume that $A^2 - Y' > 0$. First if $A^2 - Y \geq 0$ then
\begin{eqnarray*}
\meas \left( \mathcal{R} \cap \left\{ y \in \mathbb{R}, (y+A)^2 \leq A^2 \right\} \right) & \ll & \left(A^2 - Y' \right)^{1/2} - \left(A^2 - Y \right)^{1/2} \\
& \ll & \delta^{1/2} M \textrm{,}
\end{eqnarray*}
where we have used $a^{1/2} - b^{1/2} \leq (a-b)^{1/2}$, which is valid for any $a \geq b \geq 0$. Finally, if $A^2 - Y < 0$ then
$A^2 - Y' < Y - Y' \ll \delta M^2$ and thus
\begin{eqnarray*}
\meas \left( \mathcal{R} \cap \left\{ y \in \mathbb{R}, (y+A)^2 \leq A^2 \right\} \right) & \ll & \delta^{1/2} M \textrm{,}
\end{eqnarray*}
which completes the proof.
\end{proof}

We are now in position to prove lemma \ref{lemma tau2}.

\begin{proof}
We proceed as in the proof of lemma \ref{lemma tau}. The only thing we have to do is to repeat for all our new conditions what we have done for the conditions \eqref{A}, \eqref{B}, \eqref{C}, \eqref{D}, \eqref{E} and \eqref{E'}. By symmetry between the conditions \eqref{H}, \eqref{I} and \eqref{B}, \eqref{C}, it suffices to consider the cases of the conditions \eqref{F}, \eqref{G} and \eqref{J}. Reasoning as for the condition \eqref{D}, we see that the contribution corresponding to the condition \eqref{F} is
$\ll \delta X^{1 + \varepsilon} + X / L_2$. In the case of the condition \eqref{G}, we have the inequality
\begin{eqnarray}
\label{condition J}
& & \zeta^{-3} X_3 < |U| V^2 \leq X_3 \textrm{.}
\end{eqnarray}
Combined with $|U V| \leq X$, it implies
\begin{eqnarray}
\label{condition K}
|U| & \ll & \frac{X^2}{X_3} \textrm{.}
\end{eqnarray}
In terms of the variables $u$ and $v$, we have $\zeta^{-3} X_3 < |u| v^2 \leq \zeta^3 X_3$. Therefore, this contribution is seen to be
\begin{eqnarray*}
\sum_{\substack{\eqref{condition A} \\ \eqref{condition J}, \eqref{condition K}}}
\# \left( \mathcal{I} \times \mathcal{J} \cap \mathbb{Z}^2 \right) & \ll &
\# \left\{ (u,v) \in \mathbb{Z}_{\neq 0}^2, 
\begin{array}{l}
\zeta^{-3} X_3 < |u| v^2 \leq \zeta^3 X_3 \\
|u| \ll X^2 / X_3  \\
|u| \ll X / L_2
\end{array}
\right\} \\
& \ll &
\sum_{\substack{|u| \ll X^2/X_3 \\ |u| \ll X / L_2}}
\left( \frac{\delta X_3^{1/2}}{|u|^{1/2}} + 1 \right) \\
& \ll & \delta X + \frac{X}{L_2} \textrm{.}
\end{eqnarray*}
Finally, let us deal with the case of the condition \eqref{J}. For some $(s_1,s_2) \in ]0,1]^2$ and $(t_1,t_2) \in ]0,1]^2$, we have
\begin{eqnarray}
\label{condition K'}
\zeta^{s_1 + 2 s_2} |U| V^2 \left| \zeta^{s_1 + s_2} UV + T \right| & \leq & X_3 \textrm{,} \\
\label{condition K''}
\zeta^{t_1 + 2 t_2} |U| V^2 \left| \zeta^{t_1 + t_2} UV + T \right| & > & X_3 \textrm{.}
\end{eqnarray}
Using $|U V| \leq X$ and $T \leq 2 X$, the condition \eqref{condition K''} gives
\begin{eqnarray}
\label{condition M}
|V| & \gg & \frac{X_3}{X^2} \textrm{.}
\end{eqnarray}
For $t \in \mathbb{R}_{\neq 0}$, we set
\begin{eqnarray*}
M(t) & = & \max \left( \frac{X_3^{1/2}}{|t|^{3/2}}, \frac{T}{|t|} \right) \textrm{.}
\end{eqnarray*}
The condition \eqref{condition K'} shows that we have $|U| \leq 2 M(V)$. The inequalities \eqref{condition K'} and
\eqref{condition K''} imply
\begin{eqnarray}
\label{condition L}
& & \zeta^{-5} \frac{X_3}{|U| V^2} - \left( 1 - \zeta^{-2} \right) T < \left| UV + T \right|
\leq \frac{X_3}{|U| V^2} + \left( 1 - \zeta^{-2} \right) T \textrm{.}
\end{eqnarray}
To go back to the variables $u$ and $v$, we can proceed as we did to deduce the inequality \eqref{condition F} from the condition \eqref{condition D} in the proof of lemma \ref{lemma tau}. We get
\begin{eqnarray*}
& & \left( \zeta^{-5} - 3 \delta \right) \frac{X_3}{|U| V^2} - 5 \delta T < \left| uv + T \right| \leq
(1 + 3 \delta) \frac{X_3}{|U| V^2} + 5 \delta T \textrm{,}
\end{eqnarray*}
and thus, multiplying by $|U|$ and using $|U| \leq 2 M(V)$, we obtain
\begin{eqnarray*}
& & \left( \zeta^{-5} - 3 \delta \right) \frac{X_3}{V^2} - 10 \delta M(V) T < |u| \left| uv + T \right| \leq
\zeta (1 + 3 \delta) \frac{X_3}{V^2} + 10 \zeta \delta M(V) T \textrm{.}
\end{eqnarray*}
Setting $c = 10 \zeta^{3/2}$ for short and using $M(V) \leq \zeta^{3/2} M(v)$ and $T/|v| \leq M(v)$, we finally see that
\begin{eqnarray}
\label{condition N}
& & \left( \zeta^{-5} - 3 \delta \right) \frac{X_3}{|v|^3} - c \delta M(v)^2 < |u| \left| u + \frac{T}{v} \right| \leq
\zeta^{3} (1 + 3 \delta) \frac{X_3}{|v|^3} + \zeta c \delta M(v)^2 \textrm{.}
\end{eqnarray}
Applying lemma \ref{square} to count the number of $u$ subject to the inequality \eqref{condition N}, we get
\begin{eqnarray*}
\sum_{\substack{\eqref{condition B} \\ \eqref{condition M}, \eqref{condition L}}}
\# \left( \mathcal{I} \times \mathcal{J} \cap \mathbb{Z}^2 \right) & \ll &
\# \left\{ (u,v) \in \mathbb{Z}_{\neq 0}^2, 
\begin{array}{l}
\eqref{condition N} \\
|v| \gg X_3/X^2  \\
|v| \ll X/L_1
\end{array}
\right\} \\
& \ll & \sum_{\substack{|v| \gg X_3/X^2 \\ |v| \ll X/L_1}}
\left( \frac{\delta^{1/2} X_3^{1/2}}{|v|^{3/2}} + \frac{\delta^{1/2} T}{|v|} + 1 \right) \\
& \ll & \delta^{1/2} X^{1 + \varepsilon} + \frac{X}{L_1} \textrm{,}
\end{eqnarray*}
since $T \leq X$. In the case of $\mathcal{S}_1$, the proof can be completed as that of lemma \ref{lemma tau}. In the case of $\mathcal{S}_2$, the optimal choice of $\delta$ is seen to be $\delta = q^{3/5} X^{-2/5}$, which yields the result claimed.
\end{proof}

\subsection{Arithmetic functions}

\label{Arithmetic functions}

Let us introduce the following arithmetic functions,
\begin{eqnarray}
\label{def ast}
\varphi^{\ast}(n) & = & \prod_{p|n} \left( 1 - \frac1{p} \right) \textrm{,} \\
\label{def dag}
\varphi^{\dag}(n) & = & \prod_{p|n} \left( 1 - \frac1{p^2} \right) \textrm{, } \\
\label{def '}
\varphi'(n) & = & \prod_{p|n} \left( 1 - \frac1{p} \right)^{-1} \left( 1 + \frac1{p} - \frac1{p^2} \right)^{-1} \textrm{.}
\end{eqnarray}
Define also, for $a,b,c \geq 1$,
\begin{eqnarray*}
\psi_{a,b,c}(n) & = &
\begin{cases}
\varphi^{\ast}(n)^2 \varphi^{\ast}(\gcd(n,a))^{-1} \varphi^{\ast}(\gcd(n,b))^{-1} & \textrm{ if } \gcd(n,c) = 1 \textrm{,} \\
0 & \textrm{ otherwise.}
\end{cases}
\end{eqnarray*}
Finally, for $\sigma > 0$, let
\begin{eqnarray}
\label{varphi_sigma}
\varphi_{\sigma}(n) & = & \sum_{k|n} 2^{\omega(k)} k^{-\sigma} \textrm{,}
\end{eqnarray}
where $\omega(k)$ denotes the number of prime numbers dividing $k$. The next two lemmas are built following the reasoning of \cite[section 3]{MR2543667}.

\begin{lemma}
\label{arithmetic preliminary 0}
Let $0 < \sigma \leq 1$ be fixed. We have the estimate
\begin{eqnarray*}
\sum_{n \leq X} \psi_{a,b,c}(n) & = & \mathcal{P} \Psi(a,b,c) X + O_{\sigma} \left( \varphi_{\sigma}(c) X^{\sigma} \right) \textrm{,}
\end{eqnarray*}
where
\begin{eqnarray}
\label{def P}
\mathcal{P} & = & \prod_p \varphi'(p)^{-1} \textrm{,} \\
\nonumber
\Psi(a,b,c) & = & \varphi^{\ast}(c) \frac{\varphi^{\dag}(abc)}{\varphi^{\dag}(\gcd(a,b)c)} \varphi'(abc) \textrm{.}
\end{eqnarray}
\end{lemma}

\begin{proof}
Let us calculate the Dirichlet convolution of the function $\psi_{a,b,c}$ with the Möbius function $\mu$. We have
\begin{eqnarray*}
(\psi_{a,b,c} \ast \mu)(n) & = & \sum_{d|n} \psi_{a,b,c} \left( \frac{n}{d} \right) \mu(d) \\
& = & \prod_{p^\nu \parallel n} \left( \psi_{a,b,c} \left( p^\nu \right) - \psi_{a,b,c} \left( p^{\nu - 1} \right) \right) \textrm{.}
\end{eqnarray*}
Moreover $\psi_{a,b,c}(1) = 1$ and for all $\nu \geq 1$, we have
\begin{eqnarray*}
\psi_{a,b,c} \left( p^\nu \right) = \psi_{a,b,c}(p) =
\begin{cases}
\left( 1 - 1/p \right)^2 & \textrm{ if } p \nmid abc \textrm{, } \\
1 - 1/p & \textrm{ if } p \nmid c, p|ab \textrm{ and } p \nmid \gcd(a,b) \textrm{, } \\
1 & \textrm{ if } p \nmid c \textrm{ and } p | \gcd(a,b) \textrm{, } \\
0 & \textrm{ if } p | c \textrm{.}
\end{cases}
\end{eqnarray*}
Thus, we get
\begin{eqnarray*}
(\psi_{a,b,c} \ast \mu)(n) & = & \mu(n) 2^{\omega(n) - \omega(\gcd(n,abc))} \frac{\gcd(c,n)}{n}
\prod_{p|n, p\nmid abc} \left( 1 - \frac1{2p} \right) \textrm{,}
\end{eqnarray*}
if $\gcd(a,b,n) |c$ and $(\psi_{a,b,c} \ast \mu)(n) = 0$ otherwise. Writing
$\psi_{a,b,c} = (\psi_{a,b,c} \ast \mu) \ast 1$ yields
\begin{eqnarray*}
\sum_{n \leq X} \psi_{a,b,c}(n) & = & \sum_{n \leq X} \sum_{d|n} (\psi_{a,b,c} \ast \mu)(d) \\
& = & \sum_{d = 1}^{+ \infty} (\psi_{a,b,c} \ast \mu)(d) \left[ \frac{X}{d} \right] \textrm{.}
\end{eqnarray*}
Let $0 < \sigma \leq 1$ be fixed. Let us use the elementary estimate $[t] = t + O \left( t^{\sigma} \right)$ for $t = X/d$. Since
\begin{eqnarray*}
\sum_{d = 1}^{+ \infty} \frac{\left| (\psi_{a,b,c}  \ast \mu)(d)  \right|}{d^{\sigma}} & \leq &
\sum_{d = 1}^{+ \infty} 2^{\omega(d)} \frac{\gcd(c,d)}{d^{1 + \sigma}} \\
& \ll & \varphi_{\sigma}(c) \textrm{,}
\end{eqnarray*}
we have shown that
\begin{eqnarray*}
\sum_{n \leq X} \psi_{a,b,c}(n) & = & X \sum_{d = 1}^{+ \infty} \frac{(\psi_{a,b,c}  \ast \mu)(d)}{d}
 + O \left( \varphi_{\sigma}(c) X^{\sigma} \right) \textrm{.}
\end{eqnarray*}
A straightforward calculation finally yields
\begin{eqnarray*}
\sum_{d = 1}^{+ \infty} \frac{(\psi_{a,b,c}  \ast \mu)(d)}{d} & = & \prod_{p|c} \left( 1- \frac1{p} \right)
\prod_{\substack{p \nmid c, p|ab \\ p \nmid \gcd(a,b)}} \left( 1- \frac1{p^2} \right) \\
& & \prod_{p \nmid c, p \nmid ab} \left( 1- \frac{2}{p^2} \left( 1 - \frac1{2p} \right) \right) \\
& = & \varphi^{\ast}(c) \frac{\varphi^{\dag}(ab)}{\varphi^{\dag}(\gcd(ab,\gcd(a,b)c))} \mathcal{P} \varphi'(abc) \textrm{,}
\end{eqnarray*}
which completes the proof.
\end{proof}

\begin{lemma}
\label{arithmetic preliminary}
Let $0 < \sigma \leq 1$ be fixed. Let $0 \leq t_1 < t_2$ and $I = [t_1,t_2]$. Let also $g : \mathbb{R}_{> 0} \to \mathbb{R}$ be a function having a piecewise continuous derivative on $I$ whose sign changes at most $R_g(I)$ times on $I$. We have
\begin{eqnarray*}
\sum_{n \in I \cap \mathbb{Z}_{>0}} \psi_{a,b,c}(n) g(n) & = & \mathcal{P} \Psi(a,b,c) \int_I g(t) \D t +
O_{\sigma} \left( \varphi_{\sigma}(c) t_2^{\sigma} M_I(g) \right) \textrm{,}
\end{eqnarray*}
where $M_I(g) = (1 + R_g(I)) \sup_{t \in I \cap \mathbb{R}_{> 0}} |g(t)|$.
\end{lemma}

\begin{proof}
We only treat the case where $t_1 > 0$ since the statement for $t_1 = 0$ easily follows from it. Let $S$ be the function defined for $t > 0$ by
\begin{eqnarray*}
S(t) & = & \sum_{n \leq t} \psi_{a,b,c}(n) \textrm{.}
\end{eqnarray*}
Splitting $I$ into several ranges, we can assume that $g$ has a continuous derivative. An application of partial summation gives
\begin{eqnarray*}
\sum_{n \in ]t_1,t_2] \cap \mathbb{Z}_{>0}} \psi_{a,b,c}(n) g(n) & = &
S(t_2) g(t_2) - S(t_1) g(t_1) - \int_{t_1}^{t_2} S(t) g'(t) \D t \textrm{.}
\end{eqnarray*}
Lemma \ref{arithmetic preliminary 0} implies that
$S(t) = \mathcal{P} \Psi(a,b,c) t + O \left( \varphi_{\sigma}(c) t^{\sigma} \right)$. An integration by parts reveals that the sum to be estimated is equal to
\begin{eqnarray*}
& & \mathcal{P} \Psi(a,b,c) \int_I g(t) \D t +
O \left( \varphi_{\sigma}(c) t_2^{\sigma} \left( |g(t_2)| + |g(t_1)| + \int_I |g'(t)| \D t \right) \right) \textrm{.}
\end{eqnarray*}
It only remains to split $I$ in the $R_g(I) + 1$ ranges where $g'$ has constant sign to conclude.
\end{proof}

\subsection{Lemma for the final summation}

Let $r \geq 1$ and $\mathbf{n} = (n_1, \dots, n_r) \in \mathbb{Z}_{>0}^r$ (and by analogy, $\mathbf{d}$, $\mathbf{k}$). Let $\mathcal{V}$ be the set of $\mathbf{n} \in \mathbb{Z}_{>0}^r$ satisfying the following conditions, indexed by
$1 \leq j \leq N$,
\begin{eqnarray*}
\prod_{i=1}^r n_i^{\beta_{i,j}} & \leq & X^{\varepsilon_j} \textrm{,}
\end{eqnarray*}
where $X \geq 1$ is a quantity independent of $j$ and $\varepsilon_j \in \{ -1, 0, 1 \}$, $\beta_{i,j} \in \mathbb{Q}$ are bounded by an absolute constant and such that the polytope $\mathcal{C}$ defined by $t_1, \dots, t_r \geq 0$ and the $N$ inequalities
\begin{eqnarray*}
\sum_{i=1}^r \beta_{i,j} t_i & \leq & \varepsilon_j \textrm{,}
\end{eqnarray*}
satisfies $\mathcal{C} \subset [0,1]^r$. We are concerned with sums of the form
\begin{eqnarray*}
& & \sum_{\mathbf{n} \in \mathcal{V}} \frac{\Psi(\mathbf{n})}{n_1 \cdots n_r} \textrm{,}
\end{eqnarray*}
where $\Psi$ is an arithmetic function of $r$ variables.

\begin{lemma}
\label{final 0}
Let $f$ be the characteristic function of a polytope $\mathcal{D} \subset [0,1]^r$. We have
\begin{eqnarray*}
\sum_{n_1, \dots, n_r \leq X} \frac{f \left( \frac{\log(n_1)}{\log(X)}, \dots, \frac{\log(n_r)}{\log(X)} \right)}{n_1 \cdots n_r} & = & \vol(\mathcal{D}) \log(X)^r + O \left( \log(X)^{r-1} \right) \textrm{.}
\end{eqnarray*}
\end{lemma}

\begin{proof}
Let us reason by induction on $r$. Let $f$ be the characteristic function of $[a,b] \subset [0,1]$. We have
\begin{eqnarray*}
\sum_{n \leq X} \frac{f \left( \frac{\log(n)}{\log(X)} \right)}{n} & = &
\sum_{X^{a} \leq n \leq X^{b}} \frac1{n} \\
& = & (b - a) \log(X) + O(1) \textrm{,}
\end{eqnarray*}
as wished. Let us assume that the result is true for an integer $r-1 \geq 1$ and let us prove it for $r$. The result for
$r = 1$ applied to $n_r$ shows that the sum to be estimated is equal to
$$\sum_{n_1, \dots, n_{r-1} \leq X} \frac1{n_1 \cdots n_{r-1}}
\left( \log(X) \int_0^1 f \left( \frac{\log(n_1)}{\log(X)}, \dots, \frac{\log(n_{r-1})}{\log(X)}, t_r \right) \D t_r
+ O \left( 1 \right) \right) \textrm{.}$$
This quantity is plainly equal to
\begin{eqnarray*}
& & \log(X) \int_0^1 \left( \sum_{n_1, \dots, n_{r-1} \leq X}
\frac{f \left( \frac{\log(n_1)}{\log(X)}, \dots, \frac{\log(n_{r-1})}{\log(X)}, t_r \right) }{n_1 \cdots n_{r-1}} \right) \D t_r
+ O\left( \log(X)^{r-1} \right) \textrm{,}
\end{eqnarray*}
so the induction assumption applied to the $r-1$ remaining variables immediately completes the proof.
\end{proof}

\begin{lemma}
\label{final sum}
Let $\Psi$ be an arithmetic function of $r$ variables satisfying
\begin{eqnarray}
\label{Psi}
\sum_{\mathbf{n} \in \mathbb{Z}_{>0}^r} \frac{|(\Psi \ast \boldsymbol{\mu})(\mathbf{n})|}{n_1 \cdots n_r} 
\log \left( 2 \prod_{i=1}^r n_i \right) & < & + \infty \textrm{,} 
\end{eqnarray}
where $\boldsymbol{\mu}$ is the generalized Möbius function defined by
\begin{eqnarray*}
\boldsymbol{\mu}(n_1, \dots, n_r) & = & \mu(n_1) \cdots \mu(n_r) \textrm{.}
\end{eqnarray*}
We have
\begin{eqnarray*}
\sum_{\mathbf{n} \in \mathcal{V}} \frac{\Psi(\mathbf{n})}{n_1 \cdots n_r} & = & \vol(\mathcal{C})
\left( \sum_{\mathbf{n} \in \mathbb{Z}_{>0}^r} \frac{(\Psi \ast \boldsymbol{\mu})(\mathbf{n})}{n_1 \cdots n_r} \right) \log(X)^r + O \left( \log(X)^{r-1} \right) \textrm{.}
\end{eqnarray*}
\end{lemma}

\begin{proof}
Writing $\Psi = (\Psi \ast \boldsymbol{\mu}) \ast \mathbf{1}$, we get
\begin{eqnarray*}
\sum_{\mathbf{n} \in \mathcal{V}} \frac{\Psi(\mathbf{n})}{n_1 \cdots n_r} & = &
\sum_{\mathbf{n} \in \mathcal{V}} \sum_{d_1|n_1, \dots, d_r|n_r}
\frac{(\Psi \ast \boldsymbol{\mu})(\mathbf{d})}{n_1 \cdots n_r} \\
& = & \sum_{\mathbf{d} \in \mathbb{Z}_{>0}^r} \frac{(\Psi \ast \boldsymbol{\mu})(\mathbf{d})}{d_1 \cdots d_r}
\sum_{\mathbf{k}} \frac1{k_1 \cdots k_r} \textrm{,}
\end{eqnarray*}
where the latter sum is over $\mathbf{k}$ such that
\begin{eqnarray}
\label{condition k}
\left( \prod_{i=1}^r k_i^{\beta_{i,j}} \right) \left( \prod_{i=1}^r d_i^{\beta_{i,j}} \right) & \leq & X^{\varepsilon_j} \textrm{.}
\end{eqnarray}
Let us estimate the difference between this sum and the sum over $\mathbf{k}$ satisfying
\begin{eqnarray*}
\prod_{i=1}^r k_i^{\beta_{i,j}} & \leq & X^{\varepsilon_j} \textrm{.}
\end{eqnarray*}
For a certain $j_0$, we have
\begin{eqnarray*}
& & X^{\varepsilon_{j_0}} \left( \prod_{i=1}^r d_i^{\beta_{i,j_0}} \right)^{-1} \leq \prod_{i=1}^r k_i^{\beta_{i,j_0}} \leq X^{\varepsilon_{j_0}} \textrm{.}
\end{eqnarray*}
Summing first over $k_{i_0}$ for which $\beta_{i_0,j_0} \neq 0$, we see that since the $\beta_{i,j}$ are bounded by an absolute constant, the above difference is bounded by
\begin{eqnarray*}
\log \left( \prod_{i=1}^r d_i \right)
\sum_{k_1, \dots, \widehat{k_{i_0}}, \dots, k_r} \frac1{k_1 \cdots \widehat{k_{i_0}} \cdots k_r} & \ll &
\log \left( \prod_{i=1}^r d_i \right) \log(X)^{r-1} \textrm{.}
\end{eqnarray*}
Thus, lemma \ref{final 0} yields
\begin{eqnarray*}
\sum_{\mathbf{k}, \eqref{condition k}} \frac1{k_1 \cdots k_r} & = &
\vol(\mathcal{C}) \log(X)^r + O \left( \log \left( 2 \prod_{i=1}^r d_i \right) \log(X)^{r-1} \right) \textrm{.}
\end{eqnarray*}
The assumption \eqref{Psi} plainly implies the result.
\end{proof}

\section{Proof for the $3 \mathbf{A}_1$ surface}

\subsection{The universal torsor}

Using elementary techniques, Browning \cite{MR2362193} has made explicit a bijection between the set of points to be counted on $U_1$ and a certain set of integral points on the hypersurface defined by \eqref{torsor 1}. A little thought reveals that the result proved by Browning \cite[Lemma~$1$]{MR2362193} is equivalent to the following. We adopt the notation used by Derenthal in \cite{D-hyper}. Let $\mathcal{T}_1(B)$ be the set of $(\eta_1,\eta_2,\eta_3,\eta_4,\eta_5,\eta_6,\eta_7,\eta_8,\eta_9) \in \mathbb{Z}_{\neq 0}^9$ such that
$\eta_1,\eta_2,\eta_3,\eta_6,\eta_7 > 0$ and
\begin{eqnarray}
\label{tor 1}
\eta_4\eta_5 + \eta_1\eta_6\eta_7 + \eta_8\eta_9 & = & 0 \textrm{,}
\end{eqnarray}
and satisfying the coprimality conditions
\begin{eqnarray}
\label{gcd1} 
& & \gcd(\eta_8,\eta_1\eta_2\eta_4\eta_5\eta_6\eta_7) = 1 \textrm{,} \\
\label{gcd2} 
& & \gcd(\eta_4,\eta_1\eta_2\eta_6\eta_7\eta_9) = 1 \textrm{,} \\
\label{gcd3} 
& & \gcd(\eta_5,\eta_1\eta_3\eta_6\eta_7\eta_9) = 1 \textrm{,} \\
\label{gcd4} 
& & \gcd(\eta_6,\eta_2\eta_7\eta_9) = 1 \textrm{,} \\
\label{gcd5} 
& & \gcd(\eta_3,\eta_1\eta_2\eta_7\eta_9) = 1 \textrm{,} \\
\label{gcd6} 
& & \gcd(\eta_1,\eta_2\eta_9) = 1 \textrm{,} \\
\label{gcd7} 
& & \gcd(\eta_9,\eta_7) = 1 \textrm{.}
\end{eqnarray}
and the height conditions
\begin{eqnarray}
\label{condition1}
\eta_2\eta_3\eta_4^2\eta_5^2& \leq & B \textrm{,} \\
\label{condition2}
\eta_1^2\eta_2\eta_3\eta_6^2\eta_7^2 & \leq & B \textrm{,} \\
\label{condition3}
\eta_1\eta_3^2 |\eta_4| \eta_6^2 |\eta_8| & \leq & B \textrm{,} \\
\label{condition4}
\eta_1\eta_2^2 |\eta_5| \eta_7^2 |\eta_9| & \leq & B \textrm{.}
\end{eqnarray}

\begin{lemma}
\label{T}
We have the equality
\begin{eqnarray*}
N_{U_1,H}(B) & = & \frac1{2} \# \mathcal{T}_1(B) \textrm{.}
\end{eqnarray*}
\end{lemma}

Browning \cite[Theorem 3]{MR2362193} has used this description of the problem to prove the bound \eqref{bound N_{U_1,H}}.

It is important to notice here that the contribution to $N_{U_1,H}(B)$ coming from the $(\eta_1, \dots, \eta_9) \in \mathcal{T}_1(B)$ such that all the variables appearing in the torsor equation are bounded by an absolute constant is $\gg B$ since $\eta_2,\eta_3 \leq B^{1/2}$. That is why a result similar to \eqref{ImprovedManin} seems out of reach.

\subsection{Calculation of Peyre's constant}

The constant $c_{V_1,H}$ predicted by Peyre is
\begin{eqnarray*}
c_{V_1,H} & = & \alpha(\widetilde{V_1}) \beta(\widetilde{V_1}) \omega_H(\widetilde{V_1}) \textrm{,}
\end{eqnarray*}
where $\alpha(\widetilde{V_1}) \in \mathbb{Q}$ is the volume of a certain polytope in the dual of the effective cone of $\widetilde{V_1}$ with respect to the intersection form and where
$\beta(\widetilde{V_1}) = \# H^1(\Gal(\overline{\mathbb{Q}}/\mathbb{Q}),\Pic_{\overline{\mathbb{Q}}}(\widetilde{V_1}))$ and
\begin{eqnarray*}
\omega_H(\widetilde{V_1}) & = & \omega_{\infty} \prod_p \left( 1 - \frac1{p} \right)^6 \omega_p  \textrm{,}
\end{eqnarray*}
with $\omega_{\infty}$ and $\omega_p$ being respectively the archimedean and $p$-adic densities. The work of Derenthal \cite{MR2318651} reveals that
\begin{eqnarray*}
\alpha(\widetilde{V_1}) & = & \frac1{1440} \textrm{.}
\end{eqnarray*}
Moreover, $\beta(\widetilde{V}) = 1$ for any del Pezzo surface $V$ split over $\mathbb{Q}$ and finally, using a result of Loughran \cite[Lemma 2.3]{Loughran}, we get
\begin{eqnarray*}
\omega_p & = & 1 + \frac{6}{p} + \frac1{p^2} \textrm{.}
\end{eqnarray*}
Let us calculate $\omega_{\infty}$. Set $f_1(x) = x_0 x_1 - x_2^2$ and $f_2(x) = x_2^2 + x_1 x_2 + x_3 x_4$. We parametrize the points of $V_1$ by $x_0$, $x_2$ and $x_4$. We have
\begin{eqnarray*}
\det \begin{pmatrix}
\frac{\partial f_1}{\partial x_1} & \frac{\partial f_1}{\partial x_3} \\
\frac{\partial f_2}{\partial x_1} & \frac{\partial f_2}{\partial x_3}
\end{pmatrix} & = &
\begin{vmatrix}
x_0 & 0 \\
x_2 & x_4
\end{vmatrix} \\
& = & x_0 x_4 \textrm{.}
\end{eqnarray*}
Moreover, $x_1 = x_2^2 / x_0$ and $x_3 = - x_2^2 (x_2 + x_0) / \left( x_0 x_4 \right)$. Since $\mathbf{x} = - \mathbf{x}$ in $\mathbb{P}^4$, we have
\begin{eqnarray*}
\omega_{\infty} & = & 2 \int \int \int_{x_0, x_4 > 0, x_0, x_2^2/x_0,
x_2^2 \left| x_2 + x_0 \right| / \left| x_0 x_4 \right|, x_4 \leq 1} \frac{\D x_0 \D x_2 \D x_4}{x_0 x_4} \textrm{.}
\end{eqnarray*}
Define the function
\begin{eqnarray}
\label{equation h}
h & : & (t_4,t_5,t_6) \mapsto \max \{t_4^2t_5^2,t_6^2,|t_4|t_6^2 |t_4t_5+t_6|, |t_5| \} \textrm{.}
\end{eqnarray}
The change of variables given by $x_0 = t_4^2 t_5^2$, $x_2 = t_4 t_5 t_6$ and $x_4 = t_5$ yields
\begin{eqnarray}
\nonumber
\omega_{\infty} & = & 4 \int \int \int_{t_5,t_6>0, h(t_4,t_5,t_6) \leq 1} \D t_4 \D t_5 \D t_6 \\
\label{omega_infty}
& = & 2 \int \int \int_{t_6>0, h(t_4,t_5,t_6) \leq 1} \D t_4 \D t_5 \D t_6 \textrm{.}
\end{eqnarray}

\subsection{Restriction of the domain}

In order to be able to control the error terms showing up in our estimations, we need to assume that certain variables are greater in absolute value than a fixed power of $\log(B)$. The following result shows that this assumption does not affect the main term predicted by Manin's conjecture.

\begin{lemma}
\label{lemmalog}
Let $\mathcal{M}_1(B)$ be the overall contribution to $N_{U_1,H}(B)$ coming from the
$(\eta_1, \dots ,\eta_9) \in \mathcal{T}_1(B)$ such that $|\eta_i| \leq \log(B)^A$ for a certain $i \neq 2, 3$, where $A > 0$ is any fixed constant. We have
\begin{eqnarray*}
\mathcal{M}_1(B) & \ll_A & B \log(B)^4 \log(\log(B)) \textrm{.}
\end{eqnarray*}
\end{lemma}

\begin{lemma}
\label{lemmaHB}
Let $K_1, K_4, \dots, K_9 \geq 1/2$ and define $M_1 = M_1(K_1, K_4, \dots, K_9)$ as the number of
$(m_1, m_4, \dots, m_9) \in \mathbb{Z}^7$ such that $K_i < |m_i| \leq 2K_i$ for $i=1$ and $4 \leq i \leq 9$,
$\gcd(m_4m_5,m_1m_6m_7) = 1$ and
\begin{eqnarray}
\label{equation lemma}
m_4m_5 + m_1m_6m_7 + m_8m_9 & = & 0 \textrm{.} 
\end{eqnarray}
We have
\begin{eqnarray*}
M_1 & \ll & K_1 K_6 K_7 \min(K_4 K_5, K_8 K_9) \textrm{.}
\end{eqnarray*}
\end{lemma}

\begin{proof}
We can assume by symmetry that $K_4 K_5 \leq K_8 K_9$. Let us first deal with the case where $K_1K_6K_7 \leq K_4 K_5$. The equation \eqref{equation lemma} gives $K_8 K_9 \ll K_4 K_5$. Let $M_1'$ be the number of $(m_1, m_4, \dots, m_9) \in \mathbb{Z}^7$ to be counted in this case. We can assume by symmetry that $K_4 \leq K_5$. The idea is to view the equation \eqref{equation lemma} as a congruence modulo $m_4$. Since $|m_4| \ll (K_8 K_9)^{1/2}$, the number of $m_5$, $m_8$ and $m_9$ to be counted in $M_1'$ is at most
\begin{eqnarray*}
\# \left\{
(m_8,m_9) \in \mathbb{Z}^2,
\begin{array}{l}
K_i < |m_i| \leq 2 K_i, i \in \{8, 9 \} \\
\gcd(m_8m_9,m_1m_6m_7) = 1 \\
m_8m_9 \equiv - m_1m_6m_7 \imod{m_4}
\end{array} \right\} 
& \ll & \frac{K_8 K_9}{m_4} \textrm{.}
\end{eqnarray*}
Summing over $m_1$, $m_4$, $m_6$ and $m_7$, we finally get
\begin{eqnarray*}
M_1' & \ll & K_1 K_6 K_7 K_8 K_9 \sum_{K_4 < |m_4| \leq 2K_4} \frac1{m_4} \\
& \ll & K_1 K_6 K_7 K_4 K_5 \textrm{,}
\end{eqnarray*}
since $K_8 K_9 \ll K_4 K_5$. We now treat the case where $K_1K_6K_7 > K_4 K_5$. The equation \eqref{equation lemma} gives
$K_8 K_9 \ll K_1 K_6 K_7$. Let $M_1''$ be the number of $(m_1, m_4, \dots, m_9) \in \mathbb{Z}^7$ to be counted under this assumption. We assume by symmetry that $K_8 \leq K_9$, which yields $|m_8| \ll (K_1 K_6 K_7)^{1/2}$. We can therefore use
\cite[Lemma $5$]{MR2075628} to deduce that the number of $m_1$, $m_6$, $m_7$ and $m_9$ to be counted in $M_1''$ is at most
\begin{eqnarray*}
\# \left\{
(m_1,m_6,m_7) \in \mathbb{Z}^3,
\begin{array}{l}
K_i < |m_i| \leq 2 K_i, i \in \{1, 6, 7 \} \\
\gcd(m_1m_6m_7,m_4m_5) = 1 \\
m_1m_6m_7 \equiv - m_4 m_5 \imod{m_8}
\end{array} \right\}
& \ll & \frac{K_1 K_6 K_7}{\varphi(m_8)} \textrm{.}
\end{eqnarray*}
We obtain in this case
\begin{eqnarray*}
M_1'' & \ll & K_1 K_6 K_7 K_4 K_5 \sum_{K_8 < |m_8| \leq 2K_8} \frac1{\varphi(m_8)} \\
& \ll & K_1 K_6 K_7 K_4 K_5 \textrm{,}
\end{eqnarray*}
as wished.
\end{proof}

We are now in position to prove lemma \ref{lemmalog}. Note that the following proof is largely inspired by Browning's proof of \cite[Theorem 3]{MR2362193}.

\begin{proof}
Let $Y_i \geq 1/2$ for $i = 1, \dots, 9 $ and define $\mathcal{N}_1 = \mathcal{N}_1(Y_1, \dots, Y_9)$ as the \text{contribution} of the $(\eta_1, \dots, \eta_9) \in \mathcal{T}_1(B)$ satisfying $Y_i < |\eta_i| \leq 2 Y_i$ for $i = 1, \dots, 9 $. The height conditions imply that either $\mathcal{N}_1 = 0$ or we have the inequalities
\begin{eqnarray}
\label{condition1bis}
Y_2 Y_3 Y_4^2 Y_5^2 & \leq & B \textrm{,} \\
\label{condition2bis}
Y_1^2 Y_2 Y_3 Y_6^2 Y_7^2 & \leq & B \textrm{,}\\
\label{condition3bis}
Y_1 Y_3^2 Y_4 Y_6^2 Y_8 & \leq & B \textrm{,} \\
\label{condition4bis}
Y_1 Y_2^2 Y_5 Y_7^2 Y_9 & \leq & B \textrm{.}
\end{eqnarray}
Using lemma \ref{lemmaHB} and summing over $\eta_2$ and $\eta_3$, we get
\begin{eqnarray*}
\mathcal{N}_1 & \ll & Y_1 Y_2 Y_3 Y_6 Y_7 \min(Y_4 Y_5, Y_8 Y_9) \textrm{.}
\end{eqnarray*}
Let us recall the following basic estimates. Assume that we have to sum over all the ranges $Y < |y| \leq 2Y$ for all
$|y| \leq \mathcal{Y}$, then
\begin{eqnarray*}
\sum_{Y \leq \mathcal{Y}} Y^{\delta} & \ll_{\delta} &
\begin{cases}
1 & \textrm{ if } \delta < 0 \textrm{, } \\
\log(\mathcal{Y}) & \textrm{ if } \delta = 0 \textrm{, } \\
\mathcal{Y}^{\delta} & \textrm{ if } \delta > 0 \textrm{.}
\end{cases}
\end{eqnarray*}
In the following, the notation $\sum_{\widehat{Y}}$ means that the summation is over all the $Y_i \neq Y$. We only treat the case where $Y_4 Y_5 \leq Y_8 Y_9$ (the case where $Y_4 Y_5 > Y_8 Y_9$ is identical). Let us first assume that $Y_1 Y_6 Y_7 \leq Y_4 Y_5$. We start by summing over
\begin{eqnarray*}
Y_6 & \leq &
\min \left( \frac{Y_4 Y_5}{Y_1 Y_7}, \frac{B^{1/2}}{Y_{1}^{1/2} Y_3 Y_4^{1/2} Y_8^{1/2}} \right) \\
& \leq & \frac{Y_4^{1/4} Y_5^{1/2} B^{1/4}}{Y_{1}^{3/4} Y_3^{1/2} Y_7^{1/2} Y_8^{1/4}} \textrm{,}
\end{eqnarray*}
and over $Y_3$ using \eqref{condition1bis}. We get in this case
\begin{eqnarray*}
\sum_{Y_i} \mathcal{N}_1 & \ll & \sum_{Y_i} Y_1 Y_2 Y_3 Y_4 Y_5 Y_6 Y_7 \\
& \ll & B^{1/4} \sum_{\widehat{Y_6}} Y_1^{1/4} Y_2 Y_3^{1/2} Y_4^{5/4} Y_5^{3/2} Y_7^{1/2} Y_8^{-1/4} \\
& \ll & B^{3/4} \sum_{\widehat{Y_3},\widehat{Y_6}} Y_1^{1/4} Y_2^{1/2} Y_4^{1/4} Y_5^{1/2} Y_7^{1/2} Y_8^{-1/4} \textrm{.}
\end{eqnarray*}
Now sum over $Y_2$ using \eqref{condition4bis} and over $Y_4 \leq  Y_5^{-1} Y_8 Y_9$ to obtain
\begin{eqnarray*}
\sum_{Y_i} \mathcal{N}_1 & \ll &
B \sum_{\widehat{Y_2},\widehat{Y_3},\widehat{Y_6}} Y_4^{1/4} Y_5^{1/4} Y_8^{-1/4} Y_9^{-1/4} \\
& \ll & B \sum_{\widehat{Y_2},\widehat{Y_3},\widehat{Y_4},\widehat{Y_6}} 1 \textrm{.}
\end{eqnarray*}
We could have summed over $Y_5$ instead of $Y_4$ and over $Y_7$ instead of $Y_6$, so if we assume that $|\eta_i| \leq \log(B)^A$ for a certain $i \neq 2, 3$, where $A > 0$ is any fixed constant, we get an overall contribution $\ll_A B \log(B)^4 \log(\log(B))$. Let us now assume that we have $Y_1 Y_6 Y_7 > Y_4 Y_5$. Since $Y_4 Y_5 \leq Y_8 Y_9$, we deduce from the equation \eqref{tor 1} that $Y_1 Y_6 Y_7 \ll Y_8 Y_9$. Summing over $Y_3$ using \eqref{condition3bis} and over $Y_2$ using \eqref{condition4bis} yields
\begin{eqnarray*}
\sum_{Y_i} \mathcal{N}_1 & \ll & B \sum_{\widehat{Y_2},\widehat{Y_3}} Y_4^{1/2} Y_5^{1/2} Y_8^{-1/2} Y_9^{-1/2} \\
& \ll & B \sum_{\widehat{Y_2},\widehat{Y_3},\widehat{Y_4}} Y_1^{1/2} Y_6^{1/2} Y_7^{1/2} Y_8^{-1/2} Y_9^{-1/2} \\
& \ll & B \sum_{\widehat{Y_2},\widehat{Y_3},\widehat{Y_4},\widehat{Y_6}} 1 \textrm{,}
\end{eqnarray*}
where we have summed over $Y_4 < Y_1 Y_5^{-1} Y_6 Y_7$ and $Y_6 \ll Y_1^{-1} Y_7^{-1} Y_8 Y_9$. We can now conclude exactly as in the first case.
\end{proof}

\subsection{Setting up}

To be able to apply lemma \ref{lemma tau}, we need to assume that
\begin{eqnarray*}
|\eta_9| & \leq & |\eta_8| \textrm{.}
\end{eqnarray*}
Note that this assumption together with the equation \eqref{tor 1} and the height conditions \eqref{condition1} and \eqref{condition2} yield the following condition which plays a crucial role in the proof,
\begin{eqnarray}
\label{conditionkappa}
\eta_9^2 & \leq & 2 \frac{B^{1/2}}{\eta_2^{1/2} \eta_3^{1/2}} \textrm{.}
\end{eqnarray}
The symmetry given by $(\eta_3,\eta_4,\eta_6,\eta_8) \mapsto (\eta_2,\eta_5,\eta_7,\eta_9)$ and the following lemma prove that it suffices to multiply our main term by $2$ to take into account this new assumption.

\begin{lemma}
\label{lemmaequal}
Let $N_0(B)$ be the overall contribution from the $(\eta_1, \dots, \eta_9) \in \mathcal{T}_1(B)$ such that $|\eta_8| = |\eta_9|$. We have
\begin{eqnarray*}
N_0(B) & \ll & B \log(B) \textrm{.}
\end{eqnarray*}
\end{lemma}

\begin{proof}
Note that we have the inequality \eqref{conditionkappa} here too. Define
\begin{eqnarray*}
\mathcal{X} & = & \frac{B^{1/2}}{\eta_2^{1/2}\eta_3^{1/2}} \textrm{.}
\end{eqnarray*}
The number of $\eta_1$, $\eta_4$ and $\eta_5$ to be counted is
\begin{eqnarray*}
& \ll & \# \left\{ \left( \eta_1, \eta_4, \eta_5 \right) \in \mathbb{Z}_{>0} \times \mathbb{Z}_{\neq 0}^2,
\begin{array}{l}
\eta_4\eta_5 = \pm \eta_9^2 + \eta_1\eta_6\eta_7 \\
|\eta_4\eta_5| \leq \mathcal{X} \\
\end{array}
\right\} \\
& \ll & \# \left\{ \left( \eta_4, \eta_5 \right) \in \mathbb{Z}_{\neq 0}^2,
\begin{array}{l}
\eta_4\eta_5 \equiv \pm \eta_9^2 \imod{\eta_6\eta_7} \\
|\eta_4\eta_5| \leq \mathcal{X} \\
\end{array}
\right\} \\
& \ll & \sum_{\substack{1 \leq |n| \leq \mathcal{X} \\ n \equiv \pm \eta_9^2 \imod{\eta_6\eta_7}}} \tau(|n|) \\
& \ll & \mathcal{X}^{\varepsilon} \left( \frac{\mathcal{X}}{\eta_6\eta_7} + 1 \right) \textrm{,}
\end{eqnarray*}
for all $\varepsilon > 0$. Taking $\varepsilon = 1/4$ and summing over $\eta_9$ using the condition \eqref{conditionkappa}, we get
\begin{eqnarray*}
N_0(B) & \ll & \sum_{\eta_2,\eta_3,\eta_6,\eta_7} \left( \frac{B^{7/8}}{\eta_2^{7/8}\eta_3^{7/8}\eta_6\eta_7} + \frac{B^{3/8}}{\eta_2^{3/8}\eta_3^{3/8}} \right) \\
& \ll & \sum_{\eta_2,\eta_6,\eta_7} \frac{B}{\eta_2\eta_6^{5/4}\eta_7^{5/4}} \\
& \ll & B \log(B) \textrm{,}
\end{eqnarray*}
where we have summed over $\eta_3$ using \eqref{condition2}.
\end{proof}

Since $(\eta_8,\eta_9) \mapsto (-\eta_8,-\eta_9)$ is a bijection between the set of solutions with $\eta_9 > 0$ and the set of solutions with $\eta_9 < 0$, we can assume that $\eta_9 > 0$ if we multiply our main term by $2$ once again. Furthermore, we need to assume that $\eta_4$ and $\eta_5$ are greater in absolute value than a power of $\log(B)$. To sum up, denote by $N(A,B)$ the contribution to $N_{U_1,H}(B)$ from the $(\eta_1, \dots, \eta_9) \in \mathcal{T}_1(B)$ satisfying
\begin{eqnarray}
\label{condition6}
& & 0 < \eta_9 \leq |\eta_8| \textrm{,} \\
\label{condition7}
& & \log(B)^A \leq |\eta_4| \textrm{,} \\
\label{condition8}
& & \log(B)^A \leq |\eta_5| \textrm{,}
\end{eqnarray}
where $A > 0$ is a constant to be chosen later. Note that combining the conditions \eqref{condition1} and \eqref{condition7}, we get
\begin{eqnarray}
\label{new condition}
\log(B)^{2A}\eta_2\eta_3\eta_5^2 & \leq & B \textrm{.}
\end{eqnarray}
This inequality is crucial in the estimation of our error terms. Lemmas \ref{T}, \ref{lemmalog} and \ref{lemmaequal} yield the following result.

\begin{lemma} 
\label{lemmaA}
For any fixed $A > 0$, we have
\begin{eqnarray*}
N_{U_1,H}(B) & = & 2 N(A,B) + O \left(B \log(B)^4 \log(\log(B)) \right) \textrm{.}
\end{eqnarray*}
\end{lemma}

Our goal is now to estimate $N(A,B)$ and for this, we start by investigating the contribution of the variables $\eta_4$, $\eta_5$ and $\eta_8$. The idea is to view the torsor equation \eqref{tor 1} as a congruence modulo $\eta_9$. For this, we replace the height conditions \eqref{condition3} and \eqref{condition6} by the following (we keep denoting them by \eqref{condition3} and \eqref{condition6}), obtained using the torsor equation \eqref{tor 1},
\begin{eqnarray*}
\eta_1 \eta_3^2 |\eta_4| \eta_6^2 | \eta_4\eta_5 + \eta_1\eta_6\eta_7 | \eta_9^{-1} & \leq & B \textrm{,} \\
\eta_9^2 & \leq & |\eta_4\eta_5 + \eta_1\eta_6\eta_7 | \textrm{.}
\end{eqnarray*}
Set $\boldsymbol{\eta}' = (\eta_1, \eta_2, \eta_3, \eta_6, \eta_7, \eta_9) \in \mathbb{Z}_{>0}^6$. Assume that
$\boldsymbol{\eta}' \in \mathbb{Z}_{>0}^6$ is fixed and subject to the height conditions \eqref{condition2} and \eqref{conditionkappa} and to the coprimality conditions \eqref{gcd4}, \eqref{gcd5}, \eqref{gcd6} and \eqref{gcd7}. Let $N(\boldsymbol{\eta}',B)$ be the number of $\eta_4$, $\eta_5$ and $\eta_8$ satisfying the torsor equation \eqref{tor 1}, the height conditions \eqref{condition1}, \eqref{condition3} and \eqref{condition4}, the conditions \eqref{condition6}, \eqref{condition7} and \eqref{condition8} and the coprimality conditions \eqref{gcd1}, \eqref{gcd2} and \eqref{gcd3}. Recalling the definition \eqref{def ast} of $\varphi^{\ast}$, we have the following result.

\begin{lemma}
\label{lemma inter}
For any fixed $A \geq 7$, we have
\begin{eqnarray*}
N(\boldsymbol{\eta}',B) & = & \frac1{\eta_9}
\sum_{\substack{k_8|\eta_2 \\ \gcd(k_8,\eta_7) = 1}} \frac{\mu(k_8)}{k_8 \varphi^{\ast}(k_8\eta_9)}
\sum_{\substack{k_4|\eta_1\eta_2\eta_6\eta_7 \\ \gcd(k_4, k_8\eta_9)=1}} \mu(k_4)
\sum_{\substack{k_5|\eta_1\eta_3\eta_6\eta_7 \\ \gcd(k_5, k_8\eta_9)=1}} \mu(k_5) \\
& & \sum_{\substack{\ell_4|k_8\eta_9 \\ \ell_5|k_8\eta_9}} \mu(\ell_4) \mu(\ell_5) C(\boldsymbol{\eta}',B) + R(\boldsymbol{\eta}',B) \textrm{,}
\end{eqnarray*}
where, with the notations $\eta_4 = k_4 \ell_4 \eta_4''$ and $\eta_5 = k_5 \ell_5 \eta_5''$,
\begin{eqnarray*}
C(\boldsymbol{\eta}',B) & = &
\# \left\{ \left( \eta_4'', \eta_5'' \right) \in \mathbb{Z}_{\neq 0}^2,
\begin{array}{l}
\eqref{condition1}, \eqref{condition3}, \eqref{condition4} \\
\eqref{condition6}, \eqref{condition7}, \eqref{condition8} \\
\end{array}
\right\} \textrm{,}
\end{eqnarray*}
and $\sum_{\boldsymbol{\eta}'} R(\boldsymbol{\eta}',B) \ll B \log(B)^2$.
\end{lemma}

The thrust of lemma \ref{lemma inter} is that the summation over $\eta_8$ has been carried out, which explains the absence of the torsor equation in $C(\boldsymbol{\eta}',B)$. The remainder of this section is devoted to proving lemma \ref{lemma inter}.

Let us remove the coprimality condition \eqref{gcd1} using a Möbius inversion. We get
\begin{eqnarray*}
N(\boldsymbol{\eta}',B) & = & \sum_{k_8|\eta_1\eta_2\eta_4\eta_5\eta_6\eta_7} \mu(k_8) S_{k_8}(\boldsymbol{\eta}',B) \textrm{,}
\end{eqnarray*}
where
\begin{eqnarray*}
S_{k_8}(\boldsymbol{\eta}',B) & = & \# \left\{ \left( \eta_4, \eta_5, \eta_8' \right) \in \mathbb{Z}_{\neq 0}^3 ,
\begin{array}{l}
\eta_4 \eta_5 + k_8\eta_8'\eta_9 = - \eta_1\eta_6\eta_7 \\
\eqref{condition1}, \eqref{condition3}, \eqref{condition4} \\
\eqref{condition6}, \eqref{condition7}, \eqref{condition8} \\
\eqref{gcd2}, \eqref{gcd3}
\end{array}
\right\} \textrm{.}
\end{eqnarray*}
It is clear that if $\gcd(k_8,\eta_1\eta_6\eta_7) \neq 1$ or $\gcd(k_8,\eta_4\eta_5) \neq 1$ then
$\gcd(\eta_4\eta_5,\eta_1\eta_6\eta_7) \neq 1$ and thus $S_{k_8}(\boldsymbol{\eta}',B) = 0$. We can therefore assume that
$\gcd(k_8,\eta_1\eta_4\eta_5\eta_6\eta_7) = 1$. We have
\begin{eqnarray*}
S_{k_8}(\boldsymbol{\eta}',B) & = &
\# \left\{ (\eta_4, \eta_5) \in \mathbb{Z}_{\neq 0}^2 ,
\begin{array}{l}
\eta_4 \eta_5 \equiv - \eta_1\eta_6\eta_7 \imod{k_8 \eta_9} \\
\eqref{condition1}, \eqref{condition3}, \eqref{condition4} \\
\eqref{condition6}, \eqref{condition7}, \eqref{condition8} \\
\eqref{gcd2}, \eqref{gcd3}
\end{array}
\right\} + R_0(\boldsymbol{\eta}',B) \textrm{,}
\end{eqnarray*}
where the error term $R_0(\boldsymbol{\eta}',B)$ comes from the fact $\eta_8'$ has to be non-zero. Otherwise, we would have
$\eta_4 \eta_5 = - \eta_1\eta_6\eta_7$ and so the coprimality condition $\gcd(\eta_4 \eta_5,\eta_1\eta_6\eta_7) = 1$ would give
$|\eta_4| = |\eta_5| = \eta_1 = \eta_6 = \eta_7 = 1$. Summing over $\eta_9$ using \eqref{conditionkappa}, we obtain
\begin{eqnarray*}
\sum_{k_8,\boldsymbol{\eta}'} |\mu(k_8)| R_0(\boldsymbol{\eta}',B) & \ll & \sum_{\eta_2, \eta_3, \eta_9} 2^{\omega(\eta_2)} \\
& \ll & \sum_{\eta_2, \eta_3} 2^{\omega(\eta_2)} \frac{B^{1/4}}{\eta_2^{1/4} \eta_3^{1/4}} \\
& \ll & B \log(B)^2 \textrm{.}
\end{eqnarray*}
Let us remove the coprimality conditions \eqref{gcd2} and \eqref{gcd3}. The main term of $N(\boldsymbol{\eta}',B)$ is equal to
\begin{eqnarray*}
& & \sum_{\substack{k_8|\eta_2 \\ \gcd(k_8,\eta_1\eta_6\eta_7) = 1}} \mu(k_8)
\sum_{\substack{k_4|\eta_1\eta_2\eta_6\eta_7\eta_9 \\ \gcd(k_4, k_8\eta_9)=1}} \mu(k_4)
\sum_{\substack{k_5|\eta_1\eta_3\eta_6\eta_7\eta_9 \\ \gcd(k_5, k_8\eta_9)=1}} \mu(k_5) S(\boldsymbol{\eta}',B) \textrm{,}
\end{eqnarray*}
where, with the notations $\eta_4 = k_4\eta_4'$ and $\eta_5 = k_5\eta_5'$,
\begin{eqnarray*}
S(\boldsymbol{\eta}',B) & = & \# \left\{ \left( \eta_4', \eta_5' \right) \in \mathbb{Z}_{\neq 0}^2,
\begin{array}{l}
\eta_4'\eta_5' \equiv - (k_4k_5)^{-1}\eta_1\eta_6\eta_7 \imod{k_8 \eta_9} \\
\eqref{condition1}, \eqref{condition3}, \eqref{condition4} \\
\eqref{condition6}, \eqref{condition7}, \eqref{condition8} \\
\end{array}
\right\} \textrm{.}
\end{eqnarray*}
Indeed, $k_4$ and $k_5$ are invertible modulo $k_8\eta_9$ since $\gcd(k_8\eta_9,\eta_1\eta_6\eta_7) = 1$. We can therefore remove $\eta_9$ from the conditions on $k_4$ and $k_5$. Having in mind that our aim is to apply lemma \ref{lemma tau}, we define
\begin{eqnarray*}
X & = & \frac{B^{1/2}}{k_4 k_5 \eta_2^{1/2} \eta_3^{1/2}} \textrm{.}
\end{eqnarray*}
Let us prove that we can assume that $k_8 \leq (2k_4k_5)^{-1/2} X^{1/6}$, the contribution coming from the condition
$k_8 > (2k_4k_5)^{-1/2} X^{1/6}$ being negligible. Indeed, let $N'(\boldsymbol{\eta}',B)$ be this contribution and define
$a = - (k_4k_5)^{-1}\eta_1\eta_6\eta_7$. We have
\begin{eqnarray*}
S(\boldsymbol{\eta}',B) & \leq & \# \left\{ \left( \eta_4', \eta_5' \right) \in \mathbb{Z}_{\neq 0}^2,
\begin{array}{l}
\eta_4'\eta_5' \equiv a \imod{k_8\eta_9} \\
|\eta_4'\eta_5'| \leq X \\
\end{array}
\right\} \\
& = & 2 \sum_{\substack{1 \leq |n| \leq X \\ n \equiv a \imod{k_8\eta_9}}} \tau(|n|) \textrm{.}
\end{eqnarray*}
Thus, for all $\varepsilon > 0$,
\begin{eqnarray*}
S(\boldsymbol{\eta}',B) & \ll & X^{\varepsilon} \left( \frac{X}{k_8\eta_9} + 1 \right) \\
& \ll & (k_4k_5)^{1/4} \frac{X^{1 + \varepsilon - 1/12}}{k_8^{1/2} \eta_9} + X^{\varepsilon} \textrm{,}
\end{eqnarray*}
since $k_8 > k_8^{1/2} (2k_4k_5)^{-1/4} X^{1/12}$. Note that if $k_4$, $k_5$ or $k_8$ appears in the denominator then the arithmetic function involved by the corresponding Möbius inversion has average order $O(1)$ and therefore does not play any role in the estimation of the contribution of the error term. Thus we have
\begin{eqnarray*}
N'(\boldsymbol{\eta}',B) & \ll & \frac1{\eta_9} \left( \frac{B^{1/2}}{\eta_2^{1/2}\eta_3^{1/2}} \right)^{1 + \varepsilon - 1/12}
+ 2^{\omega(\eta_2)} \left( \frac{B^{1/2}}{\eta_2^{1/2}\eta_3^{1/2}} \right)^{\varepsilon} \textrm{.}
\end{eqnarray*}
Let us estimate the overall contribution of the right-hand side summing over $\boldsymbol{\eta}'$. Using the condition \eqref{conditionkappa} to sum over $\eta_9$, we get
\begin{eqnarray*}
\sum_{\eta_9} N'(\boldsymbol{\eta}',B) & \ll & \left( \frac{B^{1/2}}{\eta_2^{1/2}\eta_3^{1/2}} \right)^{1 + 2 \varepsilon - 1/12}
+ 2^{\omega(\eta_2)} \left( \frac{B^{1/2}}{\eta_2^{1/2}\eta_3^{1/2}} \right)^{1/2 + \varepsilon} \textrm{.}
\end{eqnarray*}
Taking $\varepsilon = 1/48$ and summing over $\eta_3$ using the condition \eqref{condition2}, we obtain
\begin{eqnarray*}
\sum_{\boldsymbol{\eta}'} N'(\boldsymbol{\eta}',B) & \ll &
\sum_{\eta_1,\eta_2,\eta_6,\eta_7} \left( \frac{B}{\eta_1^{25/24}\eta_2\eta_6^{25/24}\eta_7^{25/24}}
+ 2^{\omega(\eta_2)} \frac{B}{\eta_1^{71/48}\eta_2\eta_6^{71/48}\eta_7^{71/48}} \right) \\
& \ll & B \log(B)^2 \textrm{.} 
\end{eqnarray*}
Therefore, $N(\boldsymbol{\eta}',B)$ is the sum of the main term
\begin{eqnarray*}
& & \sum_{\substack{k_4|\eta_1\eta_2\eta_6\eta_7 \\ \gcd(k_4,\eta_9)=1}} \mu(k_4)
\sum_{\substack{k_5|\eta_1\eta_3\eta_6\eta_7 \\ \gcd(k_5,\eta_9)=1}} \mu(k_5)
\sum_{\substack{k_8|\eta_2, k_8 \leq (2k_4k_5)^{-1/2} X^{1/6} \\ \gcd(k_8,k_4k_5\eta_1\eta_6\eta_7) = 1}} \mu(k_8) S(\boldsymbol{\eta}',B) \textrm{,}
\end{eqnarray*}
and an error term whose overall contribution is $\ll B \log(B)^2$. Note that thanks to the condition \eqref{conditionkappa}, we now have $k_8 \eta_9 \leq X^{2/3}$. We want to apply lemma \ref{lemma tau} with
$L_1 = \log(B)^A/k_4$, $L_2 = \log(B)^A/k_5$ and $T = \eta_1\eta_6\eta_7/(k_4k_5)$. Since $T \leq X$ by \eqref{condition2} and
$k_8 \eta_9 \leq X^{2/3}$, lemma \ref{lemma tau} proves that
\begin{eqnarray*}
S(\boldsymbol{\eta}',B) & = & S^{\ast}(\boldsymbol{\eta}',B) +
O \left( \frac{X^{2/3 + \varepsilon}}{(k_8\eta_9)^{1/2}} + \frac{X}{\varphi(k_8\eta_9)} \left( \frac{k_4}{\log(B)^A} + \frac{k_5}{\log(B)^A} \right) \right) \textrm{,}
\end{eqnarray*}
for all $\varepsilon > 0$, with
\begin{eqnarray*}
S^{\ast}(\boldsymbol{\eta}',B) & = & \frac1{\varphi(k_8\eta_9)} \# \left\{ \left( \eta_4', \eta_5' \right) \in \mathbb{Z}_{\neq 0}^2,
\begin{array}{l}
\gcd(\eta_4'\eta_5',k_8\eta_9) = 1 \\
\eqref{condition1}, \eqref{condition3}, \eqref{condition4} \\
\eqref{condition6}, \eqref{condition7}, \eqref{condition8} \\
\end{array}
\right\} \textrm{.}
\end{eqnarray*}
As explained above, $k_4$, $k_5$ and $k_8$ do not play any role in the estimation of the contribution of the first error term. Using \eqref{conditionkappa} to sum over $\eta_9$, we find that the contribution of the first error term is
\begin{eqnarray*}
\sum_{\boldsymbol{\eta}'}
\frac{B^{1/3 + \varepsilon}}{\eta_2^{1/3 + \varepsilon} \eta_3^{1/3 + \varepsilon} \eta_9^{1/2}}
& \ll & \sum_{\eta_1,\eta_2,\eta_3,\eta_6,\eta_7}
\frac{B^{11/24 + \varepsilon}}{\eta_2^{11/24 + \varepsilon}\eta_3^{11/24 + \varepsilon}} \\
& \ll & \sum_{\eta_1,\eta_2,\eta_6,\eta_7}
\frac{B}{\eta_1^{13/12 - 2 \varepsilon} \eta_2 \eta_6^{13/12 - 2 \varepsilon} \eta_7^{13/12 - 2 \varepsilon}} \\
& \ll & B \log(B) \textrm{,}
\end{eqnarray*}
for $\varepsilon = 1/48$ and where we have summed over $\eta_3$ using \eqref{condition2}. Furthermore, the contribution of the second error term is
\begin{eqnarray*}
\sum_{\boldsymbol{\eta}'} 2^{\omega(\eta_1\eta_2\eta_6\eta_7)}
\frac{B^{1/2} \log(B)^{-A}}{\eta_2^{1/2} \eta_3^{1/2} \eta_9}
& \ll & \sum_{\eta_1,\eta_2,\eta_6,\eta_7,\eta_9} 2^{\omega(\eta_1\eta_2\eta_6\eta_7)}
\frac{B \log(B)^{-A}}{\eta_1\eta_2\eta_6\eta_7\eta_9} \\
& \ll & B \log(B)^{9 - A} \textrm{,}
\end{eqnarray*}
which is satisfactory if $A \geq 7$. The contribution of the third error term is easily seen to be also $\ll B \log(B)^{9 - A}$. Furthermore, we have
\begin{eqnarray*}
S^{\ast}(\boldsymbol{\eta}',B) & = & \frac1{\varphi(k_8\eta_9)} \sum_{\ell_4|k_8\eta_9} \mu(\ell_4) \sum_{\ell_5|k_8\eta_9} \mu(\ell_5) C(\boldsymbol{\eta}',B) \textrm{,}
\end{eqnarray*}
where we have set $\eta_4' = \ell_4 \eta_4''$ and $\eta_5' = \ell_5 \eta_5''$. We now prove that we can remove the condition
$k_8 \leq (2k_4k_5)^{-1/2} X^{1/6}$ from the sum over $k_8$. The height condition \eqref{condition1} plainly gives
\begin{eqnarray*}
C(\boldsymbol{\eta}',B) & \ll & \left( \frac{X}{\ell_4 \ell_5} \right)^{1 + \varepsilon} \textrm{.} 
\end{eqnarray*}
Let us bound the overall contribution corresponding to $k_8 > (2k_4k_5)^{-1/2} X^{1/6}$. Note that
$\varphi(k_8\eta_9) = k_8\eta_9 \varphi^{\ast}(k_8\eta_9)$ and write $k_8 > k_8^{1/2} (2k_4k_5)^{-1/4} X^{1/12}$. Once again, the Möbius inversions do not play any part in the estimation of the contribution of this error term, which we find to be less than
\begin{eqnarray*}
\sum_{\boldsymbol{\eta}'} \frac1{\eta_9} \left( \frac{B^{1/2}}{\eta_2^{1/2}\eta_3^{1/2}} \right)^{1 + \varepsilon - 1/12} & \ll &
\sum_{\eta_1,\eta_2,\eta_6,\eta_7,\eta_9} \frac{B}{\eta_1^{25/24}\eta_2\eta_6^{25/24}\eta_7^{25/24}\eta_9} \\
& \ll & B \log(B)^2 \textrm{,}
\end{eqnarray*}
where we have set $\varepsilon = 1/24$. Finally, we can remove the condition $\gcd(k_8,\eta_1 \eta_6) = 1$ from the sum over $k_8$ since $k_8|\eta_2$ and $\gcd(\eta_1 \eta_6,\eta_2) = 1$, which completes the proof of lemma \ref{lemma inter}.

\subsection{Summing over $\eta_4''$, $\eta_5''$ and $\eta_6$}

\label{Summing}

We intend to sum also over $\eta_6$ and thus we set
$\boldsymbol{\eta} = (\eta_1, \eta_2, \eta_3, \eta_7, \eta_9) \in \mathbb{Z}_{>0}^5$. For $(r_1,r_2,r_3,r_7,r_9) \in \mathbb{Q}^5$, we introduce the useful notation
\begin{eqnarray*}
\boldsymbol{\eta}^{(r_1,r_2,r_3,r_7,r_9)} & = & \eta_1^{r_1} \eta_2^{r_2} \eta_3^{r_3} \eta_7^{r_7} \eta_9^{r_9} \textrm{.}
\end{eqnarray*}
Setting
\begin{align*}
& Y_4 = \frac{\boldsymbol{\eta}^{(1,3/2,-1/2,2,1)}}{B^{1/2}} \textrm{,} &
& Y_4'' = \frac{Y_4}{k_4\ell_4} \textrm{,} \\
& Y_5 = \frac{B}{\boldsymbol{\eta}^{(1,2,0,2,1)}} \textrm{,} &
& Y_5'' = \frac{Y_5}{k_5\ell_5} \textrm{,} \\
& Y_6 = \frac{B^{1/2}}{\boldsymbol{\eta}^{(1,1/2,1/2,1,0)}} \textrm{,} &
\end{align*}
and recalling the definition \eqref{equation h} of the function $h$, the height conditions \eqref{condition1}, \eqref{condition2}, \eqref{condition3} and \eqref{condition4} can be rewritten as
\begin{eqnarray*}
h \left( \frac{\eta_4''}{Y_4''}, \frac{\eta_5''}{Y_5''}, \frac{\eta_6}{Y_6} \right) & \leq & 1 \textrm{.}
\end{eqnarray*}
We also define the real-valued functions
\begin{eqnarray*}
g_1 & : & (t_5,t_6,t;\boldsymbol{\eta},B) \mapsto
\int_{h(t_4,t_5,t_6) \leq 1, t \leq |t_4 t_5 + t_6|, |t_4|Y_4 \geq \log(B)^A} \D t_4 \textrm{,} \\
g_2 & : & (t_6,t;\boldsymbol{\eta},B) \mapsto \int_{|t_5| Y_5 \geq \log(B)^A} g_1(t_5,t_6,t;\boldsymbol{\eta},B) \D t_5 \textrm{,} \\
g_3 & : & (t;\boldsymbol{\eta},B) \mapsto \int_{t_6 Y_6 \geq 1} g_2(t_6,t;\boldsymbol{\eta},B) \D t_6 \textrm{,} \\
g_4 & : & t \mapsto \int \int \int_{t_6 > 0, h(t_4,t_5,t_6) \leq 1, t \leq |t_4 t_5 + t_6|} \D t_4 \D t_5 \D t_6 \textrm{.}
\end{eqnarray*}
The condition $t \leq |t_4 t_5 + t_6|$ corresponds to the condition \eqref{condition6} which becomes, in our new notations,
\begin{eqnarray*}
\frac{\eta_9^2}{Y_4Y_5} & \leq & \left| \frac{\eta_4''}{Y_4''} \frac{\eta_5''}{Y_5''} + \frac{\eta_6}{Y_6} \right| \textrm{.}
\end{eqnarray*}
We denote by $\kappa$ the left-hand side of this inequality. Note that the condition \eqref{conditionkappa} is exactly
$\kappa \leq 2$.

\begin{lemma}
\label{bounds}
We have the bounds
\begin{eqnarray*}
g_1(t_5,t_6,t;\boldsymbol{\eta},B) & \ll & |t_5|^{-2/3} t_6^{-2/3} \textrm{,} \\
g_2(t_6,t;\boldsymbol{\eta},B) & \ll & t_6^{-2/3} \textrm{.}
\end{eqnarray*}
\end{lemma}

\begin{proof}
Recall the definition \eqref{equation h} of the function $h$. To begin with, we see that the condition
$|t_4|t_6^2 |t_4t_5+t_6| \leq 1$ shows that $t_4$ runs over a set whose \text{measure} is $\ll |t_5|^{-1/2} t_6^{-1}$. Then, since we also have $|t_4t_5| \leq 1$, we can derive the bound $g_1(t_5,t_6,t;\boldsymbol{\eta},B) \ll \min \left( |t_5|^{-1/2} t_6^{-1}, |t_5|^{-1} \right) \leq |t_5|^{-2/3} t_6^{-2/3}$. The bound for $g_2$ \text{immediately} follows since $|t_5| \leq 1$.
\end{proof}

It is immediate to check that $\boldsymbol{\eta}$ is restricted to lie in the region
\begin{eqnarray}
\label{V}
\mathcal{V} & = & \left\{ \boldsymbol{\eta} \in \mathbb{Z}_{>0}^5,
Y_5 \geq \log(B)^A, Y_6 \geq 1, 2Y_4Y_5 \geq \eta_9^2 \right\} \textrm{.}
\end{eqnarray}
Assume that $\boldsymbol{\eta} \in \mathcal{V}$ and $\eta_6 \in \mathbb{Z}_{>0}$ are fixed and satisfy the height condition \eqref{condition2} and the coprimality conditions \eqref{gcd4}, \eqref{gcd5}, \eqref{gcd6} and \eqref{gcd7}.

Our next task is to estimate $C(\boldsymbol{\eta}',B)$. Recall the condition \eqref{new condition} which can be rewritten as
$|\eta_5''| \leq Y_4 Y_5''\log(B)^{-A}$. Let us sum over $\eta_4''$ using the basic estimate
$\# \{ n \in \mathbb{Z} , t_1 \leq n \leq t_2 \} = t_2 - t_1 + O(1)$. The change of variable $t_4 \mapsto Y_4'' t_4$ shows that
\begin{eqnarray*}
C(\boldsymbol{\eta}',B) & = & \sum_{\eta_5'' \leq Y_4 Y_5''\log(B)^{-A}} \left( Y_4'' g_1 \left( \frac{\eta_5''}{Y_5''}, \frac{\eta_6}{Y_6}, \kappa; \boldsymbol{\eta}, B \right) + O(1) \right) \textrm{.}
\end{eqnarray*}
The overall contribution of the error term is
\begin{eqnarray*}
\sum_{\boldsymbol{\eta}'} 2^{\omega(\eta_1\eta_2\eta_6\eta_7)} 2^{\omega(\eta_2\eta_9)}
\frac{B^{1/2} \log(B)^{-A}}{\boldsymbol{\eta}^{(0,1/2,1/2,0,1)}} & \ll & 
\sum_{\boldsymbol{\eta}} 2^{\omega(\eta_1 \eta_2 \eta_7)} 2^{\omega(\eta_2 \eta_9)}
\frac{B \log(B)^{1-A}}{\boldsymbol{\eta}^{(1,1,1,1,1)}} \\
& \ll & B \log(B)^{12-A} \textrm{,} 
\end{eqnarray*}
where we have summed over $\eta_6$ using \eqref{condition2}. Let us now sum over $\eta_5''$. Partial summation and the change of variable $t_5 \mapsto Y_5'' t_5$ yield
\begin{eqnarray*}
C(\boldsymbol{\eta}',B) & = & Y_4'' Y_5'' g_2 \left( \frac{\eta_6}{Y_6}, \kappa; \boldsymbol{\eta}, B \right) + O
\left( Y_4'' \sup_{|t_5| Y_5 \geq \log(B)^A} g_1 \left( t_5, \frac{\eta_6}{Y_6}, \kappa; \boldsymbol{\eta}, B \right) \right) \textrm{.}
\end{eqnarray*}
Since $h(t_4,t_5,t_6) \leq 1$ implies $|t_4t_5| \leq 1$, we have $g_1(t_5,t_6,t;\boldsymbol{\eta},B) \ll |t_5|^{-1}$ and thus
\begin{eqnarray*}
\sup_{|t_5| Y_5 \geq \log(B)^A} g_1 \left( t_5, \frac{\eta_6}{Y_6}, \kappa; \boldsymbol{\eta}, B \right) & \ll & Y_5 \log(B)^{-A} \textrm{.}
\end{eqnarray*}
Summing over $\eta_6$ using \eqref{condition2}, we see that the overall contribution of this error term is
\begin{eqnarray*}
\sum_{\boldsymbol{\eta}'} 2^{\omega(\eta_1\eta_3\eta_6\eta_7)} 2^{\omega(\eta_2\eta_9)}
\frac{B^{1/2} \log(B)^{-A}}{\boldsymbol{\eta}^{(0,1/2,1/2,0,1)}} & \ll &
\sum_{\boldsymbol{\eta}} 2^{\omega(\eta_1 \eta_3 \eta_7)} 2^{\omega(\eta_2 \eta_9)}
\frac{B \log(B)^{1-A}}{\boldsymbol{\eta}^{(1,1,1,1,1)}} \\
& \ll & B \log(B)^{11-A} \textrm{.} 
\end{eqnarray*}
Recalling lemma \ref{lemma inter}, for any fixed $A \geq 10$, we have obtained
\begin{eqnarray*}
N(\boldsymbol{\eta}',B) & = & \frac1{\eta_9} g_2 \left(\frac{\eta_6}{Y_6}, \kappa; \boldsymbol{\eta}, B \right) Y_4 Y_5
\sum_{\substack{k_8|\eta_2 \\ \gcd(k_8,\eta_7) = 1}} \frac{\mu(k_8)}{k_8\varphi^{\ast}(k_8\eta_9)}
\sum_{\substack{k_4|\eta_1\eta_2\eta_6\eta_7 \\ \gcd(k_4,k_8\eta_9)=1}} \frac{\mu(k_4)}{k_4} \\
& & \sum_{\substack{k_5|\eta_1\eta_3\eta_6\eta_7 \\ \gcd(k_5,k_8\eta_9)=1}} \frac{\mu(k_5)}{k_5}
\sum_{\ell_4|k_8\eta_9} \frac{\mu(\ell_4)}{\ell_4} \sum_{\ell_5|k_8\eta_9} \frac{\mu(\ell_5)}{\ell_5} + R_1(\boldsymbol{\eta}',B) \textrm{,}
\end{eqnarray*}
where $\sum_{\boldsymbol{\eta}'} R_1(\boldsymbol{\eta}',B) \ll B \log(B)^2$. A straightforward calculation reveals that the main term of $N(\boldsymbol{\eta}',B)$ is equal to
\begin{eqnarray*}
& & \theta(\boldsymbol{\eta}) \frac{\varphi^{\ast}(\eta_6)}{\varphi^{\ast}(\gcd(\eta_6,\eta_1\eta_2\eta_7))} \frac{\varphi^{\ast}(\eta_6)}{\varphi^{\ast}(\gcd(\eta_6,\eta_1\eta_3\eta_7))}
g_2 \left(\frac{\eta_6}{Y_6}, \kappa; \boldsymbol{\eta}, B \right) \frac{Y_4Y_5}{\eta_9} \textrm{,}
\end{eqnarray*}
where
\begin{eqnarray*}
\theta(\boldsymbol{\eta}) & = & \varphi^{\ast}(\eta_1\eta_2\eta_7) \varphi^{\ast}(\eta_1\eta_3\eta_7) \frac{\varphi^{\ast}(\eta_2\eta_9)}{\varphi^{\ast}(\gcd(\eta_2,\eta_7))} \textrm{.}
\end{eqnarray*}
For fixed $\boldsymbol{\eta} \in \mathcal{V}$ satisfying the coprimality conditions \eqref{gcd5}, \eqref{gcd6} and \eqref{gcd7}, let $\mathbf{N}(\boldsymbol{\eta},B)$ be the sum over $\eta_6$ of the main term of $N(\boldsymbol{\eta}',B)$, with $\eta_6$ satisfying the height \text{condition} \eqref{condition2} and the coprimality condition \eqref{gcd4}. Let us use lemma \ref{arithmetic preliminary} to sum over $\eta_6$. We find that for any fixed $A \geq 10$ and $0 < \sigma \leq 1$, we have
\begin{eqnarray}
\label{N}
\ \ \ \ \ \ \ \mathbf{N}(\boldsymbol{\eta},B) & = & \frac1{\eta_9} \mathcal{P} \Theta(\boldsymbol{\eta})
g_3 \left( \kappa; \boldsymbol{\eta}, B \right) Y_4 Y_5 Y_6 \\
\nonumber
& & + O \left( \frac{Y_4Y_5}{\eta_9} \varphi_{\sigma}(\eta_2\eta_7\eta_9) Y_6^{\sigma} 
\sup_{t_6 Y_6 \geq 1} g_2 \left( t_6, \kappa; \boldsymbol{\eta}, B \right) \right) \textrm{,}
\end{eqnarray}
where
\begin{eqnarray*}
\Theta(\boldsymbol{\eta}) & = & \theta(\boldsymbol{\eta}) \varphi^{\ast}(\eta_2\eta_7\eta_9) \varphi^{\dag}(\eta_3) \varphi'(\eta_1\eta_2\eta_3\eta_7\eta_9) \textrm{,}
\end{eqnarray*}
and where $\varphi^{\dag}$, $\varphi'$, $\varphi_{\sigma}$ and $\mathcal{P}$ are respectively introduced in \eqref{def dag}, \eqref{def '}, \eqref{varphi_sigma} and \eqref{def P}. Using the bound of lemma \ref{bounds} for $g_2$ and choosing $\sigma = 1/4$, we see that the overall contribution of the error term is
\begin{eqnarray*}
\sum_{\boldsymbol{\eta}} \varphi_{\sigma}(\eta_2\eta_7\eta_9) \frac{Y_4Y_5}{\eta_9} Y_6^{11/12} & \ll &
\sum_{\eta_2,\eta_3,\eta_7,\eta_9} \varphi_{\sigma}(\eta_2\eta_7\eta_9) \frac{B}{\boldsymbol{\eta}^{(0,1,1,1,1)}} \\
& \ll & B \log(B)^4 \textrm{,}
\end{eqnarray*}
since $\varphi_{\sigma}$ has average order $O(1)$ and where we have summed over $\eta_1$ using $Y_6 \geq 1$. Note that
\begin{eqnarray*}
\frac{Y_4 Y_5 Y_6}{\eta_9} & = & \frac{B}{\boldsymbol{\eta}^{(1,1,1,1,1)}} \textrm{.}
\end{eqnarray*}
The aim now is to remove the conditions $|t_4|Y_4, |t_5| Y_5 \geq \log(B)^A$ from the integral defining $g_3$ in the main term of $\mathbf{N}(\boldsymbol{\eta},B)$ in \eqref{N} and to replace $t_6Y_6 \geq 1$ by $t_6 > 0$. This will replace
$g_3 \left( \kappa; \boldsymbol{\eta}, B \right)$ by $g_4(\kappa)$ in the main term of $\mathbf{N}(\boldsymbol{\eta},B)$ in \eqref{N}. This is more subtle for $t_4$ than for $t_5$ and $t_6$. Indeed, since $Y_5  \geq \log(B)^A$ and $Y_6 \geq 1$, we can prove that the conditions $|t_5| < \log(B)^A/Y_5$ and $t_6 < 1 /Y_6$ in the integral both yield a negligible contribution. However, we do not have
$Y_4 \geq \log(B)^A$ so our reasoning consists in proving that the contribution corresponding to $Y_4 < \log(B)^A$ is negligible, which will allow us to assume that
$Y_4 \geq \log(B)^A$ and therefore conclude as for $t_5$ and $t_6$. For brevity, we set
\begin{eqnarray*}
D_h & = & \left\{ (t_4,t_5,t_6) \in \mathbb{R}^3, t_6 > 0, h(t_4,t_5,t_6) \leq 1 \right\} \textrm{.}
\end{eqnarray*}

\begin{lemma}
For $Z_4, Z_5, Z_6 > 0$, we have
\begin{eqnarray}
\label{1.}
\meas \{ (t_4,t_5,t_6) \in D_h, |t_4| Z_4 \geq 1 \} & \ll & Z_4^{1/4} \textrm{,} \\
\label{4.}
\meas \{ (t_4,t_5,t_6) \in D_h, |t_4| Z_4 < 1 \} & \ll & Z_4^{-1} \textrm{,} \\
\label{5.}
\meas \{ (t_4,t_5,t_6) \in D_h, |t_5| Z_5 < 1 \} & \ll & Z_5^{-1/3} \textrm{,} \\
\label{6.}
\meas \{ (t_4,t_5,t_6) \in D_h, t_6 Z_6 < 1 \} & \ll & Z_6^{-1/3} \textrm{.}
\end{eqnarray}
\end{lemma}

\begin{proof}
The conditions $|t_4|t_6^2 |t_4t_5+t_6| \leq 1$ and $|t_4t_5| \leq 1$ show that $t_5$ runs over a set whose measure is
$\ll \min \left( t_4^{-2}t_6^{-2}, |t_4|^{-1} \right) \leq |t_4|^{-5/4}t_6^{-1/2}$, which proves the bound \eqref{1.} since
$t_6 \leq 1$. The bound \eqref{4.} is clear since $|t_5|, t_6 \leq 1$. The bound \eqref{5.} follows from the bound of lemma \ref{bounds} for $g_1$ and $t_6 \leq 1$. In a similar way, \eqref{6.} is a consequence of the bound of lemma \ref{bounds} for $g_1$ and $|t_5| \leq 1$.
\end{proof}

Using the bound \eqref{5.}, we see that removing the condition $|t_5| Y_5 \geq \log(B)^A$ from the integral defining $g_3$ in the main term of $\mathbf{N}(\boldsymbol{\eta},B)$ in \eqref{N} yields an error term whose overall contribution is
\begin{eqnarray*}
\sum_{\boldsymbol{\eta}} Y_4Y_5^{2/3}Y_6\log(B)^{A/3} & \ll &
\sum_{\eta_2,\eta_3,\eta_7,\eta_9} \frac{B}{\boldsymbol{\eta}^{(0,1,1,1,1)}} \\
& \ll & B \log(B)^4\textrm{,}
\end{eqnarray*}
where we have summed over $\eta_1$ using $Y_5 \geq \log(B)^A$. In a similar fashion, the bound \eqref{6.} shows that replacing the condition $t_6Y_6 \geq 1$ by $t_6 > 0$ in the integral defining $g_3$ in the main term of $\mathbf{N}(\boldsymbol{\eta},B)$ in \eqref{N} also creates an error term whose overall contribution is $\ll B \log(B)^4$.

We now assume that $Y_4 < \log(B)^A$ and we bound the contribution of the main term of $\mathbf{N}(\boldsymbol{\eta},B)$ under this assumption. The bound \eqref{1.} shows that this contribution is
\begin{eqnarray*}
\sum_{\boldsymbol{\eta}} Y_4^{5/4}Y_5Y_6 \log(B)^{-A/4} & \ll & \sum_{\eta_2,\eta_3,\eta_7,\eta_9} \frac{B}{\boldsymbol{\eta}^{(0,1,1,1,1)}} \\
& \ll & B \log(B)^4 \textrm{.}
\end{eqnarray*}
We can therefore assume from now on that
\begin{eqnarray}
\label{condition V1}
Y_4 & \geq & \log(B)^A \textrm{.}
\end{eqnarray}
Under this assumption, exactly as for $t_5$ and $t_6$, the bound \eqref{4.} shows that the overall contribution of the error term created by removing the condition $|t_4| Y_4 \geq \log(B)^A$ from the integral defining $g_3$ in the main term of $\mathbf{N}(\boldsymbol{\eta},B)$ in \eqref{N} is $\ll B \log(B)^4$. We have proved that for any fixed $A \geq 9$,
\begin{eqnarray}
\label{estimate Ng4}
\mathbf{N}(\boldsymbol{\eta},B) & = & \mathcal{P} g_4(\kappa) \frac{B}{\boldsymbol{\eta}^{(1,1,1,1,1)}} \Theta(\boldsymbol{\eta}) + R_2(\boldsymbol{\eta},B) \textrm{,}
\end{eqnarray}
where $\sum_{\boldsymbol{\eta}} R_2(\boldsymbol{\eta},B) \ll B \log(B)^4$. The goal of the following lemma is to replace $g_4(\kappa)$  by $g_4(0)$ in the main term of $\mathbf{N}(\boldsymbol{\eta},B)$ in \eqref{estimate Ng4}. By \eqref{omega_infty}, $g_4(0)$ is equal to
\begin{eqnarray*}
\int \int \int_{t_6 > 0, h(t_4,t_5,t_6) \leq 1} \D t_4 \D t_5 \D t_6 & = & \frac{\omega_{\infty}}{2} \textrm{.}
\end{eqnarray*}

\begin{lemma}
For $t > 0$, we have
\begin{eqnarray}
\label{condition t}
\meas \{ (t_4,t_5,t_6) \in D_h, |t_4t_5 + t_6| < t \} & \ll & t^{1/2} \textrm{.}
\end{eqnarray} 
\end{lemma}

\begin{proof}
The conditions $|t_4|t_6^2 |t_4t_5+t_6| \leq 1$ and $|t_4t_5 + t_6| < t$ imply that $t_4$ runs over a set whose measure is
$\ll \min \left( |t_5|^{-1/2}t_6^{-1}, t |t_5|^{-1} \right) \leq t^{1/2}|t_5|^{-3/4}t_6^{-1/2}$, which suffices since
$|t_5|,t_6 \leq 1$.
\end{proof}

Let us estimate the overall contribution of the error term which appears if we replace $g_4(\kappa)$ by $g_4(0)$ in the main term of $\mathbf{N}(\boldsymbol{\eta},B)$ in \eqref{estimate Ng4}. Using \eqref{condition t} and summing over $\eta_9$ using the condition \eqref{conditionkappa}, we find that this contribution is
\begin{eqnarray*}
\sum_{\boldsymbol{\eta}} \frac{B}{\boldsymbol{\eta}^{(1,1,1,1,1)}} \kappa^{1/2} & \ll &
\sum_{\eta_1,\eta_2,\eta_3,\eta_7} \frac{B}{\boldsymbol{\eta}^{(1,1,1,1,0)}} \\
& \ll & B \log(B)^4 \textrm{.}
\end{eqnarray*}
We have therefore obtained the following result.

\begin{lemma}
\label{lemmafin}
For any fixed $A \geq 10$, we have the estimate
\begin{eqnarray*}
\mathbf{N}(\boldsymbol{\eta},B) & = & \mathcal{P} \frac{\omega_{\infty}}{2} \frac{B}{\boldsymbol{\eta}^{(1,1,1,1,1)}} \Theta(\boldsymbol{\eta}) + R_3(\boldsymbol{\eta},B) \textrm{,}
\end{eqnarray*}
where $\sum_{\boldsymbol{\eta}} R_3(\boldsymbol{\eta},B) \ll B \log(B)^4$.
\end{lemma}

\subsection{Conclusion}

\label{Conclusion}

Recall the definition \eqref{V} of $\mathcal{V}$. It remains to sum the main term of $\mathbf{N}(\boldsymbol{\eta},B)$ over the $\boldsymbol{\eta} \in \mathcal{V}$ satisfying \eqref{condition V1} and the coprimality conditions  \eqref{gcd5}, \eqref{gcd6} and \eqref{gcd7}. It is easy to see that replacing $\left\{ \boldsymbol{\eta} \in \mathcal{V}, \eqref{condition V1} \right\}$ by the region 
\begin{eqnarray*}
\mathcal{V}' & = &  \left\{ \boldsymbol{\eta} \in \mathbb{Z}_{>0}^5, Y_4 \geq 1, Y_5 \geq 1, Y_6 \geq 1, Y_4 Y_5 \geq \eta_9^2 \right\} \textrm{,}
\end{eqnarray*}
produces an error term whose overall contribution is $\ll B \log(B)^4 \log(\log(B))$. Let us redefine $\Theta$ as being equal to zero if the remaining coprimality conditions \eqref{gcd5}, \eqref{gcd6} and \eqref{gcd7} are not satisfied. Fixing for example $A = 10$ and combining lemmas \ref{lemmaA} and \ref{lemmafin}, we obtain
\begin{eqnarray*}
N_{U_1,H}(B) & = & \mathcal{P} \omega_{\infty} B \sum_{\boldsymbol{\eta} \in \mathcal{V}'} \frac{\Theta(\boldsymbol{\eta})}{\boldsymbol{\eta}^{(1,1,1,1,1)}} + O \left( B \log(B)^4 \log(\log(B)) \right) \textrm{.}
\end{eqnarray*}
Set $\mathbf{k} = (k_1,k_2,k_3,k_7,k_9)$ and define, for $s \in \mathbb{C}$ such that $\Re(s) > 1$,
\begin{eqnarray*}
F(s) & = & \sum_{\boldsymbol{\eta} \in \mathbb{Z}_{>0}^5}
\frac{\left|(\Theta \ast \boldsymbol{\mu})(\boldsymbol{\eta})\right|}{\eta_1^s \eta_2^s \eta_3^s \eta_7^s \eta_9^s} \\
& = & \prod_p \left( \sum_{\mathbf{k} \in \mathbb{Z}_{\geq 0}^5}
\frac{\left|(\Theta \ast \boldsymbol{\mu}) \left( p^{k_1},p^{k_2},p^{k_3},p^{k_7},p^{k_9} \right)\right|}
{p^{k_1 s}p^{k_2 s}p^{k_3 s}p^{k_7 s}p^{k_9 s}} \right) \textrm{.}
\end{eqnarray*}
If $\mathbf{k} \notin \{0,1\}^5$ then $(\Theta \ast \boldsymbol{\mu}) \left( p^{k_1},p^{k_2},p^{k_3},p^{k_7},p^{k_9} \right) = 0$ and furthermore if only one of the $k_i$ is equal to $1$, then
$(\Theta \ast \boldsymbol{\mu}) \left( p^{k_1},p^{k_2},p^{k_3},p^{k_7},p^{k_9} \right) \ll 1/p$, so the local factors $F_p$ of $F$ satisfy
\begin{eqnarray*}
F_p(s) & = & 1 + O \left( \frac1{p^{ \min \left( \Re(s)+1, 2 \Re(s) \right)}} \right) \textrm{,}
\end{eqnarray*}
and thus $F$ actually converges in the half-plane $\Re(s) > 1/2$. This proves that $\Theta$ satifies the assumption \eqref{Psi} of lemma
\ref{final sum}. We therefore get
\begin{eqnarray*}
N_{U_1,H}(B) & = & \mathcal{P} \omega_{\infty} \alpha \left( \sum_{\boldsymbol{\eta} \in \mathbb{Z}_{>0}^5}
\frac{(\Theta \ast \boldsymbol{\mu})(\boldsymbol{\eta})}{\boldsymbol{\eta}^{(1,1,1,1,1)}} \right) B \log(B)^5
+ O \left( B \log(B)^4 \log(\log(B))  \right) \textrm{,}
\end{eqnarray*}
where $\alpha$ is the volume of the polytope defined in $\mathbb{R}^5$ by $t_1,t_2,t_3,t_7,t_9 \geq 0$ and
\begin{eqnarray*}
2 t_1 + 3 t_2 - t_3 + 4 t_7 + 2 t_9 & \geq & 1 \textrm{,} \\
t_1 + 2 t_2 + 2 t_7 + t_9 & \leq & 1 \textrm{,} \\
2 t_1 + t_2 + t_3 + 2 t_7 & \leq & 1 \textrm{,} \\
t_2 + t_3 + 4 t_9 & \leq & 1 \textrm{.}
\end{eqnarray*}
A computation using Franz's additional \textit{Maple} package \cite{Convex} gives $\alpha = 1/1440$, that is to say
\begin{eqnarray*}
\alpha & = & \alpha(\widetilde{V_1}) \textrm{,}
\end{eqnarray*}
and moreover
\begin{eqnarray*}
\sum_{\boldsymbol{\eta} \in \mathbb{Z}_{>0}^5}
\frac{(\Theta \ast \boldsymbol{\mu}) (\boldsymbol{\eta})}{\boldsymbol{\eta}^{(1,1,1,1,1)}} & = & \prod_p
\left( \sum_{\mathbf{k} \in \mathbb{Z}_{\geq 0}^5} 
\frac{(\Theta \ast \boldsymbol{\mu}) \left( p^{k_1},p^{k_2},p^{k_3},p^{k_7},p^{k_9} \right)}
{p^{k_1}p^{k_2}p^{k_3}p^{k_7}p^{k_9}} \right) \\
& = & \prod_p \left( 1 - \frac1{p} \right)^5 
\left( \sum_{\mathbf{k} \in \mathbb{Z}_{\geq 0}^5} 
\frac{\Theta \left( p^{k_1},p^{k_2},p^{k_3},p^{k_7},p^{k_9} \right)}
{p^{k_1}p^{k_2}p^{k_3}p^{k_7}p^{k_9}} \right) \textrm{.}
\end{eqnarray*}
We omit the details of the calculation of the sum of the series of the right-hand side, let us just say that the remaining coprimality conditions greatly simplify the calculation. We obtain
\begin{eqnarray*}
\sum_{\mathbf{k} \in \mathbb{Z}_{\geq 0}^5}
\frac{\Theta \left( p^{k_1},p^{k_2},p^{k_3},p^{k_7},p^{k_9} \right)} {p^{k_1}p^{k_2}p^{k_3}p^{k_7}p^{k_9}} & = & \varphi'(p)
\left( 1 - \frac1{p} \right) \left( 1 + \frac{6}{p} + \frac1{p^2} \right) \textrm{,}
\end{eqnarray*}
and thus
\begin{eqnarray*}
\sum_{\boldsymbol{\eta} \in \mathbb{Z}_{>0}^5}
\frac{(\Theta \ast \boldsymbol{\mu}) (\boldsymbol{\eta})}{\boldsymbol{\eta}^{(1,1,1,1,1)}} & = &
\mathcal{P}^{-1} \prod_p \left( 1 - \frac1{p} \right)^6 \omega_p \textrm{,}
\end{eqnarray*}
which completes the proof.

\section{Proof for the $\mathbf{A}_1 + \mathbf{A}_2$ surface}

\subsection{The universal torsor}

\label{torsor section}

We now proceed to define a bijection between the set of points to be counted on $U_2$ and a certain set of integral points on the affine variety defined by \eqref{torsor 2}. Our choice of notation might be surprising but our aim is simply to highlight the similarities with the case of the $3 \mathbf{A}_1$ surface. Note that for a given $(x_0:x_1:x_2:x_3:x_4) \in V_2$, we have $(x_0:x_1:x_2:x_3:x_4) \in U_2$ if and only if $x_0x_1x_2x_3x_4 \neq 0$. Let $(x_0,x_1,x_2,x_3,x_4) \in \mathbb{Z}_{\neq 0}^5$ be such that $\gcd(x_0,x_1,x_2,x_3,x_4) = 1$ and
\begin{eqnarray*}
x_0 x_1 - x_2 x_3 & = & 0 \textrm{,} \\
x_1 x_2 + x_2 x_4 + x_3 x_4 & = & 0 \textrm{,}
\end{eqnarray*}
and $\max \{ |x_i|, 0 \leq i \leq 4 \} \leq B$. Define $\xi_6 = \gcd(x_0, x_1, x_2, x_3) > 0$ and write $x_i = \xi_6 x_i'$ for
$i = 0, 1, 2, 3$. We thus have $\gcd(\xi_6,x_4) = 1$ and $\gcd(x_0', x_1', x_2', x_3') = 1$. Now let
$\xi_3 = \gcd(x_0', x_2', x_3') > 0$. Since $\gcd(\xi_3, x_1') = 1$, it follows that $\xi_3^2|x_0'$ and we can write
$x_j' = \xi_3 x_j''$ for $j = 2, 3$ and $x_0' = \xi_3^2 x_0''$. Moreover, we have $\gcd(\xi_3 x_0'', x_2'', x_3'') = 1$. Let
$\xi_8 = \gcd(x_0'', x_3'') > 0$ and write $x_0'' = \xi_8 \xi_4$ and $x_3'' = \xi_8 y_3$ with $\gcd(\xi_4,y_3) = 1$. The first equation can be rewritten as $\xi_4 x_1' = x_2'' y_3$. Since $\gcd(\xi_4,y_3) = 1$, we have $\xi_4|x_2''$ and we can write
$x_2'' = \xi_4 y_2$ and thus $x_1' = y_2 y_3$. Let us sum up what we have done until now. We have been able to find
$(\xi_6,\xi_3,\xi_8,\xi_4,y_3,y_2) \in \mathbb{Z}_{>0}^3 \times \mathbb{Z}_{\neq 0}^3$ such that
$\gcd(\xi_6,x_4) = 1$, $\gcd(\xi_3, y_2 y_3) = 1$, $\gcd(\xi_8, \xi_4y_2) = 1$, $\gcd(\xi_4,y_3) = 1$ and
$x_0 = \xi_6 \xi_3^2 \xi_8 \xi_4$, $x_1 = \xi_6 y_2 y_3$, $x_2 = \xi_6 \xi_3 \xi_4 y_2$, $x_3 = \xi_6 \xi_3 \xi_8 y_3$. Simplifying by $\xi_3 \xi_6$, the second equation gives
\begin{eqnarray*}
\xi_6 y_2^2 y_3 \xi_4 + x_4 ( \xi_4 y_2 + \xi_8 y_3 ) & = & 0 \textrm{.}
\end{eqnarray*}
Let $y_{23} = \gcd(y_2,y_3) > 0$ and write $y_2 = y_{23} \xi_5$, $y_3 = y_{23} \xi_9$ with $\gcd(\xi_5, \xi_9) = 1$. We obtain
\begin{eqnarray*}
\xi_6 y_{23}^2 \xi_5^2 \xi_9 \xi_4 + x_4 ( \xi_4 \xi_5 + \xi_8 \xi_9 ) & = & 0 \textrm{.}
\end{eqnarray*}
Since $\gcd(\xi_4\xi_5,\xi_8\xi_9) = 1$, it is obvious that $\xi_4\xi_5^2\xi_9|x_4$. Writing $x_4 = \xi_4 \xi_5^2 \xi_9 x_4'$, the equation becomes
\begin{eqnarray*}
\xi_6 y_{23}^2 + x_4' ( \xi_4 \xi_5 + \xi_8 \xi_9 ) & = & 0 \textrm{.}
\end{eqnarray*}
We now see that since $\gcd(\xi_6,x_4') = 1$, we have $x_4'|y_{23}^2$ and thus there is a unique way to write
$y_{23} =\xi_1 \xi_2 \xi_7$ and $x_4' = \xi_2^2 \xi_7$ with $\gcd(\xi_1,\xi_2) = 1$ and $\xi_2 > 0$, $\xi_1 \xi_7 > 0$. This leads to
\begin{eqnarray}
\label{tor 2}
\xi_4\xi_5 + \xi_1^2\xi_6\xi_7 + \xi_8\xi_9 & = & 0 \textrm{,}
\end{eqnarray}
and we have finally found $x_0 = \xi_3^2 \xi_4 \xi_6 \xi_8$, $x_1 = \xi_1^2 \xi_2^2 \xi_5 \xi_6 \xi_7^2 \xi_9$,
$x_2 = \xi_1 \xi_2 \xi_3 \xi_4 \xi_5 \xi_6 \xi_7$, $x_3 = \xi_1 \xi_2 \xi_3 \xi_6 \xi_7 \xi_8 \xi_9$, $x_4 = \xi_2^2 \xi_4 \xi_5^2 \xi_7 \xi_9$ and a little thought reveals that, given the equation \eqref{tor 2}, the coprimality conditions can be rewritten as
\begin{eqnarray*}
& & \gcd(\xi_4\xi_5,\xi_1\xi_6\xi_7) = 1 \textrm{,} \\
& & \gcd(\xi_4\xi_5,\xi_8\xi_9) = 1 \textrm{,} \\
& & \gcd(\xi_1\xi_6\xi_7,\xi_8\xi_9) = 1 \textrm{,} \\
& & \gcd(\xi_2,\xi_1\xi_3\xi_4\xi_6\xi_8) = 1 \textrm{,} \\
& & \gcd(\xi_3,\xi_1\xi_5\xi_7\xi_9) = 1 \textrm{,} \\
& & \gcd(\xi_6,\xi_7) = 1 \textrm{.}
\end{eqnarray*}
Since $\xi_1 \mapsto - \xi_1$ is a bijection on the set of solutions, we can assume that $\xi_1 > 0$ and thus $\xi_7 > 0$, keeping in mind that we have to divide our result by $2$. In a similar fashion, $(\xi_8,\xi_9) \mapsto (-\xi_8,-\xi_9)$ shows that we can remove the condition $\xi_8 > 0$ multiplying our result by $2$. To sum up, let $\mathcal{T}_2(B)$ be the number of
$(\xi_1,\xi_2,\xi_3,\xi_4,\xi_5,\xi_6,\xi_7,\xi_8,\xi_9)~\in~\mathbb{Z}_{\neq 0}^9$ such that
$\xi_1, \xi_2,\xi_3,\xi_6,\xi_7 > 0$ and satisfying the equation \eqref{tor 2}, the coprimality conditions above and the height conditions
\begin{eqnarray}
\label{condition1'}
\xi_3^2 |\xi_4| \xi_6 |\xi_8| & \leq & B \textrm{,} \\
\label{condition2'}
\xi_1^2 \xi_2^2 |\xi_5| \xi_6 \xi_7^2 |\xi_9| & \leq & B \textrm{,} \\
\label{condition3'}
\xi_1 \xi_2 \xi_3 |\xi_4 \xi_5| \xi_6 \xi_7 & \leq & B \textrm{,} \\
\label{condition4'}
\xi_1 \xi_2 \xi_3 \xi_6 \xi_7 |\xi_8 \xi_9| & \leq & B \textrm{,} \\
\label{condition5'}
\xi_2^2 |\xi_4| \xi_5^2 \xi_7 |\xi_9| & \leq & B \textrm{.}
\end{eqnarray}
Since we have not taken into account that $\mathbf{x} = - \mathbf{x} \in \mathbb{P}^4$ yet, we have finally proved the following lemma.

\begin{lemma}
\label{T2}
We have the equality
\begin{eqnarray*}
N_{U_2,H}(B) & = & \frac1{2} \# \mathcal{T}_2(B) \textrm{.}
\end{eqnarray*}
\end{lemma}

\subsection{Calculation of Peyre's constant}

We have
\begin{eqnarray*}
c_{V_2,H} & = & \alpha(\widetilde{V_2}) \beta(\widetilde{V_2}) \omega_H(\widetilde{V_2}) \textrm{,}
\end{eqnarray*}
where (see \cite{MR2318651})
\begin{eqnarray*}
\alpha(\widetilde{V_2}) & = & \frac1{2160} \textrm{,}
\end{eqnarray*}
$\beta(\widetilde{V_2}) = 1$ since $V_2$ is split over $\mathbb{Q}$ and
\begin{eqnarray*}
\omega_H(\widetilde{V_2}) & = & \tau_{\infty} \prod_p \left( 1 - \frac1{p} \right)^6 \tau_p \textrm{,}
\end{eqnarray*}
where, thanks to \cite[Lemma 2.3]{Loughran}, we have
\begin{eqnarray*}
\tau_p & = & 1 + \frac{6}{p} + \frac1{p^2} \textrm{.}
\end{eqnarray*}
Let us calculate $\tau_{\infty}$. Set $f_1(x) = x_0 x_1 - x_2 x_3$ and $f_2(x) = x_1 x_2 + x_2 x_4 + x_3 x_4$. Let us parametrize the points of $V_2$ by $x_0$, $x_2$ and $x_3$. We have
\begin{eqnarray*}
\det \begin{pmatrix}
\frac{\partial f_1}{\partial x_1} & \frac{\partial f_1}{\partial x_4} \\
\frac{\partial f_2}{\partial x_1} & \frac{\partial f_2}{\partial x_4}
\end{pmatrix} & = &
\begin{vmatrix}
x_0 & 0 \\
x_2 & x_2 + x_3
\end{vmatrix} \\
& = & x_0 (x_2 + x_3) \textrm{.}
\end{eqnarray*}
Moreover, $x_1 = x_2 x_3 / x_0$ and $x_4 = - x_2^2 x_3 / \left( x_0 \left( x_2 + x_3 \right) \right)$.
Since $\mathbf{x} = - \mathbf{x}$ in $\mathbb{P}^4$, we have
\begin{eqnarray*}
\tau_{\infty} & = & - 2 \int \int \int_{x_0 > 0, x_2 + x_3 < 0, x_0, \left| x_2 x_3 / x_0 \right|, |x_2|, |x_3|,
|x_2^2 x_3 / \left( x_0 \left( x_2 + x_3 \right) \right)| \leq 1}
\frac{\D x_0 \D x_2 \D x_3}{x_0 (x_2 + x_3)} \textrm{.}
\end{eqnarray*}
We introduce the functions
\begin{eqnarray}
\label{h^{a}}
h^{a} & : & (t_4,t_5,t_1) \mapsto \max \left\{
\begin{array}{l}
|t_4||t_4t_5+t_1^2|,t_1^2|t_5|,t_1|t_4t_5| \\
t_1|t_4t_5+t_1^2|, |t_4|t_5^2
\end{array}
\right\} \textrm{,} \\
\label{h^{b}}
h^{b} & : & (t_4,t_5,t_1) \mapsto \max \left\{
\begin{array}{l}
|t_4|,t_1^2|t_5||t_4t_5+t_1^2|,t_1|t_4t_5| \\
t_1|t_4t_5+t_1^2|, |t_4|t_5^2 |t_4t_5+t_1^2|
\end{array}
\right\} \textrm{.}
\end{eqnarray}
The change of variables given by $x_0 = t_4 (t_4t_5+t_1^2)$, $x_2 = t_1 t_4 t_5$ and $x_3 = - t_1 (t_4 t_5 + t_1^2)$ yields
\begin{eqnarray}
\nonumber
\tau_{\infty} & = & 6 \int \int \int_{t_4(t_4t_5+t_1^2) > 0, t_1 >0, h^{a}(t_4,t_5,t_1) \leq 1} \D t_4 \D t_5 \D t_1 \\
\label{tau 1}
& = & 3 \int \int \int_{t_1 >0, h^{a}(t_4,t_5,t_1) \leq 1} \D t_4 \D t_5 \D t_1 \textrm{.}
\end{eqnarray}
Moreover, the change of variables $x_0 = t_4$, $x_2 = t_1 t_4 t_5$ and $x_3 = - t_1 (t_4 t_5 + t_1^2)$ gives the alternative expression
\begin{eqnarray}
\nonumber
\tau_{\infty} & = & 6 \int \int \int_{t_4 > 0, t_1 > 0, h^{b}(t_4,t_5,t_1) \leq 1} \D t_4 \D t_5 \D t_1 \\
\label{tau 2}
& = & 3 \int \int \int_{t_1 > 0, h^{b}(t_4,t_5,t_1) \leq 1} \D t_4 \D t_5 \D t_1 \textrm{.}
\end{eqnarray}

Let us repeat here that the following proof is very similar to the one before, so we will sometimes allow ourselves to be concise.

\subsection{Restriction of the domain}

The following two lemmas are the analogues of lemmas \ref{lemmalog} and \ref{lemmaHB} respectively.

\begin{lemma}
\label{lemmalog2}
Let $\mathcal{M}_{2}(B)$ be the overall contribution to $N_{U_2,H}(B)$ coming from the
$(\xi_1, \dots ,\xi_9) \in \mathcal{T}_2(B)$ such that $|\xi_i| \leq \log(B)^A$ for a certain $i \neq 1, 2, 3$, where $A > 0$ is any fixed constant. We have
\begin{eqnarray*}
\mathcal{M}_{2}(B) & \ll_A & B \log(B)^4 \log(\log(B)) \textrm{.}
\end{eqnarray*}
\end{lemma}

\begin{lemma}
\label{lemmaHB2}
Let $L_1, L_4, \dots, L_9 \geq 1/2$ and define $M_2 = M_2(L_1, L_4, \dots, L_9)$ as the number of
$(n_1, n_4, \dots, n_9) \in \mathbb{Z}^7$ such that $L_i < |n_i| \leq 2L_i$ for $i=1$ and
$4 \leq i \leq 9$, $\gcd(n_4n_5,n_1n_6n_7) = 1$ and
\begin{eqnarray}
\label{equation lemma2}
n_4n_5 + n_1^2n_6n_7 + n_8n_9 & = & 0 \textrm{.}
\end{eqnarray}
We have
\begin{eqnarray*}
M_2 & \ll & L_1 (L_4L_5L_6L_7L_8L_9)^{1/2} + L_1 L_6 L_7 \min(L_4 L_5, L_8 L_9) \textrm{.}
\end{eqnarray*}
\end{lemma}

\begin{proof}
We can assume by symmetry that $L_4 L_5 \leq L_8 L_9$. Let us first deal with the case where $L_1^2L_6L_7 \leq L_4 L_5$. The equation \eqref{equation lemma2} gives $L_8 L_9 \ll L_4 L_5$. Let $M_2'$ be the number of $(n_1, n_4, \dots, n_9) \in \mathbb{Z}^7$ to be counted in this case. The first case of the proof of lemma \ref{lemmaHB} shows that
\begin{eqnarray*}
M_2' & \ll & L_1 L_6 L_7 L_4 L_5 \textrm{.}
\end{eqnarray*}
In the other case where $L_1^2L_6L_7 > L_4 L_5$, the equation \eqref{equation lemma2} gives $L_8 L_9 \ll L_1^2 L_6 L_7$. Let $M_2''$ be the number of $(n_1, n_4, \dots, n_9) \in \mathbb{Z}^7$ to be counted here. Assume by symmetry that $L_4 \leq L_5$,
$L_6 \leq L_7$ and $L_8 \leq L_9$. Since $\gcd(n_4,n_1n_6,n_8) = 1$, using \cite[Lemma $6$]{MR2075628}, we can deduce that
\begin{eqnarray*}
\# \left\{
(n_5,n_7,n_9) \in \mathbb{Z}^3,
\begin{array}{l}
L_i < |n_i| \leq 2 L_i, i \in \{5, 7, 9 \} \\
\gcd(n_5,n_7,n_9) = 1 \\
n_4n_5 + n_1^2n_6n_7 + n_8n_9 = 0
\end{array} \right\}
& \ll & 1 + \frac{L_5 L_9}{n_1^2 n_6} \textrm{.}
\end{eqnarray*}
Summing over $n_1$, $n_4$, $n_6$ and $n_8$, we get
\begin{eqnarray*}
M_2'' & \ll & L_1 L_4 L_6 L_8 + \frac{L_4 L_5 L_8 L_9}{L_1} \\
& \ll & L_1 (L_4L_5L_6L_7L_8L_9)^{1/2} + L_4 L_5 L_1 L_6 L_7 \textrm{,}
\end{eqnarray*}
which ends the proof.
\end{proof}

Let us now prove lemma \ref{lemmalog2}.

\begin{proof}
Let $Z_i \geq 1/2$ for $i = 1, \dots, 9 $ and define $\mathcal{N}_2 = \mathcal{N}_2(Z_1, \dots, Z_9)$ as the \text{contribution} of the $(\xi_1, \dots, \xi_9) \in \mathcal{T}_2(B)$ satisfying $Z_i < |\xi_i| \leq 2 Z_i$ for $i = 1, \dots, 9 $. The height conditions imply that either $\mathcal{N}_2 = 0$ or we have the inequalities
\begin{eqnarray}
\label{condition1'bis}
Z_3^2 Z_4 Z_6 Z_8 & \leq & B \textrm{,} \\
\label{condition2'bis}
Z_1^2 Z_2^2 Z_5 Z_6 Z_7^2 Z_9 & \leq & B \textrm{,} \\
\label{condition3'bis}
Z_1 Z_2 Z_3 Z_4 Z_5 Z_6 Z_7 & \leq & B \textrm{,} \\
\label{condition4'bis}
Z_1 Z_2 Z_3 Z_6 Z_7 Z_8 Z_9 & \leq & B \textrm{,} \\
\label{condition5'bis}
Z_2^2 Z_4 Z_5^2 Z_7 Z_9 & \leq & B \textrm{.}
\end{eqnarray}
Using lemma \ref{lemmaHB2} and summing over $\xi_2$ and $\xi_3$, we get
\begin{eqnarray*}
\mathcal{N}_2 & \ll & Z_1 Z_2 Z_3 (Z_4Z_5Z_6Z_7Z_8Z_9)^{1/2} + Z_1 Z_2 Z_3 Z_6 Z_7 \min(Z_4 Z_5, Z_8 Z_9) \textrm{.}
\end{eqnarray*}
Assume that $Z_4 Z_5 \leq Z_8 Z_9$ (the case where $Z_4 Z_5 > Z_8 Z_9$ is identical). Note that the torsor equation \eqref{tor 2} then gives $Z_1^2 Z_6 Z_7 \ll Z_8 Z_9$. Denote by $\mathcal{N}_{2}'$ the first term of the right-hand side and by $\mathcal{N}_{2}''$ the second term. We proceed to prove that
\begin{eqnarray*}
\sum_{Z_i} \mathcal{N}_{2}' & \ll & B \log(B)^4 \textrm{.}
\end{eqnarray*}
Let us first assume that $Z_1^2 Z_6 Z_7 \leq Z_4 Z_5$. Summing over
\begin{eqnarray}
\nonumber
Z_1 & \leq & \min \left( \frac{Z_4^{1/2} Z_5^{1/2}}{Z_6^{1/2} Z_7^{1/2}}, \frac{B^{1/2}}{Z_2 Z_5^{1/2} Z_6^{1/2} Z_7 Z_9^{1/2}} \right) \\
\label{equationZ_1}
& \leq & \frac{Z_4^{1/4} B^{1/4}}{Z_2^{1/2} Z_6^{1/2} Z_7^{3/4} Z_9^{1/4}} \textrm{,} 
\end{eqnarray}
we get in this case
\begin{eqnarray*}
\sum_{Z_i} \mathcal{N}_{2}'
& \ll & B^{1/4} \sum_{\widehat{Z_1}} Z_2^{1/2} Z_3 Z_4^{3/4} Z_5^{1/2} Z_7^{-1/4} Z_8^{1/2} Z_9^{1/4} \\
& \ll & B^{1/2} \sum_{\widehat{Z_1},\widehat{Z_2}} Z_3 Z_4^{1/2} Z_7^{-1/2} Z_8^{1/2} \\
& \ll & B \sum_{\widehat{Z_1},\widehat{Z_2},\widehat{Z_3}} Z_6^{-1/2} Z_7^{-1/2} \\
& \ll & B \log(B)^4 \textrm{,}
\end{eqnarray*}
where we have summed over $Z_2$ and $Z_3$ using \eqref{condition5'bis} and \eqref{condition1'bis}. Let us treat the case where
$Z_1^2 Z_6 Z_7 > Z_4 Z_5$. Summing over $Z_2$ using \eqref{condition4'bis}, we obtain
\begin{eqnarray*}
\sum_{Z_i} \mathcal{N}_{2}' & \ll & B \sum_{\widehat{Z_2}} Z_4^{1/2} Z_5^{1/2} Z_6^{-1/2} Z_7^{-1/2} Z_8^{-1/2} Z_9^{-1/2}\\
& \ll & B \sum_{\widehat{Z_2},\widehat{Z_4}} Z_1 Z_8^{-1/2} Z_9^{-1/2} \\
& \ll & B \sum_{\widehat{Z_1},\widehat{Z_2},\widehat{Z_4}} Z_6^{-1/2} Z_7^{-1/2} \\
& \ll & B \log(B)^4 \textrm{,}
\end{eqnarray*}
where we have summed over $Z_4 < Z_1^2 Z_5^{-1} Z_6 Z_7$ and over $Z_1 \ll Z_6^{-1/2} Z_7^{-1/2} Z_8^{1/2} Z_9^{1/2}$.
Let us estimate the contribution of $\mathcal{N}_{2}''$ in the case where $Z_1^2 Z_6 Z_7 \leq Z_4 Z_5$. Summing over $Z_1$ using \eqref{equationZ_1}, we get
\begin{eqnarray*}
\sum_{Z_i} \mathcal{N}_{2}''
& \ll & B^{1/4} \sum_{\widehat{Z_1}} Z_2^{1/2} Z_3 Z_4^{5/4} Z_5 Z_6^{1/2} Z_7^{1/4} Z_9^{-1/4} \\
& \ll & B \sum_{\widehat{Z_1},\widehat{Z_2},\widehat{Z_3}} Z_4^{1/2} Z_5^{1/2} Z_8^{-1/2} Z_9^{-1/2} \\
& \ll & B \sum_{\widehat{Z_1},\widehat{Z_2},\widehat{Z_3},\widehat{Z_4}} 1 \textrm{,} \\
\end{eqnarray*}
where we have summed over $Z_2$ and $Z_3$ using respectively \eqref{condition5'bis} and \eqref{condition1'bis}. At the last step, we could have summed over $Z_5$ instead of $Z_4$, so if we assume that $|\xi_i| \leq \log(B)^A$ for a certain $i \neq 1, 2, 3$, where
$A > 0$ is any fixed constant, we get an overall contribution $\ll_A B \log(B)^4 \log(\log(B))$. We now deal with the case where $Z_1^2 Z_6 Z_7 > Z_4 Z_5$. Summing over $Z_2$ and $Z_3$ using \eqref{condition2'bis} and \eqref{condition1'bis}, we get
\begin{eqnarray*}
\sum_{Z_i} \mathcal{N}_{2}''
& \ll & B \sum_{\widehat{Z_2},\widehat{Z_3}} Z_4^{1/2} Z_5^{1/2} Z_8^{-1/2} Z_9^{-1/2} \\
& \ll & B \sum_{\widehat{Z_2},\widehat{Z_3},\widehat{Z_4}} Z_1 Z_6^{1/2} Z_7^{1/2} Z_8^{-1/2} Z_9^{-1/2} \\
& \ll & B \sum_{\widehat{Z_1}, \widehat{Z_2},\widehat{Z_3}, \widehat{Z_4}} 1 \textrm{,}
\end{eqnarray*}
where we have summed over $Z_4 < Z_1^2 Z_5^{-1} Z_6 Z_7$ and over $Z_1 \ll Z_6^{-1/2} Z_7^{-1/2} Z_8^{1/2} Z_9^{1/2}$. We can plainly conclude just as above.
\end{proof}

\subsection{Setting up}

First, we note that the torsor equation \eqref{tor 2} and the height conditions \eqref{condition3'} and \eqref{condition4'} give
\begin{eqnarray}
\label{condition sup}
\xi_1^3 \xi_2 \xi_3 \xi_6^2 \xi_7^2 & \leq & 2 B \textrm{.}
\end{eqnarray}
Our goal is to tackle the equation \eqref{tor 2} by viewing it as a congruence modulo $\xi_9$ if $|\xi_9| \leq |\xi_8|$ and modulo $\xi_8$ if $|\xi_9| > |\xi_8|$, that is why we split the proof into two parts. Let $N_{a}(A,B)$ be the contribution to $N_{U_2,H}(B)$ from the $(\xi_1, \dots, \xi_9) \in \mathcal{T}_2(B)$ such that
\begin{eqnarray}
\label{condition6'}
& & 0 < \xi_9 \leq |\xi_8| \textrm{,} \\
\label{condition7'}
& & \log(B)^A \leq |\xi_4| \textrm{,} \\
\label{condition8'}
& & \log(B)^A \leq |\xi_5| \textrm{,}
\end{eqnarray}
where $A > 0$ is a parameter at our disposal. Symmetrically, let $N_{b}(A,B)$ be the contribution to $N_{U_2,H}(B)$ from the
$(\xi_1, \dots, \xi_9) \in \mathcal{T}_2(B)$ such that
\begin{eqnarray}
\label{condition6''}
& & |\xi_9| > \xi_8 > 0 \textrm{,} \\
\label{condition7''}
& & \log(B)^A \leq |\xi_4| \textrm{,} \\
\label{condition8''}
& & \log(B)^A \leq |\xi_5| \textrm{.}
\end{eqnarray}
Note that in both cases, combining the conditions \eqref{condition3'} and $\log(B)^A \leq |\xi_4|$, we get
\begin{eqnarray}
\label{new condition'}
\log(B)^A \xi_1 \xi_2 \xi_3 |\xi_5| \xi_6 \xi_7 & \leq & B \textrm{.}
\end{eqnarray}
Since $(\xi_8,\xi_9) \mapsto (-\xi_8,-\xi_9)$ is a bijection on the set of solutions, assuming $\xi_9 > 0$ in the first case and $\xi_8 > 0$ in the second brings a factor $2$. Thus, by lemmas \ref{T2} and \ref{lemmalog2} we have the following.

\begin{lemma}
\label{lemma N_a N_b}
For any fixed $A > 0$, we have
\begin{eqnarray*}
N_{U_2,H}(B) & = & N_{a}(A,B) + N_{b}(A,B) + O \left( B \log(B)^4 \log( \log(B)) \right) \textrm{.}
\end{eqnarray*}
\end{lemma}

The next two sections are respectively devoted to the estimations of $N_{a}(A,B)$ and $N_{b}(A,B)$.

\subsection{Estimating $N_{a}(A,B)$}

\label{N_a}

We see that the assumption $\xi_9 \leq |\xi_8|$ and \eqref{condition4'} give the following condition which is crucial in order to apply lemma \ref{lemma tau},
\begin{eqnarray}
\label{conditionkappa'}
\xi_9^2 & \leq & \frac{B}{\xi_1\xi_2\xi_3\xi_6\xi_7} \textrm{.}
\end{eqnarray}
We first estimate the contribution of the variables $\xi_4$, $\xi_5$ and $\xi_8$. We rewrite the coprimality conditions as
\begin{eqnarray}
\label{gcd1'}
& & \gcd(\xi_8,\xi_1\xi_2\xi_4\xi_5\xi_6\xi_7) = 1 \textrm{,} \\
\label{gcd2'}
& & \gcd(\xi_4,\xi_1\xi_2\xi_6\xi_7\xi_9) = 1 \textrm{,} \\
\label{gcd3'}
& & \gcd(\xi_5,\xi_1\xi_3\xi_6\xi_7\xi_9) = 1 \textrm{,} \\
\label{gcd4'}
& & \gcd(\xi_1,\xi_2\xi_3\xi_9) = 1 \textrm{,} \\
\label{gcd5'}
& & \gcd(\xi_3,\xi_2\xi_7\xi_9) = 1 \textrm{,} \\
\label{gcd6'}
& & \gcd(\xi_6,\xi_2\xi_7\xi_9) = 1 \textrm{,} \\
\label{gcd7'}
& & \gcd(\xi_7,\xi_9) = 1 \textrm{.}
\end{eqnarray}
We view the torsor equation \eqref{tor 2} as a congruence modulo $\xi_9$. To do so, we replace the height conditions \eqref{condition1'}, \eqref{condition4'} and \eqref{condition6'} by the following (we keep denoting them by \eqref{condition1'}, \eqref{condition4'} and \eqref{condition6'} respectively), obtained using the torsor equation \eqref{tor 2},
\begin{eqnarray*}
\xi_3^2 \left| \xi_4 \right| \xi_6 \left| \xi_4\xi_5 + \xi_1^2\xi_6\xi_7 \right| \xi_9^{-1} & \leq & B \textrm{,} \\
\xi_1 \xi_2 \xi_3 \xi_6 \xi_7 \left| \xi_4\xi_5 + \xi_1^2\xi_6\xi_7 \right| & \leq & B \textrm{,} \\
\xi_9^2 & \leq & \left| \xi_4\xi_5 + \xi_1^2\xi_6\xi_7 \right| \textrm{.}
\end{eqnarray*}
Set $\boldsymbol{\xi}_{a}' = (\xi_1, \xi_2, \xi_3, \xi_6, \xi_7, \xi_9) \in \mathbb{Z}_{>0}^6$. Assume that
$\boldsymbol{\xi}_{a}' \in \mathbb{Z}_{>0}^6$ is fixed and subject to the height conditions \eqref{condition sup} and \eqref{conditionkappa'} and to the coprimality conditions \eqref{gcd4'}, \eqref{gcd5'}, \eqref{gcd6'} and \eqref{gcd7'}. Let $N_{a}(\boldsymbol{\xi}_{a}',B)$ be the number of $\xi_4$, $\xi_5$ and $\xi_8$ satisfying the torsor equation \eqref{tor 2}, the height conditions \eqref{condition1'}, \eqref{condition2'}, \eqref{condition3'}, \eqref{condition4'} and \eqref{condition5'}, the conditions \eqref{condition6'}, \eqref{condition7'}, \eqref{condition8'} and the coprimality conditions \eqref{gcd1'}, \eqref{gcd2'} and \eqref{gcd3'}.

\begin{lemma}
\label{lemma inter a}
For any fixed $A \geq 8$, we have
\begin{eqnarray*}
N(\boldsymbol{\xi}_{a}',B) & = & \frac1{\xi_9}
\sum_{\substack{k_8|\xi_2 \\ \gcd(k_8,\xi_7) = 1}} \frac{\mu(k_8)}{k_8 \varphi^{\ast}(k_8\xi_9)}
\sum_{\substack{k_4|\xi_1\xi_2\xi_6\xi_7 \\ \gcd(k_4, k_8\xi_9)=1}} \mu(k_4)
\sum_{\substack{k_5|\xi_1\xi_3\xi_6\xi_7 \\ \gcd(k_5, k_8\xi_9)=1}} \mu(k_5) \\
& & \sum_{\substack{\ell_4|k_8\xi_9 \\ \ell_5|k_8\xi_9}} \mu(\ell_4) \mu(\ell_5) C(\boldsymbol{\xi}_{a}',B) + R(\boldsymbol{\xi}_{a}',B) \textrm{,}
\end{eqnarray*}
where, setting $\xi_4 = k_4 \ell_4 \xi_4''$ and $\xi_5 = k_5 \ell_5 \xi_5''$,
\begin{eqnarray*}
C(\boldsymbol{\xi}_{a}',B) & = &
\# \left\{ \left( \xi_4'', \xi_5'' \right) \in \mathbb{Z}_{\neq 0}^2,
\begin{array}{l}
\eqref{condition1'}, \eqref{condition2'}, \eqref{condition3'}, \eqref{condition4'}, \eqref{condition5'} \\
\eqref{condition6'}, \eqref{condition7'}, \eqref{condition8'}
\end{array}
\right\} \textrm{,}
\end{eqnarray*}
and $\sum_{\boldsymbol{\xi}_{a}'} R(\boldsymbol{\xi}_{a}',B) \ll B \log(B)^2$.
\end{lemma}

Let us remove the coprimality condition \eqref{gcd1'} using a Möbius inversion. We get
\begin{eqnarray*}
N(\boldsymbol{\xi}_{a}',B) & = & \sum_{k_8|\xi_1\xi_2\xi_4\xi_5\xi_6\xi_7} \mu(k_8) S_{k_8}(\boldsymbol{\xi}_{a}',B) \textrm{,}
\end{eqnarray*}
where
\begin{eqnarray*}
S_{k_8}(\boldsymbol{\xi}_{a}',B) & = & \# \left\{ \left( \xi_4, \xi_5, \xi_8' \right) \in \mathbb{Z}_{\neq 0}^3 ,
\begin{array}{l}
\xi_4 \xi_5 + k_8 \xi_8' \xi_9 = - \xi_1^2\xi_6\xi_7 \\
\eqref{condition1'}, \eqref{condition2'}, \eqref{condition3'}, \eqref{condition4'}, \eqref{condition5'} \\
\eqref{condition6'}, \eqref{condition7'}, \eqref{condition8'} \\
\eqref{gcd2'}, \eqref{gcd3'}
\end{array}
\right\} \textrm{.}
\end{eqnarray*}
If $\gcd(k_8,\xi_1\xi_6\xi_7) \neq 1$ or $\gcd(k_8,\xi_4\xi_5) \neq 1$ then $\gcd(\xi_4\xi_5,\xi_1\xi_6\xi_7) \neq 1$ and thus
$S_{k_8}(\boldsymbol{\xi}_{a}',B) = 0$, so we can assume that $\gcd(k_8,\xi_1\xi_4\xi_5\xi_6\xi_7) = 1$. We have
\begin{eqnarray*}
S_{k_8}(\boldsymbol{\xi}_{a}',B) & = &
\# \left\{ \left( \xi_4, \xi_5 \right) \in \mathbb{Z}_{\neq 0}^2 ,
\begin{array}{l}
\xi_4 \xi_5 \equiv - \xi_1^2\xi_6\xi_7 \imod{k_8 \xi_9} \\
\eqref{condition1'}, \eqref{condition2'}, \eqref{condition3'}, \eqref{condition4'}, \eqref{condition5'} \\
\eqref{condition6'}, \eqref{condition7'}, \eqref{condition8'} \\
\eqref{gcd2'}, \eqref{gcd3'}
\end{array}
\right\} + R_0(\boldsymbol{\xi}_{a}',B) \textrm{,}
\end{eqnarray*}
where the error term $R_0(\boldsymbol{\xi}_{a}',B)$ comes from the fact $\xi_8'$ has to be non-zero. Otherwise, we would have
$\xi_4 \xi_5 = - \xi_1^2\xi_6\xi_7$ and thus $|\xi_4| = |\xi_5| = \xi_1 = \xi_6 = \xi_7 = 1$. Summing over $\xi_9$ using the condition \eqref{conditionkappa'}, we easily obtain
\begin{eqnarray*}
\sum_{k_8,\boldsymbol{\xi}_{a}'} |\mu(k_8)| R_0(\boldsymbol{\xi}_{a}',B) & \ll & B \log(B)^2 \textrm{.}
\end{eqnarray*}
Let us remove the coprimality conditions \eqref{gcd2'} and \eqref{gcd3'}. The main term of $N(\boldsymbol{\xi}_{a}',B)$ is equal to
\begin{eqnarray*}
& & \sum_{\substack{k_8|\xi_2 \\ \gcd(k_8,\xi_1\xi_6\xi_7) = 1}} \mu(k_8)
\sum_{\substack{k_4|\xi_1\xi_2\xi_6\xi_7\xi_9 \\ \gcd(k_4, k_8\xi_9)=1}} \mu(k_4)
\sum_{\substack{k_5|\xi_1\xi_3\xi_6\xi_7\xi_9 \\ \gcd(k_5, k_8\xi_9)=1}} \mu(k_5) S(\boldsymbol{\xi}_{a}',B) \textrm{,}
\end{eqnarray*}
where, with the notations $\xi_4 = k_4 \xi_4'$ and $\xi_5 = k_5 \xi_5'$,
\begin{eqnarray*}
S(\boldsymbol{\xi}_{a}',B) & = & \# \left\{ \left( \xi_4', \xi_5' \right) \in \mathbb{Z}_{\neq 0}^2,
\begin{array}{l}
\xi_4'\xi_5' \equiv - (k_4k_5)^{-1}\xi_1^2\xi_6\xi_7 \imod{k_8 \xi_9} \\
\eqref{condition1'}, \eqref{condition2'}, \eqref{condition3'}, \eqref{condition4'}, \eqref{condition5'} \\
\eqref{condition6'}, \eqref{condition7'}, \eqref{condition8'}
\end{array}
\right\} \textrm{.}
\end{eqnarray*}
Indeed, we clearly have $\gcd(k_4k_5,k_8\xi_9) = 1$ since $\gcd(k_8\xi_9,\xi_1\xi_6\xi_7) = 1$. We can therefore remove $\xi_9$ from the conditions on $k_4$ and $k_5$. We now proceed to apply lemma \ref{lemma tau2}. To do so, define
\begin{eqnarray*}
X & = & \frac{B}{k_4k_5\xi_1\xi_2\xi_3\xi_6\xi_7} \textrm{.}
\end{eqnarray*}
An argument identical to the one developed in the proof of lemma \ref{lemma inter} shows that assuming that
$k_8 \leq (k_4 k_5)^{-1/2} X^{1/6}$ produces an error term $N'(\boldsymbol{\xi}_{a}',B)$ which satisfies
\begin{eqnarray*}
\sum_{\xi_9} N'(\boldsymbol{\xi}_{a}',B) & \ll & \left( \frac{B}{\xi_1\xi_2\xi_3\xi_6\xi_7} \right)^{1 + 2 \varepsilon - 1/12}
+ 2^{\omega(\xi_2)} \left( \frac{B}{\xi_1\xi_2\xi_3\xi_6\xi_7} \right)^{1/2 + \varepsilon} \textrm{.}
\end{eqnarray*}
Choosing $\varepsilon = 1/48$ and summing over $\xi_3$ using \eqref{condition sup}, we see that
\begin{eqnarray*}
\sum_{\boldsymbol{\xi}_{a}'} N'(\boldsymbol{\xi}_{a}',B) & \ll & \sum_{\xi_1,\xi_2,\xi_6,\xi_7}
\left( \frac{B}{\xi_1^{13/12}\xi_2\xi_6^{25/24}\xi_7^{25/24}} + 2^{\omega(\xi_2)} \frac{B}{\xi_1^{47/24}\xi_2\xi_6^{71/48}\xi_7^{71/48}} \right) \\
& \ll & B \log(B)^2 \textrm{.}
\end{eqnarray*}
The assumption $k_8 \leq (k_4 k_5)^{-1/2} X^{1/6}$ and the condition \eqref{conditionkappa'} prove that we now have
$k_8 \xi_9 \leq X^{2/3}$. We apply the first estimate of lemma \ref{lemma tau2} with $L_1 = \log(B)^A/k_4$, $L_2 = \log(B)^A/k_5$ and
$T = \xi_1^2\xi_6\xi_7/(k_4k_5)$. We have $T \leq 2X$ by \eqref{condition sup} and $k_8 \xi_9 \leq X^{2/3}$ thus lemma
\ref{lemma tau2} shows that
\begin{eqnarray*}
S(\boldsymbol{\xi}_{a}',B) & = & S^{\ast}(\boldsymbol{\xi}_{a}',B) +
O \left( \frac{X^{2/3 + \varepsilon}}{(k_8 \xi_9)^{1/2}} + \frac{X}{\varphi(k_8 \xi_9)} \left( \frac{k_4}{\log(B)^A} + \frac{k_5}{\log(B)^A} \right) \right) \textrm{,}
\end{eqnarray*}
for all $\varepsilon > 0$, with
\begin{eqnarray*}
S^{\ast}(\boldsymbol{\xi}_{a}',B) & = & \frac1{\varphi(k_8\xi_9)} \# \left\{ \left( \xi_4', \xi_5' \right) \in \mathbb{Z}_{\neq 0}^2,
\begin{array}{l}
\gcd(\xi_4'\xi_5',k_8\xi_9) = 1 \\
\eqref{condition1'}, \eqref{condition2'}, \eqref{condition3'}, \eqref{condition4'}, \eqref{condition5'} \\
\eqref{condition6'}, \eqref{condition7'}, \eqref{condition8'}
\end{array}
\right\} \textrm{.}
\end{eqnarray*}
The Möbius inversions do not play any part in the estimation of the contribution of the first error term. Using \eqref{conditionkappa'} to sum over $\xi_9$, we find that this contribution is
\begin{eqnarray*}
\sum_{\xi_1,\xi_2,\xi_3,\xi_6,\xi_7} \left( \frac{B}{\xi_1\xi_2\xi_3\xi_6\xi_7} \right)^{11/12 + \varepsilon} & \ll &
\sum_{\xi_1,\xi_2,\xi_6,\xi_7}
\frac{B}{\xi_1^{7/6 - 2 \varepsilon} \xi_2 \xi_6^{13/12 - \varepsilon} \xi_7^{13/12 - \varepsilon}} \\
& \ll & B \log(B)
\textrm{,}
\end{eqnarray*}
for $\varepsilon = 1/24$ and where we have used \eqref{condition sup} to sum over $\xi_3$. The contribution of the second error term is
\begin{eqnarray*}
\sum_{\boldsymbol{\xi}_{a}'} 2^{\omega(\xi_1\xi_2\xi_6\xi_7)} \frac{B \log(B)^{-A}}{\xi_1\xi_2\xi_3\xi_6\xi_7\xi_9} & \ll & B \log(B)^{10 - A} \textrm{,}
\end{eqnarray*}
which is satisfactory provided that $A \geq 8$. The contribution of the third error term is also $\ll B \log(B)^{10 - A}$. Furthermore, we have
\begin{eqnarray*}
S^{\ast}(\boldsymbol{\xi}_{a}',B) & = & \frac1{\varphi(k_8\xi_9)} \sum_{\ell_4|k_8\xi_9} \mu(\ell_4) \sum_{\ell_5|k_8\xi_9} \mu(\ell_5) C(\boldsymbol{\xi}_{a}',B) \textrm{,}
\end{eqnarray*}
with the notations $\xi_4' = \ell_4 \xi_4''$ and $\xi_5' = \ell_5 \xi_5''$. It is obvious that
\begin{eqnarray*}
C(\boldsymbol{\xi}_{a}',B) & \ll & \left( \frac{X}{\ell_4 \ell_5} \right)^{1 + \varepsilon} \textrm{.}
\end{eqnarray*}
Let us use this bound to estimate the overall contribution of the error term produced if we remove the condition
$k_8 \leq (k_4k_5)^{-1/2} X^{1/6}$ from the sum over $k_8$. Writing $k_8 > k_8^{1/2} (k_4k_5)^{-1/4} X^{1/12}$ and choosing $\varepsilon = 1/24$, we infer that this contribution is
\begin{eqnarray*}
\sum_{\boldsymbol{\xi}_{a}'} \frac1{\xi_9} \left( \frac{B}{\xi_1\xi_2\xi_3\xi_6\xi_7} \right)^{23/24} & \ll &
\sum_{\xi_1,\xi_2,\xi_6,\xi_7,\xi_9} \frac{B}{\xi_1^{13/12}\xi_2\xi_6^{25/24}\xi_7^{25/24}\xi_9} \\
& \ll & B \log(B)^2 \textrm{.}
\end{eqnarray*}
We can remove the condition $\gcd(k_8,\xi_1 \xi_6) = 1$ from the sum over $k_8$ since it follows from $k_8|\xi_2$ and
$\gcd(\xi_1 \xi_6,\xi_2) = 1$, which completes the proof of lemma \ref{lemma inter a}.

We intend to sum over $\xi_1$ also. For this, we set
$\boldsymbol{\xi}_{a} = (\xi_2, \xi_3, \xi_6, \xi_7, \xi_9) \in \mathbb{Z}_{>0}^5$. We also define
${\boldsymbol{\xi}_{a}}^{(r_2,r_3,r_6,r_7,r_9)} = \xi_2^{r_2} \xi_3^{r_3} \xi_6^{r_6} \xi_7^{r_7} \xi_9^{r_9}$ for $(r_2,r_3,r_6,r_7,r_9) \in \mathbb{Q}^5$. Setting
\begin{align*}
& Y_4 = \frac{B^{1/3}}{\boldsymbol{\xi}_{a}^{(-2/3,4/3,2/3,-1/3,-1)}} \textrm{,} &
& Y_4'' = \frac{Y_4}{k_4\ell_4} \textrm{,} \\
& Y_5 = \frac{B^{1/3}}{\boldsymbol{\xi}_{a}^{(4/3,-2/3,-1/3,2/3,1)}} \textrm{,} &
& Y_5'' = \frac{Y_5}{k_5\ell_5} \textrm{,} \\
& Y_1 = \frac{B^{1/3}}{\boldsymbol{\xi}_{a}^{(1/3,1/3,2/3,2/3,0)}} \textrm{,} &
\end{align*}
and recalling the definition \eqref{h^{a}} of the function $h^{a}$, we can sum up the height \text{conditions} \eqref{condition1'}, \eqref{condition2'}, \eqref{condition3'}, \eqref{condition4'} and \eqref{condition5'} as
\begin{eqnarray*}
h^{a} \left( \frac{\xi_4''}{Y_4''}, \frac{\xi_5''}{Y_5''}, \frac{\xi_1}{Y_1} \right) & \leq & 1 \textrm{.}
\end{eqnarray*}
Note also that the condition \eqref{condition sup} can be rewritten as
\begin{eqnarray*}
\frac{\xi_1}{Y_1} & \leq & 2^{1/3} \textrm{.}
\end{eqnarray*}
We also define the following real-valued functions
\begin{eqnarray*}
g_1^{a} & : & (t_5,t_1,t;\boldsymbol{\xi}_{a}, B) \mapsto
\int_{h^{a}(t_4,t_5,t_1) \leq 1, t \leq |t_4 t_5 + t_1^2|, |t_4|Y_4 \geq \log(B)^A} \D t_4 \textrm{,} \\
g_2^{a} & : & (t_1,t;\boldsymbol{\xi}_{a}, B) \mapsto \int_{|t_5| Y_5 \geq \log(B)^A} g_1^{a}(t_5,t_1,t;\boldsymbol{\xi}_{a}, B) \D t_5 \textrm{,} \\
g_3^{a} & : & (t;\boldsymbol{\xi}_{a}, B) \mapsto \int_{t_1 > 0} g_2^{a}(t_1,t; \boldsymbol{\xi}_{a}, B) \D t_1 \textrm{,} \\
g_4^{a} & : & t \mapsto \int \int \int_{t_1 > 0, h^{a}(t_4,t_5,t_1) \leq 1, t \leq |t_4 t_5 + t_1^2|} \D t_4 \D t_5 \D t_1 \textrm{.}
\end{eqnarray*}
The condition $t \leq |t_4 t_5 + t_1^2|$ corresponds to the condition \eqref{condition6'} which can now be rewritten as
\begin{eqnarray*}
\frac{\xi_9^2}{Y_4Y_5} & \leq & \left| \frac{\xi_4''}{Y_4''} \frac{\xi_5''}{Y_5''} + \left( \frac{\xi_1}{Y_1} \right)^2 \right| \textrm{.}
\end{eqnarray*}
We denote by $\kappa_{a}$ the left-hand side of this inequality.

\begin{lemma}
\label{bounds'}
We have the bounds
\begin{eqnarray*}
g_1^{a}(t_5,t_1,t;\boldsymbol{\xi}_{a}, B) & \ll & t_1^{-1}|t_5|^{-1} \textrm{,} \\
g_2^{a}(t_1,t;\boldsymbol{\xi}_{a}, B) & \ll & 1 \textrm{.}
\end{eqnarray*}
\end{lemma}

\begin{proof}
Recall the definition \eqref{h^{a}} of the function $h^{a}$. The first bound is clear since $t_1|t_4t_5|\leq 1$. Moreover, the conditions $|t_4||t_4t_5+t_1^2| \leq 1$ and $|t_4|t_5^2 \leq 1$ show that $t_5$ runs over a set whose measure is
$\ll \min \left( t_4^{-2}, |t_4|^{-1/2} \right)$. Splitting the integration of this minimum over $t_4$ depending on whether $|t_4|$ is greater or less than $1$ completes the proof.
\end{proof}

It is easy to check that $\boldsymbol{\xi}_{a}$ is restricted to lie in the region
\begin{eqnarray}
\label{Va}
\mathcal{V}_a & = & \left\{ \boldsymbol{\xi}_{a} \in \mathbb{Z}_{>0}^5, Y_1 \geq 2^{-1/3}, Y_1 Y_4 Y_5 \geq \xi_9^2 \right\} \textrm{.}
\end{eqnarray}
Assume that $\boldsymbol{\xi}_{a} \in \mathcal{V}_a$ and $\xi_1 \in \mathbb{Z}_{>0}$ are fixed and satisfy the coprimality conditions \eqref{gcd4'}, \eqref{gcd5'}, \eqref{gcd6'} and \eqref{gcd7'}.

We now proceed to estimate $C(\boldsymbol{\xi}_{a}',B)$. Recall the condition \eqref{new condition'} which can be rewritten as
$|\xi_5''| \leq Y_1 \xi_1^{-1} Y_4 Y_5'' \log(B)^{-A}$. Let us sum over $\xi_4''$ using the basic estimate
$\# \{ n \in \mathbb{Z} , t_1 \leq n \leq t_2 \} = t_2 - t_1 + O(1)$.
The change of variable $t_4 \mapsto Y_4'' t_4$ shows that
\begin{eqnarray*}
C(\boldsymbol{\xi}_{a}',B) & = & \sum_{|\xi_5''| \leq Y_1 \xi_1^{-1} Y_4 Y_5'' \log(B)^{-A}} \left( Y_4'' g_1^{a} \left( \frac{\xi_5''}{Y_5''}, \frac{\xi_1}{Y_1}, \kappa_{a}; \boldsymbol{\xi}_{a}, B \right) + O(1) \right) \textrm{.}
\end{eqnarray*}
We see that the overall contribution of the error term is
\begin{eqnarray*}
\sum_{\boldsymbol{\xi}_{a}'} 2^{\omega(\xi_1\xi_2\xi_6\xi_7)} 2^{\omega(\xi_2\xi_9)}
\frac{B \log(B)^{-A}}{\boldsymbol{\xi}_{a}^{(1,1,1,1,1)}\xi_1} & \ll & \log(B)^{13-A} \textrm{.} 
\end{eqnarray*}
Let us now sum over $\xi_5''$. Partial summation and the change of variable $t_5 \mapsto Y_5'' t_5$ yield
$$C(\boldsymbol{\xi}_{a}',B) = Y_4'' Y_5'' g_2^{a} \left( \frac{\xi_1}{Y_1}, \kappa_{a}; \boldsymbol{\xi}_{a}, B \right) + O
\left( Y_4'' \sup_{|t_5| Y_5 \geq \log(B)^A} g_1^{a} \left( t_5, \frac{\xi_1}{Y_1}, \kappa_{a}; \boldsymbol{\xi}_{a}, B \right) \right) \textrm{.}$$
Using the bound of lemma \ref{bounds'} for $g_1^{a}$, we get
\begin{eqnarray*}
\sup_{|t_5| Y_5 \geq \log(B)^A} g_1^{a} \left( t_5, \frac{\xi_1}{Y_1}, \kappa; \boldsymbol{\xi}_{a}, B \right) & \ll &
Y_5 \log(B)^{-A} \frac{Y_1}{\xi_1} \textrm{.}
\end{eqnarray*}
The overall contribution coming from this error term is therefore
\begin{eqnarray*}
\sum_{\boldsymbol{\xi}_{a}'} 2^{\omega(\xi_1\xi_3\xi_6\xi_7)} 2^{\omega(\xi_2\xi_9)}
\frac{B \log(B)^{-A}}{\boldsymbol{\xi}_{a}^{(1,1,1,1,1)}\xi_1}
& \ll & B \log(B)^{12-A} \textrm{.}
\end{eqnarray*}
Recalling lemma \ref{lemma inter a}, we find that for any fixed $A \geq 9$, we have
\begin{eqnarray*}
N(\boldsymbol{\xi}_{a}',B) & = & \frac1{\xi_9} \theta_{a}(\boldsymbol{\xi}_{a}) \frac{\varphi^{\ast}(\xi_1)}{\varphi^{\ast}(\gcd(\xi_1,\xi_2\xi_6\xi_7))} \frac{\varphi^{\ast}(\xi_1)}{\varphi^{\ast}(\gcd(\xi_1,\xi_3\xi_6\xi_7))} \\
& & g_2^{a} \left(\frac{\xi_1}{Y_1}, \kappa_a; \boldsymbol{\xi}_{a}, B \right) Y_4Y_5 + R_1(\boldsymbol{\xi}_{a}',B) \textrm{,}
\end{eqnarray*}
where
\begin{eqnarray*}
\theta_{a}(\boldsymbol{\xi}_{a}) & = & \varphi^{\ast}(\xi_2\xi_6\xi_7) \varphi^{\ast}(\xi_3\xi_6\xi_7) \frac{\varphi^{\ast}(\xi_2\xi_9)}{\varphi^{\ast}(\gcd(\xi_2,\xi_7))} \textrm{,}
\end{eqnarray*}
and where $\sum_{\boldsymbol{\xi}_{a}'} R_1(\boldsymbol{\xi}_{a}',B) \ll B \log(B)^4$. For fixed
$\boldsymbol{\xi}_{a} \in \mathcal{V}_a$ satisfying the coprimality conditions \eqref{gcd5'}, \eqref{gcd6'} and \eqref{gcd7'}, let $\mathbf{N}(\boldsymbol{\xi}_{a},B)$ be the sum over $\xi_1$ of the main term of $N(\boldsymbol{\xi}_{a}',B)$, with $\xi_1$ subject to the coprimality condition \eqref{gcd4'}. Let us make use of lemma \ref{arithmetic preliminary} to sum over $\xi_1$. We find that for any fixed $A \geq 9$ and $0 < \sigma \leq 1$, we have
\begin{eqnarray}
\label{Na}
\ \ \ \ \ \ \ \mathbf{N}(\boldsymbol{\xi}_{a},B) & = & \frac1{\xi_9} \mathcal{P} \Theta_{a}(\boldsymbol{\xi}_{a}) g_3^{a} \left( \kappa_a; \boldsymbol{\xi}_{a}, B \right) Y_4 Y_5 Y_1 \\
\nonumber
& & + O \left( \frac{Y_4Y_5}{\xi_9} \varphi_{\sigma}(\xi_2\xi_3\xi_9) Y_1^{\sigma}
\sup_{t_1 > 0} g_2^{a} \left( t_1, \kappa_a; \boldsymbol{\xi}_{a}, B \right) \right) \textrm{,}
\end{eqnarray}
where
\begin{eqnarray*}
\Theta_{a}(\boldsymbol{\xi}_{a}) & = & \theta_{a}(\boldsymbol{\xi}_{a}) \varphi^{\ast}(\xi_2\xi_3\xi_9) \varphi'(\xi_2\xi_3\xi_6\xi_7\xi_9) \textrm{.}
\end{eqnarray*}
Using the bound of lemma \ref{bounds'} for $g_2^{a}$ and choosing $\sigma = 1/2$, we see that the overall contribution of the error term is
\begin{eqnarray*}
\sum_{\boldsymbol{\xi}_{a}} \varphi_{\sigma}(\xi_2\xi_3\xi_9) \frac{Y_4Y_5}{\xi_9} Y_1^{1/2} & \ll &
\sum_{\xi_2,\xi_3,\xi_6,\xi_9} \varphi_{\sigma}(\xi_2\xi_3\xi_9) \frac{B}{\boldsymbol{\xi}_{a}^{(1,1,1,0,1)}}\\
& \ll & B \log(B)^4 \textrm{,}
\end{eqnarray*}
where we have summed over $\xi_7$ using $Y_1 \geq 2^{-1/3}$ and used the fact that $\varphi_{\sigma}$ has average order $O(1)$. Note that
\begin{eqnarray*}
\frac{Y_4 Y_5 Y_1}{\xi_9} & = & \frac{B}{\boldsymbol{\xi}_{a}^{(1,1,1,1,1)}} \textrm{.}
\end{eqnarray*}
For brevity, we set
\begin{eqnarray*}
D_{h^{a}} & = & \left\{ (t_4,t_5,t_1) \in \mathbb{R}^3, t_1 > 0, h^{a}(t_4,t_5,t_1) \leq 1 \right\} \textrm{.}
\end{eqnarray*}

\begin{lemma}
\label{bounds integrals}
For $Z_4, Z_5 > 0$, we have
\begin{eqnarray}
\label{1}
\meas \{ (t_4,t_5,t_1) \in D_{h^{a}}, |t_4| Z_4 \geq 1 \} & \ll & Z_4 \textrm{,} \\
\label{2}
\meas \{ (t_4,t_5,t_1) \in D_{h^{a}}, |t_5| Z_5 \geq 1 \} & \ll & Z_5 \textrm{,} \\
\label{4}
\meas \{ (t_4,t_5,t_1) \in D_{h^{a}}, |t_4| Z_4 < 1 \} & \ll & Z_4^{-1/2} \textrm{,} \\
\label{5}
\meas \{ (t_4,t_5,t_1) \in D_{h^{a}}, |t_5| Z_5 < 1 \} & \ll & Z_5^{-1/2} \textrm{.}
\end{eqnarray}
\end{lemma}

\begin{proof}
First, the conditions $t_1|t_4t_5| \leq 1$ and $t_1|t_4t_5 + t_1^2| \leq 1$ show that we always have $t_1^3 \leq 2$. Using
$|t_4||t_4t_5 + t_1^2| \leq 1$, we see that $t_5$ runs over a set whose measure is $\ll |t_4|^{-2}$ and thus integrating over $t_1$ using $t_1^3 \leq 2$ gives \eqref{1}. Since $|t_4|t_5^2 \leq 1$, we see that $t_4$ runs over a set whose measure is $\ll t_5^{-2}$. Integrating this over $t_1$ using $t_1^3 \leq 2$ leads to \eqref{2}. Furthermore, integrating over $t_5$ using $|t_4|t_5^2 \leq 1$ and then over $t_1$ using $t_1^3 \leq 2$ leads to \eqref{4}. Finally, the condition $|t_4||t_4t_5 + t_1^2| \leq 1$ shows that $t_4$ runs over a set whose measure is $\ll |t_5|^{-1/2}$ and integrating this quantity over $t_1$ using $t_1^3 \leq 2$ proves \eqref{5}.
\end{proof}

Exactly as in section \ref{Summing} for the case of the $3 \mathbf{A}_1$ surface, the bounds \eqref{1} and \eqref{2} show that if we do not have $Y_4, Y_5 \geq \log(B)^A$, the contribution of the main term of $\mathbf{N}(\boldsymbol{\xi}_{a},B)$ is $\ll B \log(B)^4$. Thus we can assume from now on that
\begin{eqnarray}
\label{condition Va1}
Y_4 & \geq & \log(B)^A \textrm{,} \\
\label{condition Va2}
Y_5 & \geq & \log(B)^A \textrm{,}
\end{eqnarray}
and the two bounds \eqref{4}, \eqref{5} therefore show that removing the conditions
$|t_4| Y_4, |t_5| Y_5 \geq \log(B)^A$ from the integral defining $g_3^{a}$ in the main term of $\mathbf{N}(\boldsymbol{\xi}_{a},B)$  in \eqref{Na} creates an error term whose overall contribution is $\ll B \log(B)^4$. We have thus proved that for any fixed
$A \geq 9$,
\begin{eqnarray}
\label{estimate Na}
\mathbf{N}(\boldsymbol{\xi}_{a},B) & = & \mathcal{P} g_4^{a}(\kappa_a) \frac{B}{\boldsymbol{\xi}_{a}^{(1,1,1,1,1)}}
\Theta_{a}(\boldsymbol{\xi}_{a}) + R_2(\boldsymbol{\xi}_{a},B) \textrm{,}
\end{eqnarray}
where $\sum_{\boldsymbol{\xi}_{a}} R_2(\boldsymbol{\xi}_{a},B) \ll B \log(B)^4$.

\begin{lemma}
For $t > 0$, we have
\begin{eqnarray}
\label{7}
\meas \{ (t_4,t_5,t_1) \in D_{h^{a}}, |t_4t_5 + t_1^2| \geq t \} & \ll & t^{-3/2} \textrm{,} \\
\label{8}
\meas \{ (t_4,t_5,t_1) \in D_{h^{a}}, |t_4t_5 + t_1^2| < t \} & \ll & t^{1/2} \textrm{.}
\end{eqnarray}
\end{lemma}

\begin{proof}
First, the conditions $|t_4||t_4t_5 + t_1^2| \leq 1$, $t_1|t_4t_5 + t_1^2| \leq 1$ and $|t_4t_5 + t_1^2| \geq t$ yield $t|t_4| \leq 1$ and $t t_1 \leq 1$. Therefore, integrating over $t_5$ using $|t_4|t_5^2 \leq 1$ and then over $|t_4|,t_1 \leq t^{-1}$ yields \eqref{7}. In addition, the condition $|t_4t_5 + t_1^2| < t$ shows that $t_4$ runs over a set whose measure is
$\ll \min \left( t|t_5|^{-1}, |t_5|^{-1/2} \right) \leq t^{1/2} |t_5|^{-3/4}$. Integrating this quantity over $t_5$ using $t_1^2|t_5| \leq 1$ and then over $t_1$ using $t_1^3 \leq 2$ gives \eqref{8}.
\end{proof}

The bound \eqref{7} shows that if $\kappa_a > 1$, the contribution of the main term of $\mathbf{N}(\boldsymbol{\xi}_{a},B)$ is
$\ll B \log(B)^4$, thus we assume from now on that $\kappa_a \leq 1$, namely
\begin{eqnarray}
\label{condition Va3}
Y_4Y_5 & \geq & \xi_9^2 \textrm{.}
\end{eqnarray}
Replacing $g_4^{a}(\kappa_a)$ by $g_4^{a}(0)$ in the main term of $\mathbf{N}(\boldsymbol{\xi}_{a},B)$ in \eqref{estimate Na} therefore produces an error term whose overall contribution is $\ll B \log(B)^4$ thanks to \eqref{8}. Since $g_4^{a}(0) = \tau_{\infty}/3$ by \eqref{tau 1}, we have obtained the following result.

\begin{lemma}
\label{lemmafin a}
For any fixed $A \geq 9$, we have
\begin{eqnarray*}
\mathbf{N}(\boldsymbol{\xi}_{a},B) & = & \mathcal{P} \frac{\tau_{\infty}}{3} \frac{B}{\boldsymbol{\xi}_{a}^{(1,1,1,1,1)}} \Theta_{a}(\boldsymbol{\xi}_{a}) + R_3(\boldsymbol{\xi}_{a},B) \textrm{,}
\end{eqnarray*}
where $\sum_{\boldsymbol{\xi}_{a}} R_3(\boldsymbol{\xi}_{a},B) \ll B \log(B)^4$.
\end{lemma}

Recall the definition \eqref{Va} of $\mathcal{V}_a$. It remains to sum the main term of $\mathbf{N}(\boldsymbol{\xi}_{a},B)$ over the $\boldsymbol{\xi}_{a} \in \mathcal{V}_a$ satifying \eqref{condition Va1}, \eqref{condition Va2} and \eqref{condition Va3} and the coprimality conditions \eqref{gcd5'}, \eqref{gcd6'} and \eqref{gcd7'}. It is easy to check that replacing
$\left\{ \boldsymbol{\xi}_{a} \in \mathcal{V}_a, \eqref{condition Va1}, \eqref{condition Va2}, \eqref{condition Va3} \right\}$ by the region 
\begin{eqnarray*}
\mathcal{V}_a' & = & \left\{ \boldsymbol{\xi}_{a} \in \mathbb{Z}_{>0}^5, Y_4 \geq 1, Y_5 \geq 1, Y_1 \geq 1, Y_4Y_5 \geq \xi_9^2, \right\}  \textrm{,}
\end{eqnarray*}
produces an error term whose overall contribution is $\ll B \log(B)^4 \log(\log(B))$. Let us redefine $\Theta_{a}$ as being equal to zero if the remaining coprimality conditions \eqref{gcd5'}, \eqref{gcd6'} and \eqref{gcd7'} are not satisfied. Lemma \ref{lemmafin a} proves that for any fixed $A \geq 9$, we have
\begin{eqnarray*}
N_{a}(A,B) & = & \mathcal{P} \frac{\tau_{\infty}}{3} B \sum_{\boldsymbol{\xi}_{a} \in \mathcal{V}_a'} \frac{\Theta_{a}(\boldsymbol{\xi}_{a})}{\boldsymbol{\xi}_{a}^{(1,1,1,1,1)}} + O \left( B \log(B)^4 \log(\log(B)) \right) \textrm{.}
\end{eqnarray*}
As in section \ref{Conclusion}, $\Theta_{a}$ satifies the assumption \eqref{Psi} of lemma \ref{final sum} and thus we get
$$N_{a}(A,B) = \mathcal{P} \frac{\tau_{\infty}}{3} \alpha_{a} \left( \sum_{\boldsymbol{\xi}_{a} \in \mathbb{Z}_{>0}^5}
\frac{(\Theta_{a} \ast \boldsymbol{\mu})(\boldsymbol{\xi}_{a})}{\boldsymbol{\xi}_{a}^{(1,1,1,1,1)}} \right) B \log(B)^5
+ O \left( B \log(B)^4 \log(\log(B)) \right) \textrm{,}$$
where $\alpha_{a}$ is the volume of the polytope defined in $\mathbb{R}^5$ by $t_2,t_3,t_6,t_7,t_9 \geq 0$ and
\begin{eqnarray*}
- 2 t_2 + 4 t_3 + 2 t_6 - t_7 - 3 t_9 & \leq & 1 \textrm{,} \\
4 t_2 - 2 t_3 - t_6 + 2 t_7 + 3 t_9 & \leq & 1 \textrm{,} \\
t_2 + t_3 + 2 t_6 + 2 t_7 & \leq & 1 \textrm{,} \\
2 t_2 + 2 t_3 + t_6 + t_7 + 6 t_9 & \leq & 2 \textrm{.}
\end{eqnarray*}
A computation using \cite{Convex} gives
\begin{eqnarray}
\label{alpha_a}
\alpha_a & = & \frac{1871}{2016000} \textrm{,}
\end{eqnarray}
and moreover, as in section \ref{Conclusion}, we get
\begin{eqnarray*}
\sum_{\boldsymbol{\xi}_{a} \in \mathbb{Z}_{>0}^5}
\frac{(\Theta_{a} \ast \boldsymbol{\mu})(\boldsymbol{\xi}_{a})}{\boldsymbol{\xi}_{a}^{(1,1,1,1,1)}} & = &
\mathcal{P}^{-1} \prod_p \left( 1 - \frac1{p} \right)^6 \tau_p \textrm{,}
\end{eqnarray*}
thus we have obtained the following lemma.
\begin{lemma}
\label{lemma N_a}
For any fixed $A \geq 9$, we have
\begin{eqnarray*}
N_{a}(A,B) & = & \frac1{3} \alpha_{a} \omega_H(\widetilde{V_2}) B \log(B)^5 + O \left( B \log(B)^4 \log(\log(B)) \right) \textrm{.}
\end{eqnarray*}
\end{lemma}

\subsection{Estimating $N_{b}(A,B)$}

Note that the assumption $|\xi_9| > \xi_8$ and \eqref{condition4'} yield in this case
\begin{eqnarray}
\label{conditionkappa''}
\xi_8^2 & < & \frac{B}{\xi_1\xi_2\xi_3\xi_6\xi_7} \textrm{.}
\end{eqnarray}
We estimate the contribution of the variables $\xi_4$, $\xi_5$ and $\xi_9$. To do so, we rewrite the coprimality conditions as
\begin{eqnarray}
\label{gcd1''} 
& & \gcd(\xi_9,\xi_1\xi_3\xi_4\xi_5\xi_6\xi_7) = 1 \textrm{,} \\
\label{gcd2''} 
& & \gcd(\xi_4,\xi_1\xi_2\xi_6\xi_7\xi_8) = 1 \textrm{,} \\
\label{gcd3''} 
& & \gcd(\xi_5,\xi_1\xi_3\xi_6\xi_7\xi_8) = 1 \textrm{,} \\
\label{gcd4''} 
& & \gcd(\xi_1,\xi_2\xi_3\xi_8) = 1 \textrm{,} \\
\label{gcd5''}
& & \gcd(\xi_2,\xi_3\xi_6\xi_8) = 1 \textrm{,} \\
\label{gcd6''} 
& & \gcd(\xi_7,\xi_3\xi_6\xi_8) = 1 \textrm{,} \\
\label{gcd7''}
& & \gcd(\xi_8,\xi_6) = 1 \textrm{,}
\end{eqnarray}
This time, we want to view the torsor equation \eqref{tor 2} as a congruence modulo $\xi_8$. To do so, we replace \eqref{condition2'}, \eqref{condition4'}, \eqref{condition5'} and \eqref{condition6''} by the following (we keep denoting them by \eqref{condition2'}, \eqref{condition4'}, \eqref{condition5'} and \eqref{condition6''}), obtained using the torsor equation
\eqref{tor 2},
\begin{eqnarray*}
\xi_1^2 \xi_2^2 | \xi_5 | \xi_6 \xi_7^2 \left| \xi_4\xi_5 + \xi_1^2\xi_6\xi_7 \right| \xi_8^{-1} & \leq & B \textrm{,} \\
\xi_1 \xi_2 \xi_3 \xi_6 \xi_7 \left| \xi_4\xi_5 + \xi_1^2\xi_6\xi_7 \right| & \leq & B \textrm{,} \\
\xi_2^2 | \xi_4 | \xi_5^2 \xi_7 \left| \xi_4\xi_5 + \xi_1^2\xi_6\xi_7 \right| \xi_8^{-1} & \leq & B \textrm{,} \\
\xi_8^2 & < & \left| \xi_4\xi_5 + \xi_1^2\xi_6\xi_7 \right| \textrm{.}
\end{eqnarray*}
Set $\boldsymbol{\xi}_{b}' = (\xi_1, \xi_2, \xi_3, \xi_6, \xi_7, \xi_8) \in \mathbb{Z}_{>0}^6$. Assume that
$\boldsymbol{\xi}_{b}' \in \mathbb{Z}_{>0}^6$ is fixed and satisfies the height conditions \eqref{condition sup} and \eqref{conditionkappa''} and the coprimality conditions \eqref{gcd4''}, \eqref{gcd5''}, \eqref{gcd6''} and \eqref{gcd7''}. Let
$N_{b}(\boldsymbol{\xi}_{b}',B)$ be the number of $\xi_4$, $\xi_5$ and $\xi_9$ satisfying the torsor equation \eqref{tor 2}, the height conditions \eqref{condition1'}, \eqref{condition2'}, \eqref{condition3'}, \eqref{condition4'}, \eqref{condition5'}, the conditions \eqref{condition6''}, \eqref{condition7''} and \eqref{condition8''} and the coprimality conditions \eqref{gcd1''}, \eqref{gcd2''} and \eqref{gcd3''}.

\begin{lemma}
\label{lemma inter b}
For any fixed $A \geq 8$, we have
\begin{eqnarray*}
N(\boldsymbol{\xi}_{b}',B) & = & \frac1{\xi_8}
\sum_{\substack{k_9|\xi_3 \\ \gcd(k_9,\xi_6) = 1}} \frac{\mu(k_9)}{k_9 \varphi^{\ast}(k_9\xi_8)}
\sum_{\substack{k_4|\xi_1\xi_2\xi_6\xi_7 \\ \gcd(k_4, k_9\xi_8)=1}} \mu(k_4)
\sum_{\substack{k_5|\xi_1\xi_3\xi_6\xi_7 \\ \gcd(k_5, k_9\xi_8)=1}} \mu(k_5) \\
& & \sum_{\substack{\ell_4|k_9\xi_8 \\ \ell_5|k_9\xi_8}} \mu(\ell_4) \mu(\ell_5) C(\boldsymbol{\xi}_{b}',B) + R(\boldsymbol{\xi}_{b}',B) \textrm{,} 
\end{eqnarray*}
where, setting $\xi_4 = k_4 \ell_4 \xi_4''$ and $\xi_5 = k_5 \ell_5 \xi_5''$,
\begin{eqnarray*}
C(\boldsymbol{\xi}_{b}',B) & = &
\# \left\{ \left( \xi_4'', \xi_5'' \right) \in \mathbb{Z}_{\neq 0}^2,
\begin{array}{l}
\eqref{condition1'}, \eqref{condition2'}, \eqref{condition3'}, \eqref{condition4'}, \eqref{condition5'} \\
\eqref{condition6''}, \eqref{condition7''}, \eqref{condition8''}
\end{array}
\right\} \textrm{,}
\end{eqnarray*}
and $\sum_{\boldsymbol{\xi}_{b}'} R(\boldsymbol{\xi}_{b}',B) \ll B \log(B)^2$.
\end{lemma}

Let us remove the coprimality condition \eqref{gcd1'} using a Möbius inversion. We get
\begin{eqnarray*}
N(\boldsymbol{\xi}_{b}',B) & = & \sum_{k_9|\xi_1\xi_3\xi_4\xi_5\xi_6\xi_7} \mu(k_9) S_{k_9}(\boldsymbol{\xi}_{b}',B) \textrm{,}
\end{eqnarray*}
where
\begin{eqnarray*}
S_{k_9}(\boldsymbol{\xi}_{b}',B) & = & \# \left\{ \left( \xi_4, \xi_5, \xi_9' \right) \in \mathbb{Z}_{\neq 0}^3 ,
\begin{array}{l}
\xi_4 \xi_5 + k_9 \xi_8 \xi_9' = - \xi_1^2\xi_6\xi_7 \\
\eqref{condition1'}, \eqref{condition2'}, \eqref{condition3'}, \eqref{condition4'}, \eqref{condition5'} \\
\eqref{condition6''}, \eqref{condition7''}, \eqref{condition8''} \\
\eqref{gcd2''}, \eqref{gcd3''}
\end{array}
\right\} \textrm{.}
\end{eqnarray*}
Note that if $\gcd(k_9,\xi_1\xi_6\xi_7) \neq 1$ or $\gcd(k_9,\xi_4\xi_5) \neq 1$ then $\gcd(\xi_4\xi_5,\xi_1\xi_6\xi_7) \neq 1$ and so $S_{k_9}(\boldsymbol{\xi}_{b}',B) = 0$ thus we can assume that $\gcd(k_9,\xi_1\xi_4\xi_5\xi_6\xi_7) = 1$. We have
\begin{eqnarray*}
S_{k_9}(\boldsymbol{\xi}_{b}',B) & = &
\# \left\{ \left( \xi_4, \xi_5 \right) \in \mathbb{Z}_{\neq 0}^2 ,
\begin{array}{l}
\xi_4 \xi_5 \equiv - \xi_1^2\xi_6\xi_7 \imod{k_9 \xi_8} \\
\eqref{condition1'}, \eqref{condition2'}, \eqref{condition3'}, \eqref{condition4'}, \eqref{condition5'} \\
\eqref{condition6''}, \eqref{condition7''}, \eqref{condition8''} \\
\eqref{gcd2''}, \eqref{gcd3''}
\end{array}
\right\} + R_0(\boldsymbol{\xi}_{b}',B) \textrm{,}
\end{eqnarray*}
where the error term $R_0(\boldsymbol{\xi}_{b}',B)$ comes from the fact $\xi_9'$ has to be non-zero. Otherwise, we would have $\xi_4 \xi_5 = - \xi_1^2\xi_6\xi_7$ and thus $|\xi_4| = |\xi_5| = \xi_1 = \xi_6 = \xi_7 = 1$. Summing over $\xi_8$ using the condition \eqref{conditionkappa''}, we easily obtain
\begin{eqnarray*}
\sum_{k_9,\boldsymbol{\xi}_{b}'} |\mu(k_9)| R_0(\boldsymbol{\xi}_{b}',B) & \ll & B \log(B)^2 \textrm{.}
\end{eqnarray*}
We now remove the coprimality conditions \eqref{gcd2''} and \eqref{gcd3''}. The main term of $N(\boldsymbol{\xi}_{b}',B)$ is equal to
\begin{eqnarray*}
& & \sum_{\substack{k_9|\xi_3 \\ \gcd(k_9,\xi_1\xi_6\xi_7) = 1}} \mu(k_9)
\sum_{\substack{k_4|\xi_1\xi_2\xi_6\xi_7\xi_8 \\ \gcd(k_4, k_9\xi_8)=1}} \mu(k_4)
\sum_{\substack{k_5|\xi_1\xi_3\xi_6\xi_7\xi_8 \\ \gcd(k_5, k_9\xi_8)=1}} \mu(k_5) S(\boldsymbol{\xi}_{b}',B) \textrm{,}
\end{eqnarray*}
where, setting $\xi_4 = k_4 \xi_4'$ and $\xi_5 = k_5 \xi_5'$,
\begin{eqnarray*}
S(\boldsymbol{\xi}_{b}',B) & = & \# \left\{ \left( \xi_4', \xi_5' \right) \in \mathbb{Z}_{\neq 0}^2,
\begin{array}{l}
\xi_4'\xi_5' \equiv - (k_4k_5)^{-1}\xi_1^2\xi_6\xi_7 \imod{k_9 \xi_8} \\
\eqref{condition1'}, \eqref{condition2'}, \eqref{condition3'}, \eqref{condition4'}, \eqref{condition5'} \\
\eqref{condition6''}, \eqref{condition7''}, \eqref{condition8''}
\end{array}
\right\} \textrm{.}
\end{eqnarray*}
Indeed, since $\gcd(k_9\xi_8,\xi_1\xi_6\xi_7) = 1$, we have $\gcd(k_4k_5,k_9\xi_8) = 1$. We can therefore remove $\xi_8$ from the conditions on $k_4$ and $k_5$. Everything is now in place to apply lemma \ref{lemma tau2}. Set
\begin{eqnarray*}
X & = & \frac{B}{k_4k_5\xi_1\xi_2\xi_3\xi_6\xi_7} \textrm{.}
\end{eqnarray*}
An argument identical to the one given in the proof of lemma \ref{lemma inter} shows that assuming that
$k_9 \leq (k_4 k_5)^{-1/2} X^{1/6}$ produces an error term $N'(\boldsymbol{\xi}_{b}',B)$ which satisfies
\begin{eqnarray*}
\sum_{\xi_8} N'(\boldsymbol{\xi}_{b}',B) & \ll & \left( \frac{B}{\xi_1\xi_2\xi_3\xi_6\xi_7} \right)^{1 + 2 \varepsilon - 1/12}
+ 2^{\omega(\xi_3)} \left( \frac{B}{\xi_1\xi_2\xi_3\xi_6\xi_7} \right)^{1/2 + \varepsilon} \textrm{.}
\end{eqnarray*}
Choosing $\varepsilon = 1/48$ and summing over $\xi_2$ using \eqref{condition sup}, we get
\begin{eqnarray*}
\sum_{\boldsymbol{\xi}_{b}'} N'(\boldsymbol{\xi}_{b}',B) & \ll & \sum_{\xi_1,\xi_3,\xi_6,\xi_7}
\left( \frac{B}{\xi_1^{13/12}\xi_3\xi_6^{25/24}\xi_7^{25/24}}
+ 2^{\omega(\xi_3)} \frac{B}{\xi_1^{47/24}\xi_3\xi_6^{71/48}\xi_7^{71/48}} \right) \\
& \ll & B \log(B)^2 \textrm{.}
\end{eqnarray*}
The assumption $k_9 \leq (k_4 k_5)^{-1/2} X^{1/6}$ and \eqref{conditionkappa''} give $k_9 \xi_8 \leq X^{2/3}$. We proceed to apply the second estimate of lemma \ref{lemma tau2}. Set as in the first case $L_1 = \log(B)^A/k_4$, $L_2 = \log(B)^A/k_5$ and
$T = \xi_1^2\xi_6\xi_7/(k_4k_5)$. We have $T \leq 2X$ by \eqref{condition sup} and $k_9 \xi_8 \leq X^{2/3}$ thus lemma \ref{lemma tau2} shows that
\begin{eqnarray*}
S(\boldsymbol{\xi}_{b}',B) & = & S^{\ast}(\boldsymbol{\xi}_{b}',B)
+ O \left( \frac{X^{4/5 + \varepsilon}}{(k_9 \xi_8)^{7/10}} + \frac{X}{\varphi(k_9 \xi_8)} \left( \frac{k_4}{\log(B)^A} + \frac{k_5}{\log(B)^A} \right) \right) \textrm{,}
\end{eqnarray*}
for all $\varepsilon > 0$, with
\begin{eqnarray*}
S^{\ast}(\boldsymbol{\xi}_{b}',B) & = & \frac1{\varphi(k_9\xi_8)} \# \left\{ \left( \xi_4', \xi_5' \right) \in \mathbb{Z}_{\neq 0}^2,
\begin{array}{l}
\gcd(\xi_4'\xi_5',k_9\xi_8) = 1 \\
\eqref{condition1'}, \eqref{condition2'}, \eqref{condition3'}, \eqref{condition4'}, \eqref{condition5'} \\
\eqref{condition6''}, \eqref{condition7''}, \eqref{condition8''}
\end{array}
\right\} \textrm{.}
\end{eqnarray*}
Using \eqref{conditionkappa''} to sum over $\xi_8$, we obtain that the contribution of the first error term is
\begin{eqnarray*}
\sum_{\xi_1,\xi_2,\xi_3,\xi_6,\xi_7} \left( \frac{B}{\xi_1\xi_2\xi_3\xi_6\xi_7} \right)^{19/20 + \varepsilon} & \ll & B \log(B)
\textrm{,}
\end{eqnarray*}
for $\varepsilon = 1/40$ and where we have summed over $\xi_3$ using \eqref{condition sup}. The contributions of the second and the third error terms are easily seen to be both $\ll B \log(B)^{10 - A}$, which is satisfactory if $A \geq 8$. Furthermore, we have
\begin{eqnarray*}
S^{\ast}(\boldsymbol{\xi}_{b}',B) & = & \frac1{\varphi(k_9\xi_8)} \sum_{\ell_4|k_9\xi_8} \mu(\ell_4) \sum_{\ell_5|k_9\xi_8} \mu(\ell_5) C(\boldsymbol{\xi}_{b}',B) \textrm{,}
\end{eqnarray*}
where we have set $\xi_4' = \ell_4 \xi_4''$ and $\xi_5' = \ell_5 \xi_5''$. It is plain that
\begin{eqnarray*}
C(\boldsymbol{\xi}_{b}',B) & \ll & \left( \frac{X}{\ell_4 \ell_5} \right)^{1 + \varepsilon} \textrm{.}
\end{eqnarray*}
Let us use this bound to estimate the overall contribution of the error term produced by removing the condition
$k_9 \leq (k_4k_5)^{-1/2} X^{1/6}$ from the sum over $k_9$. Writing $k_9 > k_9^{1/2} (k_4k_5)^{-1/4} X^{1/12}$ and choosing $\varepsilon = 1/24$, we see that this contribution is
\begin{eqnarray*}
\sum_{\boldsymbol{\xi}_{a}'} \frac1{\xi_8} \left( \frac{B}{\xi_1\xi_2\xi_3\xi_6\xi_7} \right)^{23/24} & \ll &
B \log(B)^2 \textrm{,}
\end{eqnarray*}
as in section \ref{N_a}. We can remove the condition $\gcd(k_9,\xi_1 \xi_7) = 1$ from the sum over $k_9$ since it follows from
$k_9|\xi_3$ and $\gcd(\xi_1 \xi_7,\xi_3) = 1$. This ends the proof of lemma \ref{lemma inter b}.

We intend to sum also over $\xi_1$ and we therefore set
$\boldsymbol{\xi}_{b} = (\xi_2, \xi_3, \xi_6, \xi_7, \xi_8) \in \mathbb{Z}_{>0}^5$. We also set,
${\boldsymbol{\xi}_{b}}^{(r_2,r_3,r_6,r_7,r_8)} = \xi_2^{r_2} \xi_3^{r_3} \xi_6^{r_6} \xi_7^{r_7} \xi_8^{r_8}$ for $(r_2,r_3,r_6,r_7,r_8) \in \mathbb{Q}^5$ and finally
\begin{align*}
& Y_4 = \frac{B}{\boldsymbol{\xi}_{b}^{(0,2,1,0,1)}} \textrm{,} &
& Y_4'' = \frac{Y_4}{k_4\ell_4} \textrm{,} \\
& Y_5 = \frac{\boldsymbol{\xi}_{b}^{(-2/3,4/3,2/3,-1/3,1)}}{B^{1/3}} \textrm{,} &
& Y_5'' = \frac{Y_5}{k_5\ell_5} \textrm{,} \\
& Y_1 = \frac{B^{1/3}}{\boldsymbol{\xi}_{b}^{(1/3,1/3,2/3,2/3,0)}} \textrm{.} &
\end{align*}
Recalling the definition \eqref{h^{b}} of the function $h^{b}$, we can sum up the height conditions \eqref{condition1'}, \eqref{condition2'}, \eqref{condition3'}, \eqref{condition4'} and \eqref{condition5'} as
\begin{eqnarray*}
h^{b} \left( \frac{\xi_4''}{Y_4''}, \frac{\xi_5''}{Y_5''}, \frac{\xi_1}{Y_1} \right) & \leq & 1 \textrm{.}
\end{eqnarray*}
Note that, as in the first case, the condition \eqref{condition sup} can be rewritten as
\begin{eqnarray*}
\frac{\xi_1}{Y_1} & \leq & 2^{1/3} \textrm{.}
\end{eqnarray*}
We also introduce the following real-valued functions
\begin{eqnarray*}
g_1^{b} & : & (t_5,t_1,t;\boldsymbol{\xi}_{b}, B) \mapsto
\int_{h^{b}(t_4,t_5,t_1) \leq 1, t < |t_4 t_5 + t_1^2|, |t_4|Y_4 \geq \log(B)^A} \D t_4 \textrm{,} \\
g_2^{b} & : & (t_1,t;\boldsymbol{\xi}_{b}, B) \mapsto \int_{|t_5| Y_5 \geq \log(B)^A} g_1^{b}(t_5,t_1,t;\boldsymbol{\xi}_{b}, B) \D t_5 \textrm{,} \\
g_3^{b} & : & (t;\boldsymbol{\xi}_{b}, B) \mapsto \int_{t_1 > 0} g_2^{b}(t_1,t; \boldsymbol{\xi}_{b}, B) \D t_1 \textrm{,} \\
g_4^{b} & : & t \mapsto \int \int \int_{t_1 > 0, h^{b}(t_4,t_5,t_1) \leq 1, t < |t_4 t_5 + t_1^2|} \D t_4 \D t_5 \D t_1 \textrm{.}
\end{eqnarray*}
The condition $t < |t_4 t_5 + t_1^2|$ corresponds to the condition \eqref{condition6''} which can now be rewritten as
\begin{eqnarray*}
\frac{\xi_8^2}{Y_4 Y_5} & < & \left| \frac{\xi_4''}{Y_4''} \frac{\xi_5''}{Y_5''} + \left( \frac{\xi_1}{Y_1} \right)^2 \right| \textrm{.}
\end{eqnarray*}
We denote by $\kappa_{b}$ the left-hand side of this inequality.

\begin{lemma}
\label{bounds''}
We have the bounds
\begin{eqnarray*}
g_1^{b}(t_5,t_1,t;\boldsymbol{\xi}_{b}, B) & \ll & t_1^{-1}|t_5|^{-1} \textrm{,} \\
g_2^{b}(t_1,t;\boldsymbol{\xi}_{b}, B) & \ll & 1 \textrm{.}
\end{eqnarray*}
\end{lemma}

\begin{proof}
Recall the definition \eqref{h^{b}} of the function $h^{b}$. The first bound is clear since $t_1|t_4t_5|\leq 1$. For the other one, the conditions $|t_4| \leq 1 $ and $|t_4|t_5^2|t_4t_5+t_1^2| \leq 1$ show that $t_4$ runs over a set whose measure is
$\ll \min \left( 1, |t_5|^{-3/2} \right)$. Splitting the integration of this minimum over $t_5$ depending on whether $|t_5|$ is greater or less than $1$ provides the desired bound.
\end{proof}

It is immediate to check that $\boldsymbol{\xi}_{b}$ is restricted to lie in the region
\begin{eqnarray}
\label{Vb}
\mathcal{V}_b & = & \left\{ \boldsymbol{\xi}_{b} \in \mathbb{Z}_{>0}^5, Y_4 \geq \log(B)^A, Y_1 \geq 2^{-1/3} \right\} \textrm{.}
\end{eqnarray}
Assume that $\boldsymbol{\xi}_{b} \in \mathcal{V}_b$ and $\xi_1 \in \mathbb{Z}_{>0}$ are fixed and satisfy the coprimality conditions \eqref{gcd4''}, \eqref{gcd5''}, \eqref{gcd6''} and \eqref{gcd7''}. 

We now turn to the estimation of $C(\boldsymbol{\xi}_{b}',B)$. Let us sum over $\xi_4''$ using the basic estimate $\# \{ n \in \mathbb{Z} , t_1 \leq n \leq t_2 \} = t_2 - t_1 + O(1)$. The change of variable $t_4 \mapsto Y_4'' t_4$ shows that
\begin{eqnarray*}
C(\boldsymbol{\xi}_{b}',B) & = & \sum_{|\xi_5''| \leq Y_1 \xi_1^{-1} Y_4 Y_5'' \log(B)^{-A}} \left( Y_4'' g_1^{b} \left( \frac{\xi_5''}{Y_5''}, \frac{\xi_1}{Y_1}, \kappa_{b}; \boldsymbol{\xi}_{b}, B \right) + O(1) \right) \textrm{.}
\end{eqnarray*}
The overall contribution of the error term is
\begin{eqnarray*}
\sum_{\boldsymbol{\xi}_{b}'} 2^{\omega(\xi_1\xi_2\xi_6\xi_7)} 2^{\omega(\xi_3\xi_8)}
\frac{B \log(B)^{-A}}{\boldsymbol{\xi}_{b}^{(1,1,1,1,1)}\xi_1} & \ll & \log(B)^{12-A} \textrm{.}
\end{eqnarray*}
Let us now sum over $\xi_5''$. Partial summation and the change of variable $t_5 \mapsto Y_5'' t_5$ yield
$$C(\boldsymbol{\xi}_{b}',B) = Y_4'' Y_5'' g_2^{b} \left( \frac{\xi_1}{Y_1}, \kappa_{b}; \boldsymbol{\xi}_{b}, B \right)
+ O \left( Y_4'' \sup_{|t_5| Y_5 \geq \log(B)^A} g_1^{b} \left( t_5, \frac{\xi_1}{Y_1}, \kappa_{b}; \boldsymbol{\xi}_{b}, B \right) \right) \textrm{.}$$
Using the bound of lemma \ref{bounds''} for $g_1^{b}$, we get
\begin{eqnarray*}
\sup_{|t_5| Y_5 \geq \log(B)^A} g_1^{b} \left( t_5, \frac{\xi_1}{Y_1}, \kappa; \boldsymbol{\xi}_{b}, B \right) & \ll & 
Y_5 \log(B)^{-A} \frac{Y_1}{\xi_1} \textrm{.}
\end{eqnarray*}
The overall contribution coming from this error term is therefore
\begin{eqnarray*}
\sum_{\boldsymbol{\xi}_{b}'} 2^{\omega(\xi_1\xi_3\xi_6\xi_7)} 2^{\omega(\xi_3\xi_8)}
\frac{B \log(B)^{-A}}{\boldsymbol{\xi}_{b}^{(1,1,1,1,1)}\xi_1}
& \ll & B \log(B)^{13-A} \textrm{.}
\end{eqnarray*}
Recalling lemma \ref{lemma inter b}, we find that for any fixed $A \geq 9$, we have
\begin{eqnarray*}
N(\boldsymbol{\xi}_{b}',B) & = & \frac1{\xi_8} \theta_{b}(\boldsymbol{\xi}_{b}) \frac{\varphi^{\ast}(\xi_1)}{\varphi^{\ast}(\gcd(\xi_1,\xi_2\xi_6\xi_7))} \frac{\varphi^{\ast}(\xi_1)}{\varphi^{\ast}(\gcd(\xi_1,\xi_3\xi_6\xi_7))} \\
& & g_2^{b} \left(\frac{\xi_1}{Y_1}, \kappa_b; \boldsymbol{\xi}_{b}, B \right) {Y_4Y_5} + R_1(\boldsymbol{\xi}_{b}',B) \textrm{,}
\end{eqnarray*}
where
\begin{eqnarray*}
\theta_{b}(\boldsymbol{\xi}_{b}) & = & \varphi^{\ast}(\xi_2\xi_6\xi_7) \varphi^{\ast}(\xi_3\xi_6\xi_7) \frac{\varphi^{\ast}(\xi_3\xi_8)}{\varphi^{\ast}(\gcd(\xi_3,\xi_6))} \textrm{,}
\end{eqnarray*}
and where $\sum_{\boldsymbol{\xi}_{b}'} R_1(\boldsymbol{\xi}_{b}',B) \ll B \log(B)^4$. For fixed
$\boldsymbol{\xi}_{b} \in \mathcal{V}_b$ satisfying the coprimality conditions \eqref{gcd5''}, \eqref{gcd6''} and \eqref{gcd7''}, let $\mathbf{N}(\boldsymbol{\xi}_{b},B)$ be the sum over $\xi_1$ of the main term of $N(\boldsymbol{\xi}_{b}',B)$, with $\xi_1$ subject to the coprimality condition \eqref{gcd4''}. Let us make use of lemma \ref{arithmetic preliminary} to sum over $\xi_1$. We find that for any fixed $A \geq 9$ and $0 < \sigma \leq 1$, we have
\begin{eqnarray}
\label{Nb}
\ \ \ \ \ \ \ \mathbf{N}(\boldsymbol{\xi}_{b},B) & = & \frac1{\xi_8} \mathcal{P} \Theta_{b}(\boldsymbol{\xi}_{b}) g_3^{b} \left( \kappa_b; \boldsymbol{\xi}_{b}, B \right) Y_4 Y_5 Y_1 \\
\nonumber
& & + O \left( \frac{Y_4Y_5}{\xi_8} \varphi_{\sigma}(\xi_2\xi_3\xi_8) Y_1^{\sigma}
\sup_{t_1 > 0} g_2^{b} \left( t_1, \kappa_b; \boldsymbol{\xi}_{b}, B \right) \right) \textrm{,}
\end{eqnarray}
where
\begin{eqnarray*}
\Theta_{b}(\boldsymbol{\xi}_{b}) & = & \theta_{b}(\boldsymbol{\xi}_{b}) \varphi^{\ast}(\xi_2\xi_3\xi_8) \varphi'(\xi_2\xi_3\xi_6\xi_7\xi_8) \textrm{.}
\end{eqnarray*}
Using the bound of lemma \ref{bounds''} for $g_2^{b}$ and choosing $\sigma = 1/2$, we see that the overall contribution of the error term is
\begin{eqnarray*}
\sum_{\boldsymbol{\xi}_{b}} \varphi_{\sigma}(\xi_2\xi_3\xi_8) \frac{Y_4 Y_5}{\xi_8} Y_1^{1/2} & \ll & \sum_{\xi_2,\xi_3,\xi_6,\xi_8} \varphi_{\sigma}(\xi_2\xi_3\xi_8) \frac{B}{\boldsymbol{\xi}_{b}^{(1,1,1,0,1)}} \\
& \ll & B \log(B)^4 \textrm{,}
\end{eqnarray*}
where we have summed over $\xi_7$ using $Y_1 \geq 2^{-1/3}$. Note that
\begin{eqnarray*}
\frac{Y_4 Y_5 Y_1}{\xi_8} & = & \frac{B}{\boldsymbol{\xi}_{b}^{(1,1,1,1,1)}} \textrm{.}
\end{eqnarray*}
For brevity, we set
\begin{eqnarray*}
D_{h^{b}} & = & \left\{ (t_4,t_5,t_1) \in \mathbb{R}^3, t_1 > 0, h^{b}(t_4,t_5,t_1) \leq 1 \right\} \textrm{.}
\end{eqnarray*}

\begin{lemma}
\label{bounds integrals'}
For $Z_4, Z_5 > 0$, we have
\begin{eqnarray}
\label{2'}
\meas \{ (t_4,t_5,t_1) \in D_{h^{b}}, |t_5| Z_5 \geq 1 \} & \ll & Z_5^{1/2} \textrm{,} \\
\label{4'}
\meas \{ (t_4,t_5,t_1) \in D_{h^{b}}, |t_4| Z_4 < 1 \} & \ll & Z_4^{-1/3} \textrm{,} \\
\label{5'}
\meas \{ (t_4,t_5,t_1) \in D_{h^{b}}, |t_5| Z_5 < 1 \} & \ll & Z_5^{-1} \textrm{.}
\end{eqnarray}
\end{lemma}

\begin{proof}
As in lemma \ref{bounds integrals}, we have $t_1^3 \leq 2$. The condition $|t_4|t_5^2|t_4t_5+t_1^2| \leq 1$ shows that $t_4$ runs over a set whose measure is $\ll |t_5|^{-3/2}$ and integrating this over $t_1 \ll 1$ yields \eqref{2'}. Concerning \eqref{4'}, we split the proof into two cases depending on whether $|t_4t_5 + t_1^2|$ is greater or less than $|t_4|^{1/3}$. If
$|t_4t_5 + t_1^2| \geq |t_4|^{1/3}$, the condition $|t_4|t_5^2|t_4t_5+t_1^2| \leq 1$ gives $|t_4|^{4/3} t_5^2 \leq 1$ and integrating over $t_5$ using this inequality and over $t_1 \ll 1$ gives the result. In the other case, $t_5$ runs over a set whose measure is
$\ll |t_4|^{-2/3}$ and integrating this over $t_1 \ll 1$ also gives the result. Finally, the statement \eqref{5'} is an immediate consequence of $|t_4| \leq 1$ and $t_1 \ll 1$.
\end{proof}

As in section \ref{Summing}, the bound \eqref{2'} shows that if we do not have $Y_5 \geq \log(B)^A$, the contribution of the main term of $\mathbf{N}(\boldsymbol{\xi}_{b},B)$ is $\ll B \log(B)^4$. Thus we can assume from now on that
\begin{eqnarray}
\label{condition Vb1}
Y_5 & \geq & \log(B)^A \textrm{,}
\end{eqnarray}
and since we also have $Y_4 \geq \log(B)^A$, the two bounds \eqref{4'} and \eqref{5'} prove that removing the conditions
$|t_4| Y_4, |t_5| Y_5 \geq \log(B)^A$ from the integral defining $g_3^{b}$ in the main term of $\mathbf{N}(\boldsymbol{\xi}_{b},B)$ in \eqref{Nb} produces an error term whose overall contribution is $\ll B \log(B)^4$. We have thus proved that for any fixed
$A \geq 9$, we have
\begin{eqnarray}
\label{estimate Nb}
\mathbf{N}(\boldsymbol{\xi}_{b},B) & = & \mathcal{P} g_4^{b}(\kappa_b) \frac{B}{\boldsymbol{\xi}_{b}^{(1,1,1,1,1)}}
\Theta_{b}(\boldsymbol{\xi}_{b}) + R_2(\boldsymbol{\xi}_{b},B) \textrm{,}
\end{eqnarray}
where $\sum_{\boldsymbol{\xi}_{b}} R_2(\boldsymbol{\xi}_{b},B) \ll B \log(B)^4$.

\begin{lemma}
For $t > 0$, we have
\begin{eqnarray}
\label{7'}
\meas \{ (t_4,t_5,t_1) \in D_{h^{b}}, |t_4t_5+t_1^2| > t \} & \ll & t^{-3/2} \textrm{,} \\
\label{8'}
\meas \{ (t_4,t_5,t_1) \in D_{h^{b}}, |t_4t_5+t_1^2| \leq t \} & \ll & t^{1/2} \textrm{.}
\end{eqnarray}
\end{lemma}

\begin{proof}
For the bound \eqref{7'}, the conditions $t_1|t_4t_5+t_1^2| \leq 1$, $|t_4|t_5^2|t_4t_5+t_1^2| \leq 1$ and $|t_4t_5+t_1^2| > t$ give $t_1t \leq 1$ and $|t_4|t_5^2t \leq 1$. Integrating over $t_5$ using this inequality then over $|t_4| \leq 1$ and over $t_1$ using $t_1 t \leq 1$ yields \eqref{7'}. For \eqref{8'}, the conditions $t_1^2|t_5||t_4t_5+t_1^2| \leq 1$ and $|t_4t_5+t_1^2| \leq t$ show that $t_5$ runs over a set whose measure is
$\ll \min \left( t_1^{-1} |t_4|^{-1/2}, t |t_4|^{-1} \right) \leq t^{1/2} t_1^{-1/2} |t_4|^{-3/4}$. This concludes the proof since
$t_1, |t_4| \ll 1$.
\end{proof}

The bound \eqref{7'} shows that if $\kappa_b > 1$, the contribution of the main term of $\mathbf{N}(\boldsymbol{\xi}_{b},B)$ is
$\ll B \log(B)^4$, thus we assume from now on that $\kappa_b \leq 1$, namely
\begin{eqnarray}
\label{condition Vb2}
Y_4 Y_5 & \geq & \xi_8^2 \textrm{.}
\end{eqnarray}
Replacing $g_4^{b}(\kappa_b)$ by $g_4^{b}(0)$ in the main term of $\mathbf{N}(\boldsymbol{\xi}_{b},B)$ in \eqref{estimate Nb} therefore creates an error term whose overall contribution is $\ll B \log(B)^4$. Since $g_4^{b}(0) = \tau_{\infty}/3$ by
\eqref{tau 2}, we have obtained the following result.

\begin{lemma}
\label{lemmafin b}
For any fixed $A \geq 9$, we have
\begin{eqnarray*}
\mathbf{N}(\boldsymbol{\xi}_{b},B) & = & \mathcal{P} \frac{\tau_{\infty}}{3} \frac{B}{\boldsymbol{\xi}_{b}^{(1,1,1,1,1)}} \Theta_{b}(\boldsymbol{\xi}_{b}) + R_3(\boldsymbol{\xi}_{b},B) \textrm{,}
\end{eqnarray*}
where $\sum_{\boldsymbol{\xi}_{b}} R_3(\boldsymbol{\xi}_{b},B) \ll B \log(B)^4$.
\end{lemma}

Recall the definition \eqref{Vb} of $\mathcal{V}_b$. It remains to sum the main term of $\mathbf{N}(\boldsymbol{\xi}_{b},B)$ over the $\boldsymbol{\xi}_{b} \in \mathcal{V}_b$ satisfying \eqref{condition Vb1} and \eqref{condition Vb2} and the coprimality conditions \eqref{gcd5''}, \eqref{gcd6''} and \eqref{gcd7''}. It is easy to see that replacing
$\left\{ \boldsymbol{\xi}_{b} \in \mathcal{V}_b, \eqref{condition Vb1}, \eqref{condition Vb2} \right\}$ by the region
\begin{eqnarray*}
\mathcal{V}_b' & = & \left\{ \boldsymbol{\xi}_{b} \in  \mathbb{Z}_{>0}^5, Y_4 \geq 1, Y_5 \geq 1, Y_1 \geq 1, Y_4Y_5 \geq \xi_8^2 \right\}  \textrm{,}
\end{eqnarray*}
produces an error term whose overall contribution is $\ll B \log(B)^4 \log(\log(B))$. Let us redefine $\Theta_{b}$ as being equal to zero if the remaining coprimality conditions \eqref{gcd5''}, \eqref{gcd6''} and \eqref{gcd7''} are not satisfied. Lemma
\ref{lemmafin b} proves that for any fixed $A \geq 9$, we have
\begin{eqnarray*}
N_{b}(A,B) & = & \mathcal{P} \frac{\tau_{\infty}}{3} B \sum_{\boldsymbol{\xi}_{b} \in \mathcal{V}_b'} \frac{\Theta_{b}(\boldsymbol{\xi}_{b})}{\boldsymbol{\xi}_{b}^{(1,1,1,1,1)}} + O \left( B \log(B)^4 \log(\log(B)) \right) \textrm{.}
\end{eqnarray*}
As in section \ref{Conclusion}, $\Theta_{b}$ satifies the assumption \eqref{Psi} of lemma \ref{final sum} and thus we get
$$N_{b}(A,B) = \mathcal{P} \frac{\tau_{\infty}}{3} \alpha_{b} \left( \sum_{\boldsymbol{\xi}_{b} \in \mathbb{Z}_{>0}^5}
\frac{(\Theta_{b} \ast \boldsymbol{\mu})(\boldsymbol{\xi}_{b})}{\boldsymbol{\xi}_{b}^{(1,1,1,1,1)}} \right) B \log(B)^5
+ O \left( B \log(B)^4 \log(\log(B)) \right) \textrm{,}$$
where $\alpha_{b}$ is the volume of the polytope defined in $\mathbb{R}^5$ by $t_2,t_3,t_6,t_7,t_8 \geq 0$ and
\begin{eqnarray*}
2 t_3 + t_6 + t_8 & \leq & 1 \textrm{,} \\
- 2 t_2 + 4 t_3 + 2 t_6 - t_7 + 3 t_8 & \geq & 1 \textrm{,} \\
t_2 + t_3 + 2 t_6 + 2 t_7 & \leq & 1 \textrm{,} \\
2 t_2 + 2 t_3 + t_6 + t_7 + 6 t_8 & \leq & 2 \textrm{.}
\end{eqnarray*}
A computation using \cite{Convex} gives
\begin{eqnarray}
\label{alpha_b}
\alpha_b & = & \frac{929}{2016000} \textrm{,}
\end{eqnarray}
and exactly as in the case of $N_{a}(A,B)$, a calculation provides
\begin{eqnarray*}
\sum_{\boldsymbol{\xi}_{b} \in \mathbb{Z}_{>0}^5}
\frac{(\Theta_{b} \ast \boldsymbol{\mu})(\boldsymbol{\xi}_{b})}{\boldsymbol{\xi}_{b}^{(1,1,1,1,1)}} & = &
\mathcal{P}^{-1} \prod_p \left( 1 - \frac1{p} \right)^6 \tau_p \textrm{.}
\end{eqnarray*}
We have proved the following lemma.

\begin{lemma}
\label{lemma N_b}
For any fixed $A \geq 9$, we have
\begin{eqnarray*}
N_{b}(A,B) & = & \frac1{3} \alpha_b \omega_H(\widetilde{V_2}) B \log(B)^5 + O \left( B \log(B)^4 \log(\log(B)) \right) \textrm{.}
\end{eqnarray*}
\end{lemma}

We now fix $A = 9$ for example. Given the equalities \eqref{alpha_a} and \eqref{alpha_b}, we get
\begin{eqnarray*}
\alpha_a + \alpha_b & = & 3 \alpha(\widetilde{V_2}) \textrm{,}
\end{eqnarray*}
and we immediately complete the proof putting together lemmas \ref{lemma N_a N_b}, \ref{lemma N_a} and \ref{lemma N_b}.

\bibliographystyle{is-alpha}
\bibliography{biblio}

\end{document}